\UseRawInputEncoding
\documentclass[12pt, reqno]{amsart}
\usepackage[margin=1in]{geometry}
\usepackage{amssymb,latexsym,amsmath,amscd,amsfonts}
\usepackage{latexsym}
\usepackage[mathscr]{eucal}
\usepackage{bm}
\usepackage{mathptmx}

\def\textmatrix#1&#2\\#3&#4\\{\bigl({#1 \atop #3}\ {#2 \atop #4}\bigr)}
\def\dispmatrix#1&#2\\#3&#4\\{\left({#1 \atop #3}\ {#2 \atop #4}\right)}
\newcommand{\beg}{\begin{equation}}
	\newcommand{\eeg}{\end{equation}}
\newcommand{\ben}{\begin{eqnarray*}}
	\newcommand{\een}{\end{eqnarray*}}

\newlength{\bibitemsep}
\newlength{\bibparskip}
\let\oldthebibliography\thebibliography
\renewcommand\thebibliography[1]{%
	\oldthebibliography{#1}%
	\setlength{\parskip}{\bibitemsep}%
	\setlength{\itemsep}{\bibparskip}%
}
\newtheorem{thm}{Theorem}[section]

\newtheorem{lem}[thm]{Lemma}

\newtheorem{prop}[thm]{Proposition}
\numberwithin{equation}{section} 
\theoremstyle{definition}
\newtheorem{defn}[thm]{Definition}

\newtheorem{note}[thm]{Note}
\newtheorem{eg}[thm]{Example}

\newcommand{\HS}{\mathcal H}

\newcommand{\A}{\rm{Aut}}

\begin{document}
	\title[Lifting $Q$-commuting and $Q$-intertwining operators]
	{Lifting $Q$-commuting and $Q$-intertwining operators}
	
	\author[Pal and Sahasrabuddhe]{Sourav Pal and Prajakta Sahasrabuddhe}
	
	\address[Sourav Pal]{Mathematics Department, Indian Institute of Technology Bombay,
		Powai, Mumbai - 400076, India.} \email{sourav@math.iitb.ac.in , souravmaths@gmail.com}
	
	\address[Prajakta Sahasrabuddhe]{Mathematics Department, Indian Institute of Technology Bombay,
		Powai, Mumbai - 400076, India.} \email{prajakta@math.iitb.ac.in , praju1093@gmail.com}

	
	
	\thanks{The first named author is supported by the Seed Grant of IIT Bombay, the CPDA and the MATRICS Award (Award No. MTR/2019/001010) of Science and Engineering Research Board (SERB), India. The second named author has been supported by the Ph.D Fellowship of Council of Scientific and Industrial Research (CSIR), India.}

	\begin{abstract}
	   For bounded operators $Q,T_1,T_2$ acting on a Hilbert space $\mathcal{H}$, a pair $(T_1,T_2)$ is said to be $Q$-\textit{commuting} if $T_1T_2=QT_2T_1$ or $T_1T_2=T_2QT_1$ or $T_1T_2=T_2T_1Q$. An operator $B$ on a Hilbert space $\mathcal{K}$ is said to be a \textit{lifting} of $A$ on $\mathcal{H}\subseteq \mathcal{K}$ if $\mathcal H$ is an invariant subspace for $B^*$ and $B^*|_{\mathcal{H}}=A^*$. We obtain the following main results in this article.\\ 
	   
	   \begin{enumerate}
	   	\item The classical commutant lifting theorem states that if $T$ is a contraction on a Hilbert space $\mathcal{H}$ and if $V$ on $\mathcal{K}$ is a minimal isometric dilation of $T$, then for any bounded operator $X \in \mathcal B(\mathcal H)$ satisfying $TX=XT$, there is a norm preserving lifting $Y$ of $X$ on a bigger Hilbert space $\mathcal{K}$ such that $VY=YV$. We generalize this to $Q$-commuting operators. We show that if $T$ is any contraction on $\mathcal{H}$ and $V$ is any isometric lift of $T$ on a space $\mathcal{K} \supseteq \mathcal H$ and if $Q\in \mathcal B(\mathcal H)$ is any bounded operator with $QT$ being a contraction satisfying $XT=QTX$ (or $TQ$ is a contraction with $XT=TQX$) then for any given lifting $\overline{Q}$ on $\mathcal{K}$ of $Q$ satisfying $\|\overline{Q}V\|\leq 1$, there is a norm preserving lifting $Y$ of $X$ such that $YV=\overline{Q}VY$ (or $YV=V\overline{Q}Y$). We provide several proofs to this result and its variants. We also prove that if in particular $Q$ is a contraction such that $XT=TXQ$ then $X$ possesses a lifting $Y$ such that $YV=VY\overline{Q}$ for some lifting $\overline{Q}$ of $Q$. We find generalizations of some of these results for $Q$-intertwining operators.\\
	   	
	   	\item  We find generalized Ando's isometric dilation in the $Q$-commuting setting. Indeed, for a given pair of $Q$-commuting contractions $(T_1,T_2)$, we show different explicit constructions of an isometric lift $\overline{Q}$ of $Q$ and a pair of $\overline{Q}$-commuting isometries $(V_1,V_2)$ dilating $(T_1, T_2)$.\\
	   		   	
	   	\item We find a few new lifting results for intertwining operators and also give new proofs to some existing results in this direction.\\
	   	
	\item It is well known that for $n\geq 3$, a commuting $n$-tuple of contractions not necessarily possesses an isometric dilation. Op$\check{e}$la proved that any $n$-tuple of contractions that commute according to a graph without a cycle dilates to an $n$-tuple of unitaries that commute according to the same graph. We generalize this result to the $Q$-commuting contractions.
	  
	   \end{enumerate}
	     
	      	\end{abstract}
	
	\maketitle
	
	\tableofcontents
	
	\section{Introduction}
	
	\vspace{0.4cm}
	
\noindent Throughout the paper, all operators are bounded linear operators acting on complex Hilbert spaces. For a Hilbert space $\mathcal H$, $\mathcal{B}(\mathcal{H})$ denotes the algebra of operators on $\mathcal{H}$. For two Hilbert spaces $\mathcal H$ and $\mathcal K$, we write $\mathcal H \subseteq \mathcal K$ to mean $\mathcal H$ is a closed linear subspace of $\mathcal K$. A contraction is an operator with norm not greater than one. An operator $B$ on a Hilbert space $\mathcal{K}$ is said to be a \textit{lifting} or \textit{lift} of $A$ on $\mathcal{H}\subseteq \mathcal{K}$ if $\mathcal H$ is an invariant subspace for $B^*$ and $B^*|_{\mathcal{H}}=A^*$.\\
		
	 In this paper, we study dilation and lifting of $Q$-commuting and $Q$-intertwining pairs of operators.
	 
	\begin{defn}
	Suppose $T_1,T_2,X,Q$ are operators acting on a Hilbert space $\mathcal H$.	
		\begin{itemize}
		\item[(a)] We say that the pair $(T_1,T_2)$ is $Q$-\textit{commuting} if $T_1T_2=QT_2T_1$ or $T_1T_2=T_2QT_1$ or $T_1T_2=T_2T_1Q$. More precisely, $(T_1,T_2)$ is called $Q_L$-\textit{commuting} if $T_1T_2=QT_2T_1$, $Q_M$-\textit{commuting} if $T_1T_2=T_2QT_1$ and $Q_R$-\textit{commuting} if $T_1T_2=T_2T_1Q$.
		
		\item[(b)] Similarly, we say that $X$ is $Q$-\textit{intertwining} with the pair $(T_1,T_2)$ if $XT_1=QT_2X$ or $XT_1=T_2QX$ or $XT_1=T_2XQ$. In an analogous manner as in part-(a) one can define $Q_L$-\textit{intertwining}, $Q_M$-\textit{intertwining} and $Q_R$-\textit{intertwining} pair $(T_1,T_2)$ with respect to the operator $X$.		
		\end{itemize}		 
	\end{defn}
The classical commutant lifting theorem due to Sz.-Nagy and Foias \cite{foias:Nagy-2} states that if $T$ is a contraction on $\mathcal{H}$ and if $V$ on $\mathcal{K} \supseteq \mathcal H$ is the minimal isometric dilation of $T$, then any commutant $X$ of $T$ has a norm preserving lifting $Y$ on $\mathcal{K}$ such that $VY=YV$. This result was first proved by Sarason in a special case, see \cite{sarason1}. Indeed, Sarason showed that if the minimal co-isometric extension $V$ of $T$ is such that $V^*$ is a unilateral shift of multiplicity one, then any commutant $X$ of $T$ has a norm preserving extension $Y$ such that $YV =VY$. In \cite{foias:Nagy-1}, Sz.-Nagy and Foias also proved a variant of the commutant lifting theorem by replacing the minimal isometric dilation by the minimal co-isometric extension. Later, Douglas, Muhly and Pearcy \cite{Dou:Muh:Pea} gave a simple and elegant proof to the commutant lifting theorem. Interestingly there are several independent proofs of the commutant lifting theorem in the literature, see \cite{foias1990commutant}. In \cite{Sebestyen1}, Sebestyen took an important next step by producing an anti-commutant counterpart of the commutant lifting theorem and Ando's dilation theorem for commuting contractions. Recently, Keshari and Mallick found generalization of the commutant lifting theorem, Ando's dilation theorem and intertwining lifting theorem in the $q$-commuting setting with $|q|=1$. Also, Keshari, Mallick and Sumesh \cite{K.M., sumesh} studied dilation and lifting for $Q$-commuting contractions.\\

The aims of this article are to provide several proofs to the commutnt lifting theorems in $Q$-commuting and $Q$-intertwining setting and to find some applications to graph theory. The proofs that we present here are analogous to the different proofs (as in \cite{foias1990commutant}) of the commutant lifting theorem. The following are two main results of this article and their different proofs are given in Section \ref{QSection1}.
\begin{thm}
	Let $T$ on a Hilbert space $\mathcal{H}$ be a contraction and let $Q\in\mathcal{B}(\mathcal{H})$ be arbitrary. Suppose $V$ on $\mathcal{K} \supseteq \mathcal H$ is an isometric lift of $T$ and $\overline{Q}\in \mathcal{B}(\mathcal{K})$ is any lift of $Q$.
	\begin{enumerate}
		\item If $XT=QTX$ and if $QT$, $\overline{Q}V$ are contractions, then there is a norm preserving lifting $Y\in \mathcal{B}(\mathcal{K})$ of $X$ such that $YV=\overline{Q}VY$.
		\item If $XT=TQX$ and if $TQ$, $V\overline{Q}$ are contractions, then there is a norm preserving lifting $Y\in \mathcal{B}(\mathcal{K})$ of $X$ such that $YV=V\overline{Q}Y$.
		\item If $TX=XTQ$ and if $Q$ is a contraction and $\overline{Q}$ is an isometry, then there is a norm preserving lifting $Y\in \mathcal{B}(\mathcal{K})$ of $X$ such that $VY=YV\overline{Q}$.
	\end{enumerate} 
\end{thm} 
\begin{thm}
	Let $T$ be any contraction on a Hilbert space $\mathcal{H}$ with $(V,\mathcal{K})$ being an isometric lift of $(T, \mathcal H)$. Also, let $Q\in\mathcal{B}(\mathcal{H})$ be arbitrary.
	\begin{enumerate}
		\item If $TX=QXT$ and if $Q$ is a contraction, then there is an isometric lift $\overline{Q}\in \mathcal{B}(\mathcal{K})$ of $Q$ and a norm preserving lift $Y\in \mathcal{B}(\mathcal{K})$ of $X$ such that $VY=\overline{Q}YV$.
		\item If $XT=TXQ$, then there is a lifting $\overline{Q}\in \mathcal{B}(\mathcal{K})$ of $Q$, a norm preserving lifting $Y\in \mathcal{B}(\mathcal{K})$ of $X$ such that $YV=VY\overline{Q}$.
	\end{enumerate} 
\end{thm}
In Section \ref{QSectioninter}, we find generalizations of some lifting and extension results for $Q$-intertwining operators.
 
	For a pair of commuting contractions $(T_1,T_2)$ acting on a Hilbert space $\mathcal H$, Ando's famous dilation theorem \cite{Ando} tells us that there is a pair of commuting isometries $(V_1,V_2)$ acting on a bigger Hilbert space $\mathcal K \supseteq \mathcal H$ such that $(V_1,V_2)$ dilates $(T_1,T_2)$, that is
	\[
	f(T_1,T_2)=P_{\mathcal H} f(V_1,V_2)|_{\mathcal H}
	\]
	for any polynomial $f\in \mathbb c[z_1,z_2]$, $P_{\mathcal H}$ being the orthogonal projection of $\mathcal K$ onto $\mathcal H$. In \cite{Ando}, Ando showed an explicit construction of such isometries $V_1,V_2$. In Section \ref{QSectionAndo}, we generalize Ando's theorem to $Q$-commuting contractions. Following Ando's method, we present several explicit constructions of such dilation. Indeed, for any contraction $Q$ on $\mathcal{H}$ if $(T_1,T_2)$ is a $Q$-commuting pair of contractions on $\mathcal{H}$, then there are isometric lifts $\overline{Q},V_1,V_2$ of $Q,T_1,T_2$ respectively such that the pair $(V_1,V_2)$ is $\overline{Q}$-commuting. Also, we obtain a variant of this result when $Q$ is any bounded operator when $\|QT_1\|\leq 1$ or $\|T_1Q\|\leq 1$.  
		
		 The well-known intertwining lifting theorem states that given any contractions $T_1,T_2,X$ on $\mathcal{H}$ satisfying $T_1X=XT_2$, $X$ has a norm preserving lift $Y$ that intertwines isometric lifts $V_1,V_2$ of $T_1,T_2$ respectively. We establish a slightly different version of this result. We show that if any isometric lift $(W,\mathcal{K})$ of $X$ is given, then there is a contractive lift $S_1$ of $T_1$ and a norm preserving lift $S_2$ of $T_2$ such that $S_1W=WS_2$. As an application of this, we obtain in Section \ref{Sectioninter} a new proof of Ando-type dilation for intertwining contractions. In addition to this, we give alternative proofs to a few existing intertwining lifting theorems.
		
		In general for $n\geq 3$, a commuting $n$-tuple of contractions does not possess an isometric dilation. Op$\check{e}$la \cite{Ope} proved that any $n$-tuple of contractions that commute according to a graph without a cycle dilates to a $n$-tuple of unitaries that commute according to the same graph. In \cite{K.M.}, Keshari and Mallick obtained a few generalized results in this direction. Let $G=(V,E)$ be any simple graph on $n$ vertices and let $\textbf{Q}$ denote a function that identifies an edge of the graph with some unitary on $\mathcal{H}$. In Section \ref{QSectiongraph}, we extend Op$\check{e}$la's notion of `commutativity with respect to a graph' and define $(G,\textbf{Q},L)$-commuting, $(G,\textbf{Q},M)$-commuting and $(G,\textbf{Q},R)$-commuting system of contractions. We prove that if $G$ is a simple connected graph on $n$ vertices without any cycles and if $(T_1,T_2,\dots ,T_n)$ is a $(G,\textbf{Q},L)$-commuting, $(G,\textbf{Q},M)$-commuting or $(G,\textbf{Q},R)$-commuting system of contractions, then there is a $(G,\overline{\textbf{Q}},L)$-commuting, $(G,\overline{\textbf{Q}},M)$-commuting or $(G,\overline{\textbf{Q}},R)$-commuting system of isometries, $(V_1,V_2,\dots ,V_n)$ on some Hilbert space $\mathcal{K}$ containing $\mathcal{H}$ such that each $V_i$ is a lift of $T_i$, where $\overline{\textbf{Q}}$ is a function that maps an edge $(i,j)$ to the unitary $\textbf{Q}(i,j)\oplus I $ on $\mathcal{K}=\mathcal{H}\oplus (\mathcal{K}\ominus \mathcal H)$. In the same Section, we present a few more results which connect our lifting results with graph theory.
		
We present a few examples of $Q$-commuting operators along with a few relevant associated results in Section \ref{Sec:Examples}. Section \ref{Sec:background} deals with some background material and necessary preparatory results.
	
	\vspace{0.2cm}  
	
	\section{Preparatory results and background material} \label{Sec:background}
	
	\vspace{0.4cm}
	
\noindent In this Section, we recall a few theorems from the literature which will be used in sequel. We start with a dual version of Parrot's theorem proved by Sebestyen.

	\begin{thm}[\cite{Sebestyen}, Theorem 1]\label{dualparrot}
		Let ${K}$ and ${K}'$ be Hilbert spaces, $H \subseteq K$ and $H' \subseteq K'$ be subspaces, and $X : H \rightarrow K'$ and $X': H' \rightarrow K$ be given bounded linear transformations. Then there is a map $Y:K \rightarrow K'$ extending $X$ so that $Y^*$ extends $X'$ if and only if the following identity holds true:
		\[
		\langle Xh,h' \rangle = \langle h, X'h' \rangle \quad \text{for all }h \in H \text{ and } h' \in H'.
		\]
		Moreover $Y$ can be of norm max$\{\|X\|, \|X'\|\}$ possible at most.
	\end{thm}
	
	The following result from \cite{Dou:Muh:Pea}, known as Douglas's lemma is omnipresent throughout the operator theory results.
	
	\begin{lem} [\cite{Dou:Muh:Pea}, Lemma 2.1] \label{Dlemma}
		Suppose that $\mathcal{S}$, $\mathcal{H}$, $\mathcal{K}$ are Hilbert spaces and $A\in \mathcal{B}(\mathcal{S}, \mathcal{K})$ and $B\in \mathcal{B}(\mathcal{H}, \mathcal{K})$. Then there is a contraction $Z\in \mathcal{B}(\mathcal{S},\mathcal{H})$ satisfying $A=BZ$ if and only if $AA^*\leq BB^*$.
	\end{lem}
	Now we recall a generalization of the above lemma.	
	
	\begin{thm}[\cite{Dou:Muh:Pea}, Theorem 1]\label{Dthm1}
		Let $\mathcal{H}_0, \mathcal{H}_1,\mathcal{H}_2$ and $\mathcal{K}$ be Hilbert spaces, and for $0\leq i \leq 2$ let $A_i$ be an operator mapping $\mathcal{H}_i$ into $\mathcal{K}$. Then there are operators $Z_1$ and $Z_2$ that map $\mathcal{H}_0$ into $\mathcal{H}_1$ and $\mathcal{H}_2$, respectively, and that satisfy the two conditions 
		\begin{enumerate}
			\item  $A_1Z_1+A_2Z_2=A_0,$
			\item  $Z_1^*Z_1+Z_2^*Z_2\leq I_{\mathcal{H}_0}$
		\end{enumerate} 
		if and only if 
		$ A_1A_1^*+A_2A_2^*\geq A_0A_0^*. $
	\end{thm}
	Using Lemma \ref{Dlemma}, Douglas, Muhly and Pearcy gave an Alternative proof of the following proposition due to Sz.-Nagy and Foias (\cite{foias-2}).   
	\begin{prop}[\cite{Dou:Muh:Pea}, Proposition 2.2]\label{prop:21}
		For $i=1,2$, let $T_i$ be a contraction on a Hilbert space $\mathcal{H}_i$, and let $X$ be an operator mapping $\mathcal{H}_2 $ into $\mathcal{H}_1$. A necessary and sufficient condition that the operator on $\mathcal{H}_1 \oplus \mathcal{H}_2$ defined by the matrix 
		$
		\begin{pmatrix}
			T_1 & X\\
			0 & T_2
		\end{pmatrix}
		$ be a contraction is that there is a contraction $C$ mapping $\mathcal{H}_2 $ into $\mathcal{H}_1$ such that 
		\[
		X = D_{T_1^*}CD_{T_2}.
		\]
	\end{prop}
	The following theorem will be used quite often in the last Chapter of the thesis.
	\begin{thm}[\cite{foias1990commutant}, Chapter VII, Corollary 1.3]\label{contractive lift}
		Let $T\in \mathcal{B}(\mathcal{H})$, $T'\in \mathcal{B}(\mathcal{H}')$ and $A\in \mathcal{B}(\mathcal{H},\mathcal{H}')$ be any contractions such that $T'A=AT$. Let $V\in \mathcal{B}(\mathcal{K})$ be an isometric lift of $T$ and $W\in \mathcal{B}$ be a contractive lift of $T'$. Then there is a contractive lifting $B\in \mathcal{B}(\mathcal{K},\mathcal{K}')$ of $A$ such that $WB=BV $. 
	\end{thm}
	
	Sarason in  (\cite{sarason}), showed that if minimal co-isometric extension $V$ of $T$ is such that $V^*$ is a unilateral shift of multiplicity one, then any commutator $X$ of $T$ has a norm preserving extension $ Y $ such that $YV=VY$. Sz.-Nagy and Foias generalized the seminal work of Sarason for an arbitrary contraction $T$ which is known as the classical commutant lifting theorem.
	\begin{thm} \label{Classic:Comm}
		Let $T$ be a contraction on a Hilbert space $\HS$ and let $(V, \mathcal K)$ be the minimal isometric (or minimal unitary) dilation of $T$. If $R$ commutes with $T$, then there is an operator $S$ commuting with $V$ such that $\|R\|=\|S\|$ and $RT^n=P_{\HS}SV^n|_{\HS}$ for all $n \geq 0$.
	\end{thm}
	Chapter VII of \cite{foias1990commutant} contains four different proofs of the commutant lifting theorem due to Foias and Frazho. Further, Sz.-Nagy and Foias \cite{Nagy} proved a variant of this theorem replacing the minimal isometric dilation by the minimal co-isometric extension of $T$. A simple but elegant proof of this variant of commutant lifting theorem was given by R. G. Douglas, P. S. Muhly, and Carl Pearcy in 1968 (\cite{Dou:Muh:Pea}). 
	\begin{thm}[\cite{Dou:Muh:Pea}, Theorem 4]
		Let $T$ be a contraction on Hilbert space $\mathcal{H}$, and $Z$ be the unique minimal co-isometric extension of $T$, acting on Hilbert space $\mathcal{K}$ containing $\mathcal{H}$. For every operator $X\in \mathcal{B}(\mathcal{H})$ that commutes with $T$, there is an operator $Y\in \mathcal{B}(\mathcal{K})$ such that 
		\[ ZY=YZ,\; Y(\mathcal{H})\subseteq\mathcal{H},\; Y|_{\mathcal{H}}=X, \;\|Y\|=\|X\|.\]
	\end{thm} 
	We will make numerous applications of the following intertwining lifting theorem due to Sz.-Nagy and Foias. 
	\begin{thm}[Sz.-Nagy \& Foias, \cite{Nagy}]\label{intertwining}
		Suppose that for $i=1,2$, $T_i$ is a contraction acting on a Hilbert space $\mathcal{H}_i$ and $V_i$ is the unique minimal isometric lift of $T_i$ acting on the Hilbert space $\mathcal{K}_i\supseteq \mathcal{H}_i$. Let $X$ be an operator that maps $\mathcal{H}_2$ into $\mathcal{H}_1$ and satisfies the equation $T_1X=XT_2$. Then there is an operator $Y$ mapping $\mathcal{K}_2$ into $\mathcal{K}_1$ such that 
		\[
		V_1Y = YV_2, \quad Y^*\mathcal{H}_1\subseteq \mathcal{H}_2, \quad Y^*|_{\mathcal{H}_1}=X^*, \quad \|Y\|=\|X\|.	\] 
	\end{thm}
	Similarly, the intertwining extension theorem is as follows.
	\begin{thm}[Sz.-Nagy \& Foias, \cite{Nagy}]\label{douglasintertwining}
		Suppose that for $i=1,2$, $T_i$ is a contraction acting on a Hilbert space $\mathcal{H}_i$ and $Z_i$ is the unique minimal co-isometric extension of $T_i$ acting on the Hilbert space $\mathcal{K}_i\supseteq \mathcal{H}_i$. Let $X$ be an operator that maps $\mathcal{H}_2$ into $\mathcal{H}_1$ and satisfies the equation $T_1X=XT_2$. Then there is an operator $Y$ mapping $\mathcal{K}_2$ into $\mathcal{K}_1$ such that 
		\[
		Z_1Y = YZ_2, \quad Y\mathcal{H}_2\subseteq \mathcal{H}_1, \quad Y|_{\mathcal{H}_2}=X, \quad \|Y\|=\|X\|.	\] 
	\end{thm}
	   Let $S_1, S_2$ be contractions acting on Hilbert spaces $\mathcal{H}_1$ and $\mathcal{H}_2$ respectively. Let $R_1$ and $R_2$ be the unique minimal co-isometric extensions of $S_1$ and $S_2$ acting on the Hilbert spaces $\mathcal{K}_1$ and $\mathcal{K}_2$ respectively. Let $P_n$ and $P'_n$ be the orthogonal projections of the spaces $\mathcal{K}_1$ and $\mathcal{K}_2$ to the corresponding subspaces 
	   $ \bigvee\limits_{k=0}^n R_1^{*k}(\mathcal{H}_1)$ and $\bigvee\limits_{k=0}^n R_2^{*k}(\mathcal{H}_2)$ respectively, for $n = 0, 1, 2,\dots$ . The following theorem is a generalization of the commutant lifting theorem due to Sz. Nagy and Foias.
	   
	   \begin{thm}[\cite{Sebestyen1}, Theorem 3]\label{sebintlift}
	   	Let $S_1,S_2$ and $R_1,R_2$ be as above and let $Q : \mathcal{K}_1 \to \mathcal{K}_1$ be any contraction that commutes with all of the projections $\{P_n\}_{n=0}^{\infty}$. Let $X : \mathcal{H}_1\to\mathcal{H}_2$ be an operator satisfying 
	   	\[S_2X=XS_1QP_0.\] 
	   	Then there is an operator $Y:\mathcal{K}_1\to \mathcal{K}_2$ which extends $X$ has the same norm as $X$ and intertwines $R_2$ and $R_1Q$: 
	   	\[ R_2Y=YR_1Q. \]	
	   \end{thm}
Mallick and Sumesh have established a more general operator theoretic version of the same results in \cite{sumesh}. For $Q\in \mathcal{B}(\mathcal{H})$ and for any Hilbert space $\mathcal{K}$ containing $\mathcal{H}$, let $\overline{Q}\in \mathcal{B}(\mathcal{K})$ be such that $\overline{Q} =Q \oplus Q'$ with respect to the decomposition $\mathcal{K}=\mathcal{H}\oplus \mathcal{H}^{\perp}$. 
\begin{thm}[\cite{sumesh}, Theorem 2.6]\label{Qcomliftsumesh}
	Let $T\in \mathcal{B}(\mathcal{H})$ be contraction with isometric lifting $V$ on $\mathcal{K}$ and let $X\in \mathcal{B}(\mathcal{H})$. Suppose $Q\in \mathcal{B}(\mathcal{H})$, $\overline{Q}\in \mathcal{B}(\mathcal{K})$ be contractions.  
	\begin{itemize}
		\item[(1)] If $XT=QTX$, then there is a lifting $Y\in \mathcal{B}(\mathcal{K})$ of $X$ such that $YV = \overline{Q}VY$.
		\item[(2)] If $XT=TQX$, then there is a lifting $Y\in \mathcal{B}(\mathcal{K})$ of $X$ such that $YV = V\overline{Q}Y$.
	\end{itemize}
	Further assume that $Q$, $\overline{Q}$ are unitary.
	\begin{itemize}	
		\item[(3)] If $XT=TXQ$, then there is a lifting $Y\in \mathcal{B}(\mathcal{K})$ of $X$ such that $YV = VY\overline{Q}$.
	\end{itemize}
	In all cases $\mathcal{H}$ is invariant under $S$,
	$S|_{\mathcal{H}}=T_2$. Further $T^nX^m=P_{\mathcal{H}}V^nY^m|_{\mathcal{H}}$ and $X^nT^m=P_{\mathcal{H}}Y^nV^m|_{\mathcal{H}}$ for all $n,m\geq 0$. Moreover $Y$ can be chosen such that $\|Y\|=\|X\|$.
\end{thm}
\begin{thm}[\cite{sumesh}, Theorem 2.10]\label{Qcomextsumesh}
	Let $T\in \mathcal{B}(\mathcal{H})$ be contraction with co-isometric extension $Z$ on $\mathcal{K}$ and let $X\in \mathcal{B}(\mathcal{H})$. Suppose $Q\in \mathcal{B}(\mathcal{H})$, $\overline{Q}\in \mathcal{B}(\mathcal{K})$ be contractions.  
	\begin{itemize}
		\item[(1)] If $TX=XTQ$, then there is an extension $Y\in \mathcal{B}(\mathcal{K})$ of $X$ such that $ZY=YZ\overline{Q}$.
		\item[(2)] If $TX=XQT$, then there is an extension $Y\in \mathcal{B}(\mathcal{K})$ of $X$ such that $ZY=Y\overline{Q}Z$.
	\end{itemize}
	Further assume that $Q$, $\overline{Q}$ are unitary.
	\begin{itemize}	
		\item[(3)] If $TX=QXT$, then there is an extension $Y\in \mathcal{B}(\mathcal{K})$ of $X$ such that $ZY=\overline{Q}YZ $.
	\end{itemize}
	In all cases $\mathcal{H}$ is invariant under $S$,
	$S|_{\mathcal{H}}=T_2$. Further $T^nX^m=P_{\mathcal{H}}V^nY^m|_{\mathcal{H}}$ and $X^nT^m=P_{\mathcal{H}}Y^nV^m|_{\mathcal{H}}$ for all $n,m\geq 0$. Moreover $Y$ can be chosen such that $\|Y\|=\|X\|$.
\end{thm}

\begin{thm}[\cite{sumesh}, Theorem 2.16]\label{dil-1}
	Let $T_1,T_2\in \mathcal{B}(\mathcal{H})$ be contractions and $Q\in\mathcal{B}(\mathcal{H})$ be a unitary such that $T_2T_1=QT_1T_2$. Then there exists a Hilbert space $\mathcal{K}$ containing $\mathcal{H}$, isometries $V_1,V_2$ and $\overline{Q}\in \mathcal{B}(\mathcal{K})$ unitary such that 
	\begin{enumerate}
		\item[(i)] $V_2V_1=\overline{Q}V_1V_2$; and
		\item[(ii)] $V_i$ is a lifting (and hence dilation) of $T_i$ so that $T_1^nT_2^m=P_{\mathcal{H}}V_1^nV_2^m|_{\mathcal{H}}$ and $T_2^nT_1^m=P_{\mathcal{H}}V_2^nV_1^m|_{\mathcal{H}}$ for all $n,m\geq 0$.
	\end{enumerate}
\end{thm}	   

\begin{thm}[\cite{sumesh}, Corollary 2.17]\label{dil-2}
	Let $T_1,T_2\in \mathcal{B}(\mathcal{H})$ be contractions and $Q\in\mathcal{B}(\mathcal{H})$ be a unitary such that $T_2T_1=T_1T_2Q$. Then there exists a Hilbert space $\mathcal{K}$ containing $\mathcal{H}$, co-isometries $Z_1,Z_2$ and $\overline{Q}\in \mathcal{B}(\mathcal{K})$ unitary such that 
	\begin{enumerate}
		\item[(i)] $Z_2Z_1=Z_1Z_2\overline{Q}$; and
		\item[(ii)] $Z_i$ is an extension (and hence dilation) of $T_i$ so that $T_1^nT_2^m=P_{\mathcal{H}}Z_1^nZ_2^m|_{\mathcal{H}}$ and $T_2^nT_1^m=P_{\mathcal{H}}Z_2^nZ_1^m|_{\mathcal{H}}$ for all $n,m\geq 0$.
	\end{enumerate}
\end{thm}

\vspace{0.2cm}
	   
\section{Examples} \label{Sec:Examples}

\vspace{0.4cm}

\noindent In this Section, we construct a few examples of $Q$-commuting operators with $Q$ being in three different positions (on the left, in the middle and on the right respectively). Also, we provide a couple of results concerning the positions of $Q$ while considering $Q$-commuting operators.
  
	   \begin{eg}
	   	Let $a_1,a_2,b_1,b_2,b_3$ be any non zero complex numbers and let $T_1=\begin{pmatrix}
	   		0&a_1\\
	   		0&a_2
	   	\end{pmatrix}$, $T_2= \begin{pmatrix}
	   	b_1&b_2\\
	   	0&b_3
   	\end{pmatrix}$ on $\mathbb{C}^2$. Clearly,
   \[  T_1T_2=\begin{pmatrix}
   	0&a_1b_3\\
   	0&a_2b_3
   \end{pmatrix} \; \& \; T_2T_1=\begin{pmatrix}
   0&b_1a_1+b_2a_2\\
   0&b_3a_2
\end{pmatrix}. \] Then one can easily observe that for any scalar $q$, one need not have $T_1T_2=qT_2T_1$ but for
\[Q=\begin{pmatrix}
	0&a_1/a_2\\
	0&1
\end{pmatrix}, \; T_1T_2=QT_2T_1.  \]
Again simple calculation shows that for $Q_1=\begin{pmatrix}
	(a_1b_3-a_2b_2)/a_1b_1&0\\
	0&1
\end{pmatrix}$, $T_1T_2=T_2Q_1T_1$. 

Now suppose there is a bounded operator $Q'=\begin{pmatrix}
	x_1&x_2\\
	x_3&x_4
\end{pmatrix}$ such that $T_1T_2=T_2T_1Q'$ then 
\[ \begin{pmatrix}
	0&a_1b_3\\
	0&a_2b_3
\end{pmatrix}= \begin{pmatrix}
(b_1a_1+b_2a_2)x_3&(b_1a_1+b_2a_2)x_4\\
b_3a_2x_3&b_3a_2x_4
\end{pmatrix}.\] Thus as $b_3a_2\neq  0$, $x_3=0$ and $x_4=1$. But this forces the equation $a_1b_3=b_1a_1+b_2a_2$ in $(1,2)$ position. Hence, there exists some bounded operator $Q'$ such that $T_1T_2=T_2T_1Q$ if and only if $a_1b_1+b_2a_2=a_1b_3$. 
	   \end{eg}
  
As observed in Example 2.4 in \cite{K.M.}, given two bounded operators $T_1,T_2$, it is not necessary that there is a $Q$ such that $(T_1,T_2)$ are $Q$-commuting or $(T_2,T_1)$-are $Q$-commuting. Here we provide one more such example on infinite dimensional Hilbert space.     
   
\begin{eg}
	Let $\{ e_1,e_2,\dots \}$ be the standard orthonormal basis of $l^2$. Let $T_1, T_2$ on $l^2$ be such that for even $j\in \mathbb{N}$, 
	\[ T_1e_j=\dfrac{1}{3}(e_j+e_{j+1})\; , \; T_2e_j=0 ,\] and for odd $j$,
	\[ T_1e_j=0, T_2e_j=\dfrac{1}{2}e_{j+1} .\] So clearly $T_1,T_2\in \mathcal{B}(\mathcal{H})$ are contractions. Then we prove that there is no bounded operator $Q$ on $l^2 $ for which $(T_1,T_2)$ or $(T_2,T_1)$ are $Q_L, Q_M$ or $Q_R$ commuting. (Recall that $(A,B)$ will be $Q_L$ or $Q_M$ or $Q_R$ commuting if $AB=QBA$ or $AB=BQA$ or $AB=BAQ$ respectively.) Let $Q\in \mathcal{B}(\mathcal{H})$ be arbitrary. Since $T_1e_1=0,$ we obtain $QT_2T_1e_1=T_2QT_1e_1=0$. But $T_1T_2e_1=T_1(\dfrac{1}{2}e_2)=\dfrac{1}{6}(e_2+e_3)\neq 0$. Therefore, $T_1T_2\neq QT_2T_1$ and $T_1T_2\neq T_2QT_1$. Again $T_2e_2=0$ implies that $QT_1T_2e_2=T_1QT_2e_2=0$ but $T_2T_1e_2=T_2(\dfrac{1}{3}(e_2+e_3))=\dfrac{1}{6}e_4\neq  0$.  Hence, $T_2T_1\neq QT_1T_2$ and $T_2T_1\neq T_1QT_2$.
	
	Now to prove that $T_1T_2\neq T_2T_1Q$, observe that $T_1T_2e_1=\dfrac{1}{6}(e_2+e_3) $. Now if $Qe_1=0$ then the result is obvious. So let $Qe_1\neq 0$. Let $Qe_1=\sum_{j=1}^{\infty} q_{j1}e_j$. Then 
	\begin{align*}
		T_2T_1Qe_1&=T_2(\sum_{j=1}^{\infty} q_{(2j)1}(\dfrac{1}{3}(e_{2j}+e_{2j+1})))\\
		&=\dfrac{1}{6}\sum_{j=1}^{\infty} q_{(2j)1}e_{2j+2}\in \overline{Span}\{e_2,e_4,\dots \}
	\end{align*}
Hence now due to orthonormality of $e_i$, it is clear that $T_1T_2e_1\neq T_2T_1Qe_1 $. Hence $T_1T_2\neq T_2T_1Q$. Now, $T_2T_1e_2=\dfrac{1}{6}e_4$ and for $Qe_2=\sum_{j=1}^{\infty} q_{j2}e_j$,	
\begin{align*}
	T_1T_2Qe_2&=T_1(\sum_{j=1}^{\infty} q_{(2j-1)2}(\dfrac{1}{2}e_{2j}))\\
	&=\dfrac{1}{6}\sum_{j=1}^{\infty} q_{(2j-1)2}(e_{2j}+e_{2j+1})
\end{align*}
  Hence, again by orthonormality of $\{e_j:j\in \mathbb{N}\}$, for $T_2T_1e_2=T_1T_2Qe_2$, $q_{22}=0$ being a coefficient of $e_{5}$ but comparing the coefficients of $e_{4}$, $q_{22}=1$ which is not possible. Hence, $T_2T_1\neq T_1T_2Q$. 
  This proves that neither $(T_1,T_2)$ nor $(T_2,T_1)$ are $Q$-commuting for any bounded operator $Q$.
\end{eg}
The following example shows that $(T_1,T_2)$ is $Q_L$-commuting but not $Q_R$ or $Q_M$-commuting, whereas $(T_2,T_1)$ is $Q_R$, $Q_M$-commuting but not $Q_L$-commuting. 
\begin{eg}
	Let $\{ e_1,e_2,\dots \}$ be the standard orthonormal basis of $l^2$ and let $e_0=0$. Let $T_1, T_2, Q, Q'$ on $l^2=\mathcal{H}$ be such that for even $j\in \mathbb{N}$, 
	\[ T_1e_j=0\; , \; T_2e_j=\dfrac{1}{2}e_{j+2}\; , \; Qe_j=e_{j-2}\; , \; Q'e_j=e_{j+2}\] and for odd $j$,
	\[ T_1e_j=e_{j+1}\; , \; T_2e_j=\dfrac{1}{2}(e_j+e_{j+1})\; , \; Qe_j=e_{j+2} \; , \; Q'e_j=e_{j+2}.\] So clearly $T_1,T_2\in \mathcal{B}(\mathcal{H})$ are contractions. Then for odd $j$, $T_1T_2e_j=\dfrac{1}{2}e_{j+1}, $ $T_2T_1e_j=\dfrac{1}{2}e_{j+3}, $ hence $QT_2T_1e_j=\dfrac{1}{2}e_{j+1}=T_1T_2e_j$. We also have,
		 \[T_1T_2Qe_j=T_1T_2e_{j+2}=T_1(\dfrac{1}{2}(e_{j+2}+e_{j+3}=\dfrac{1}{2}e_{j+3})=T_2(e_{j+1})=T_2T_1e_j\] for $j$ odd. Again for even $j$, $T_1T_2e_j=0=T_2T_1e_j$. Hence, $T_1T_2=QT_2T_1$ and $T_2T_1=T_1T_2Q$.
		Observe that for odd $j$,
		\[ T_1Q'T_2e_j=T_1Q'(\dfrac{1}{2}(e_j+e_{j+1}))=T_1\dfrac{1}{2}(e_{j+2}+e_{j+3})=\dfrac{1}{2}e_{j+3}=T_2T_1e_j \] and for even $j$,
		\[  T_1Q'T_2e_j=T_1Q'(\dfrac{1}{2}e_{j+2})=T_1\dfrac{1}{2}e_{j+4}=0=T_2T_1e_j.\]
		Therefore, $T_2T_1=T_1Q'T_2$. Now we prove that for any $\widetilde{Q}\in \mathcal{B}(\mathcal{H}), $ $T_1T_2\neq T_2T_1\widetilde{Q}$, $T_1T_2\neq T_2\widetilde{Q}T_1 $ and $T_2T_1\neq \widetilde{Q}T_1T_2 $. Note that $T_1T_2e_1=\dfrac{1}{2}e_2$ but for any $i$ if $\widetilde{Q}e_i=\sum_{j=1}^{\infty} q_{ji}e_j$ then, \[T_2T_1\widetilde{Q}e_1=T_2(\sum_{j=1}^{\infty} q_{(2j-1)1}e_{2j})=\dfrac{1}{3}\sum_{j=1}^{\infty} q_{(2j-1)1}e_{2j+2}\in \overline{Span}\{e_4,e_6,\dots\}.\] 
		Hence, $T_1T_2e_1\neq T_2T_1\widetilde{Q}e_1.$ Let $q_{0i}=0$ for all $i$. We can show that \[T_2\widetilde{Q}T_1e_1=\sum_{j=1}^{\infty}(q_{(2j-1)2}e_{2j-1}+(q_{(2j-1)2}+q_{(2j-2)2})e_{2j}).\] So clearly for $ T_1T_2e_1=T_2\widetilde{Q}T_1e_1$ by comparing the coefficient of $e_{2j-1}$, we must have $q_{(2j-1)i}= 0 $ and hence by comparing the coefficient of $e_{2j}$ for $ j\geq 2$, we must have $q_{(2j-2)}=0$. But since $q_{02}=0$, the coefficient of $e_2$ in $T_2\widetilde{Q}T_1e_1$ which is $q_{12}=0$ now must be equal to that in $T_1T_2e_1$, which is $ \dfrac{1}{2}$. This is clearly not possible. Hence the result.    
\end{eg}
Next example shows that $(T_2,T_1)$ are $Q_L,Q_{1M}$ and $Q_{2R}$-commuting with respect to some $Q$, $Q_1$, $Q_2$ but $(T_1,T_2)$ are not $Q$-commuting for any $Q\in \mathcal{B}(\mathcal{H})$.
\begin{eg}
	Consider $\mathcal{H}=l^2$ with $B=\{e_1,e_2,\dots\}$ being the standard orthonormal basis and let $a,b$ be any non zero complex numbers. Now define, 
	\[T_1(x_1,x_2,\dots )=(ax_2,ax_3,\dots ),\; T_2(x_1,x_2,\dots )=(0,bx_3,bx_4,\dots). \] Then, clearly
	$T_1T_2(x_1,x_2,\dots)= (abx_3,abx_4,\dots )$ and $T_2T_1(x_1,x_2,\dots )=(0,abx_4,abx_5,\dots)$. Then clearly for $Q(x_1,x_2,x_3,\dots )=(0,x_2,x_3,\dots)$, $T_2T_1=QT_1T_2$. Further if we define 
	\[ Q_1(x_1,x_2,x_3,\dots)=(x_1,0,x_3,x_4,\dots)\;,\;Q_2(x_1,x_2,x_3,\dots)=(x_1,x_2,0,x_4,\dots) \] then we obtain 
	\[ T_1Q_1T_2(x_1,x_2,\dots)=T_1(0,0,bx_4,bx_5,\dots)=(0,abx_4,abx_5,\dots)=T_2T_1(x_1,x_2,\dots ) \]  and 
	\[  T_1T_2Q_2(x_1,x_2,\dots )=T_1(0,0,bx_4,bx_5,\dots)=(0,abx_4,abx_5,\dots)=T_2T_1(x_1,x_2,\dots ) .\]
	Hence, $T_2T_1=QT_1T_2=T_1Q_1T_2=T_1T_2Q_2$. 
	So, $(T_2,T_1)$ are $Q_L,Q_M$ and $Q_R$-commuting. We will now prove that there is no $\widetilde{Q}$ for which $(T_1,T_2)$ are $\widetilde{Q}$ commuting. So let $\widetilde{Q}\in \mathcal{B}(\mathcal{H})$ be arbitrary, then $T_1T_2\neq \widetilde{Q}T_2T_1,$ $T_1T_2\neq T_2\widetilde{Q}T_1$ and $T_1T_2\neq T_2T_1\widetilde{Q}$.
	This can be proved as below. Observe that $T_1T_2e_3=abe_1$ but $T_2T_1e_3=0$ and hence $\widetilde{Q}T_2T_1e_3=0$, so $T_1T_2\neq \widetilde{Q}T_2T_1$. Note that $\langle T_2T_1\widetilde{Q}e_3, e_1\rangle=0$ whereas $\langle T_1T_2e_3, e_1\rangle=ab\neq 0$. Hence clearly, $T_1T_2\neq T_2T_1\widetilde{Q}$. Again let $y=\widetilde{Q}T_1e_3$, then $\langle T_2\widetilde{Q}T_1e_3, e_1\rangle=\langle T_2y, e_1\rangle=0$ whereas $\langle T_1T_2e_3, e_1\rangle=ab\neq 0$. Hence we have, $T_1T_2\neq T_2\widetilde{Q}T_1$.
\end{eg}
As an application of Douglas, Muhly and Pearcy's lemma, Lemma  \ref{Dlemma}, we note the following as an observation.
\begin{thm}
	Let $T_1,T_2\in \mathcal{B}(\mathcal{H})$ be arbitrary. Then the following hold.
	\begin{enumerate}
		\item There is a bounded operator $Q\in \mathcal{B}(\mathcal{H})$ such that $T_1T_2=QT_2T_1$ if and only if $Ran T_2^*T_1^*\subseteq Ran T_1^*T_2^*$. 
		\item There is a bounded operator $Q\in \mathcal{B}(\mathcal{H})$ such that $T_1T_2=T_2T_1Q$ if and only if $Ran T_1T_2\subseteq Ran T_2T_1$.
	\end{enumerate}
\end{thm}
\begin{proof}
	Follows trivially from Lemma \ref{Dlemma}.
\end{proof}
Note that if $T_1T_2$ and $T_2T_1$ have closed ranges then $Ran T_2^*T_1^*\subseteq Ran T_1^*T_2^*$ is equivalent to $Ker T_2T_1\subseteq Ker T_1T_2$ and $Ran T_1T_2\subseteq Ran T_2T_1$ is equivalent to $Ker T_1^*T_2^*\subseteq Ker T_2^*T_1^*$. Hence in particular we have the following result.
\begin{thm}
    	Let $\mathcal{H}$ be any finite dimensional Hilbert space and $T_1,T_2\in \mathcal{B}(\mathcal{H})$ be arbitrary. Then the following hold.
    \begin{enumerate}
    	\item There is a bounded operator $Q\in \mathcal{B}(\mathcal{H})$ such that $T_1T_2=QT_2T_1$ if and only if $Ker T_2T_1\subseteq Ker T_1T_2$. 
    	\item There is a bounded operator $Q\in \mathcal{B}(\mathcal{H})$ such that $T_1T_2=T_2T_1Q$ if and only if $Ker T_1^*T_2^*\subseteq Ker T_2^*T_1^*$.
\end{enumerate}
\end{thm}
\begin{proof}
	This follows using Lemma \ref{Dlemma} and the explanation in paragraph before this theorem. But we also prove this independently via an elementary construction. 
	
	First suppose there is a $Q\in \mathcal{B}(\mathcal{H})$ such that $T_1T_2=QT_2T_1$. Then clearly, $T_2T_1x=0$ implies that $T_1T_2x=QT_2T_1x=0$. Hence $Ker T_2T_1\subseteq Ker T_1T_2$. 
	
	Conversely, suppose $Ker T_2T_1\subseteq Ker T_1T_2$. Since, $\mathcal{H}$ is a finite dimensional Hilbert space with dimension $n$ say, let $\mathcal{B}_1=\{ T_2T_1e_1,T_2T_1e_2,\dots , T_2T_1e_k  \}$ be an orthonormal basis of $Ran T_2T_1 $ where, $k\leq n$. Let $\mathcal{M}=Span \mathcal{B}_1$, define $Q$ on $\mathcal{B}_1$ by $Q(T_2T_1e_j)=T_1T_2e_j$ for all $1\leq j \leq k$ and extend it to $\mathcal{M}$ by linearity. Now for any $m\in \mathcal{M}$, \[QT_2T_1m=\sum_{j=1}^{k}a_jT_1T_2e_j=T_1T_2\sum_{j=1}^{k}a_je_j.\] 
	Now clearly, $T_2T_1m\in \mathcal{M}$, therefore, $T_2T_1m=\sum_{j=1}^{k}a_jT_2T_1e_j=T_2T_1\sum_{j=1}^{k}a_je_j$. Hence, $m-\sum_{j=1}^{k}a_je_j\in Ker T_2T_1$. Now as $Ker T_2T_1\subseteq Ker T_1T_2$, we obtain, $T_1T_2(m-\sum_{j=1}^{k}a_je_j)=0$, so 
	\[T_1T_2m=T_1T_2(\sum_{j=1}^{k}a_je_j)=QT_2T_1m.\]
	 This proves that $T_1T_2=QT_2T_1$ on $\mathcal{M}$. Clearly, $Q$ is a bounded linear operator from $\mathcal{M}$ to $\mathcal{H}$. Now on $\mathcal{M}^{\perp}$ define $Q=T_1T_2|_{\mathcal{M}^{\perp}}$. Since $\mathcal{H}$ is finite dimensional, $\mathcal{M}^{\perp}=Ker T_2T_1$. So for $h\in \mathcal{H}$, $h=m+m'$ for some $m\in \overline{\mathcal{M}}$ and $m'\in \mathcal{M}^{\perp}=Ker T_2T_1$. Hence,  
	 \[QT_2T_1h=QT_2T_1m+QT_2T_1m'=QT_2T_1m=T_1T_2m.\] 
	 Since $Ker T_2T_1\subseteq Ker T_1T_2$ and $m'\in Ker T_2T_1$, we obtain 
	 \[ QT_2T_1h=T_1T_2m=T_1T_2m+T_1T_2m'=T_1T_2h.\]
	 
	 This proves $(1)$. We can prove $(2)$ similarly.
\end{proof}

\vspace{0.4cm}

	\section{The Q-commutant liftings and Q-commuting co-isometric extensions}\label{QSection1}
	
	\vspace{0.4cm}
	
\noindent In this Section, we generalize the commutant lifting theorem and obtain an analogue of it for $Q$-commuting operators. We look at the $Q_L$, $Q_M$ and $Q_R$-commuting generalizations separately.  Recall that if $\mathcal{H}_1\subseteq\mathcal{K}_1$ and $\mathcal{H}_2\subseteq\mathcal{K}_2$ are any Hilbert spaces then $Y\in\mathcal{B}(\mathcal{K}_1,\mathcal{K}_2)$ is said to be a lift of $X\in \mathcal{B}(\mathcal{H}_1,\mathcal{H}_2)$ if $\mathcal{H}_2$ is invariant for $Y^*$ and $ Y^*(h)=X^*(h)$ for all $h\in \mathcal{H}_2$. Given $Q\in \mathcal{B}(\mathcal{H})$, any two operators $S_1,S_2$ on a Hilbert space $\mathcal{H}$ are said to be $Q$ commuting if either $S_1S_2=QS_2S_1$ ($Q_L$-commuting) or $S_1S_2=S_2QS_1$ ($Q_M$-commuting) or $S_1S_2=S_2S_1Q$ ($Q_R$-commuting). 
	
	\subsection{The $Q_L$-commutant lifting theorems}
	Let $Q\in \mathcal{B}(\mathcal{H})$ be arbitrary. Then the commutant lifting analogue for $Q_L$-commuting contractions can be obtained in the following two ways. 
	\begin{enumerate}
		\item[I.] If $T\in \mathcal{B}(\mathcal{H})$ is a contraction with $V\in \mathcal{B}(\mathcal{K})$ being an isometric lift of $T$ and $X\in \mathcal{B}(\mathcal{H})$ is any operator such that $XT=QTX $ then we want to obtain a norm preserving lift $Y\in \mathcal{B}(\mathcal{K})$ of $X$ such that $YV=\overline{Q}VY$, for some $\overline{Q}\in \mathcal{B}(\mathcal{K})$. 
		\item[II.] If $T\in \mathcal{B}(\mathcal{H})$ is a contraction with $V\in \mathcal{B}(\mathcal{K})$ being an isometric lift of $T$ and $X\in \mathcal{B}(\mathcal{H})$ is any operator such that $TX=QXT$ then we want to obtain a lift $Y\in \mathcal{B}(\mathcal{K})$ of $X$ such that $VY=\overline{Q}YV$, for some $\overline{Q}\in \mathcal{B}(\mathcal{K})$.
	\end{enumerate} We work via the first approach to obtain (I) when $Q$ is any bounded operator for which $QT$ is a contraction. Indeed, this is one of the main results of this section.      
	\begin{thm}\label{QLT commutant lift}
		Let $Q,T,X\in \mathcal{B}(\mathcal{H})$ be such that $T$ and $QT$ are contractions and $XT=QTX$. Let $(V,\mathcal{K})$ be an isometric lift of $T$. Let $(\overline{Q},\mathcal{K})$ be any lift of $Q$ such that $\overline{Q}V$ is a contraction. Then there is a norm preserving lift $Y$ of $X$ such that $YV=\overline{Q}VY$. 
	\end{thm}
\begin{proof}
		First assume that $\|X\|=1$. Given the isometric lift $(V,\mathcal{K})$ of $(T,\mathcal{H})$ and a lift $\overline{Q}$ of $Q$ such that $\overline{Q}V$ is a contraction, we obtain that $\overline{Q}V$ is a contractive lift of $QT$. Therefore, by Theorem \ref{contractive lift}, there is a contractive lift $Y$ of $X$ such that $YV=\overline{Q}VY$. Since $\|X\|=1$ and $Y$ is a contractive lift of $X$, we have $1\geq \|Y\|\geq \|X\|=1$. Hence $\|Y\|=\|X\|=1$. Now if $0<\|X\|\neq 1$, then $XT=QTX$ implies that $\dfrac{X}{\|X\|}T=QT\dfrac{X}{\|X\|}$. Hence there is a norm preserving lift $\hat{Y}$ of $\dfrac{X}{\|X\|}$ such that $\hat{Y}V=\overline{Q}V\hat{Y}$. Therefore, $\|X\|\hat{Y}V=\overline{Q}V\|X\|\hat{Y}$. Hence $Y=\|X\|\hat{Y}$ is a lift of $X$ such that $YV=\overline{Q}VY$ and $\|Y\|=\|X\|$.      
	\end{proof}
Note that in previous theorem we did not assume $Q$ to be a contraction. Rather the assumption that $QT$ is a contraction is weaker than that. Clearly, $(V,\mathcal{K})$ is an isometric lift of contraction $(T,\mathcal{H})$ if and only if $V^*$ is a co-isometric extension of $T^*$. Hence Theorem \ref{QLT commutant lift} can be restated as follows.
\begin{thm}\label{QRT co-iso ext}
	Let $Q,T,X\in \mathcal{B}(\mathcal{H})$ be such that $T$ and $TQ$ are contractions and $TX=XTQ$. Let $(Z,\mathcal{K})$ be a co-isometric extension of $T$ and $(\overline{Q},\mathcal{K})$ be any extension of $Q$ such that $Z\overline{Q}$ is a contraction. Then there is a norm preserving extension $Y$ of $X$ such that $ZY=YZ\overline{Q}$. 
\end{thm}
	\begin{proof}
		Clearly $TX=XTQ$ implies that $X^*T^*=Q^*T^*X^*$. Since $(Z, \mathcal{K})$ is a co-isometric extension of $T$, we know that $Z^*$ will be an isometric lift of $T^*$. Also $Z\overline{Q}$ is a contraction implies that $\overline{Q}^*Z^*$ is so. Hence by Theorem \ref{QLT commutant lift}, there is a norm preserving lift $S$ of $X^*$ such that $SZ^*=\overline{Q}^*Z^*S $. Taking $Y=S^*$ proves the desired result.
	\end{proof}

Note that in Theorem \ref{QLT commutant lift} and Theorem \ref{QRT co-iso ext}, if in particular $Q$, $\overline{Q}$ are considered to be contractions such that $\overline{Q}=Q\oplus Q'$ with respect to the decomposition $\mathcal{K}=\mathcal{H}\oplus \mathcal{H}^{\perp}$ then as a special case of our result we obtain Theorem \ref{Qcomliftsumesh}-(i) and Theorem \ref{Qcomextsumesh}-(i).\\

	Now we look at the second way, i.e. point (II) mentioned at the beginning of this section and in particular consider $TX=QXT$ for when $Q\in \mathcal{B}(\mathcal{H})$ is a co-isometry, $T\in \mathcal{B}(\mathcal{H})$ is a contraction and $X\in \mathcal{B}(\mathcal{H})$. We use techniques similar to that of Douglas, Muhly and Pearcy as in \cite{Dou:Muh:Pea} to obtain a $Q$-commutant lifting result (Theorem \ref{QminiisoliftDouglass}). For that we need the following theorem which is proved using a similar idea as that of the proof of Theorem 3 in \cite{Dou:Muh:Pea}.  
	\begin{thm}\label{partialisointertwine}
	Suppose $T$ is a contraction and $X, Y$ are bounded linear operators on a Hilbert space $\mathcal{H}$ such that $Y\neq 0, \|X\|\leq \|Y\|$ and $XT=TY$. Let $S_0\in \mathcal{B}(\mathcal{L},\mathcal{H})$ be a contraction such that $Z=\begin{bmatrix}
		T & S_0\\
		0 & 0
	\end{bmatrix}$ on $\mathcal{H}\oplus \mathcal{L}$ satisfies $TT^*+S_0S_0^*=I_{\mathcal{H}}$. Then there is an operator $Y_1=\begin{bmatrix}
		Y & C \\
		0 & W
	\end{bmatrix}$ on $\mathcal{H}\oplus \mathcal{L}$ such that for any $X_1=\begin{bmatrix}
		X & A\\
		0 & B
	\end{bmatrix}$ on $\mathcal{H}\oplus \mathcal{L}$, we have $X_1Z=ZY_1$ and $\|Y_1\|=\|Y\|$.
\end{thm}
\begin{proof}
	Without loss of generality assume that $\|Y\|=1$ and $\|X\|\leq 1$. Otherwise, for any non zero bounded operator $Y$, we divide both the sides of equation $ XT=TY$ by $\|Y\|$. We want to find bounded operators $C$ and $W$ such that \[Y_1=\begin{bmatrix}
		Y & C\\
		0 & W
	\end{bmatrix}:\mathcal{H}\oplus \mathcal{L}\to \mathcal{H}\oplus \mathcal{L}\] is a contraction and for any $X_1=\begin{bmatrix}
		X & A\\
		0 & B
	\end{bmatrix}$ on $\mathcal{H}\oplus \mathcal{L}$, $X_1Z=ZY_1$. Therefore, we want to find bounded operators $C$ and $W$ such that $XS_0=TC+S_0W$ and $\|Y_1\|\leq 1$. Due to the Proposition \ref{prop:21}, it is equivalent to find a contraction $F$ from $\mathcal{L}$ to $\mathcal{H}$ such that $C=D_{Y^*}FD_{W}$ and $XS_0=TC+S_{0}W$, that is to find a contraction $F:\mathcal{L}\to \mathcal{H}$ such that,
	\begin{equation}\label{eqn1}
		XS_0 = TD_{Y^*}FD_W+S_0W.
	\end{equation}  
	We observe that,  
	\begingroup
	\allowdisplaybreaks
	\begin{alignat*}{2}
		&(TD_{Y^*})(TD_{Y^*})^*+S_0S_0^*\\
		=&T (I_{\mathcal{H}}-YY^*)T^* + S_0S_0^*\\
		= & I_{\mathcal{H}}-TYY^*T^*&[\text{As } TT^*+S_0S_0^*=I_{\mathcal{H}} ] \\
		= & I_{\mathcal{H}}-XTT^*X^* &[\text{As } TY=XT]\\
		\geq &I_{\mathcal{H}}-XTT^*X^* - (I-XX^*)\\
		=& XX^*-XTT^*X^*=(XS_0)(XS_0)^*. &\quad[\text{As } TT^*+S_0S_0^*=I_{\mathcal{H}} ]
	\end{alignat*}
	\endgroup
	Hence  
	\[
	(TD_{Y^*})(TD_{Y^*})^*+S_0S_0^* \geq (XS_0)(XS_0)^*.
	\]
	Therefore, by Theorem \ref{Dthm1}, there are contractions $R$ and $W$ satisfying \begin{equation}\label{eqn2}
		XS_0= TD_{Y^*}R + S_0W
	\end{equation} and $R^*R+W^*W\leq I_{\mathcal{H}}$. Since $R^*R\leq I_{\mathcal{H}}-W^*W$, applying Lemma \ref{Dlemma}, we obtain a contraction $F^*$ such that $R^* = D_WF^*$. Let $C=D_{Y^*}FD_W$. Hence we have contractions $F, W$ for which \eqref{eqn1} holds. Hence $X_1Z=ZY_1$. Further since $Y,W$ are contractions and $C=D_{Y^*}FD_W$ for some contraction $F$, it follows from Proposition \ref{prop:21} that $Y_1=\begin{bmatrix}
		Y & C\\
		0 & W
	\end{bmatrix}$ is a contraction. Hence $1=\|Y\|\leq \|Y_1\|\leq 1$, that is, $\|Y_1\|=\|Y\|=1$. 
\end{proof}
	We will prove the desired lifting result for $Q_L$-commuting operators (Theorem \ref{QminiisoliftDouglass}) by first proving the following $Q_R$-commuting extension result.
	\begin{thm}\label{Q_Lcommminext}
		Let $T\in \mathcal{B}(\mathcal{H})$ be a contraction, $Q$ be any co-isometry on $\mathcal{H}$ and $X\in \mathcal{B}(\mathcal{H})$ be such that $XT=TXQ$. Let $(Z, \mathcal{K})$ be the minimal co-isometric extension of $T$. Then there is a co-isometric extension $\overline{Q}\in \mathcal{B}(\mathcal{K})$ of $Q$ and a norm preserving extension $\widetilde{{X}}\in \mathcal{B}(\mathcal{K})$ of $X$ such that $\widetilde{{X}}Z=Z\widetilde{{X}}\overline{Q}$. 
		           
	\end{thm}
	\begin{proof}
		As the minimal co-isometric extension of a contraction is unique, we consider 
		\[
		Z=\begin{bmatrix}
			T & D_{T^*} & 0 & 0 & \dots\\
			0 & 0 & I_{\mathcal{D}_{T^*}} & 0 & \dots\\
			0 & 0 & 0 & I_{\mathcal{D}_{T^*}} & \dots\\
			\vdots & \vdots & \vdots & \vdots & \ddots
		\end{bmatrix} \text{ on } \mathcal{K}= \mathcal{H}\oplus \mathcal{D}_{T^*}\oplus \mathcal{D}_{T^*}\oplus \mathcal{D}_{T^*}\oplus \cdots
		,\]
		to be the minimal co-isometric extension of $T$. Suppose 
		\[
		\mathcal{K}_n=\mathcal{H} \oplus \underbrace{\mathcal{D}_{T^*}\oplus\dots\oplus \mathcal{D}_{T^*}}_{n \text{ copies}} \oplus \{0\}\oplus \{0\}\oplus\cdots, \quad \text{ for }n \geq 1
		\]
		and 
		\[
		\mathcal{L}_n=\{0\}\oplus \dots\oplus\{0\}\oplus \underbrace{\mathcal{D}_{T^*}}_{(n+1)\text{th position}} \oplus \{0\}\oplus \cdots.
		\]
		Let $q$ be any complex number of modulus one and let $\overline{Q}=Q\oplus qI$ with respect to the decomposition $\mathcal{K}=\mathcal{H}\oplus \mathcal{H}^{\perp}$. Clearly, for any $n\in \mathbb{N},$ $\mathcal{K}_n$ is invariant subspace for $Z$. Therefore, define $Z_n=Z|_{\mathcal{K}_n}$ for all $n\in\mathbb{N}$. Observe that for $h=(h_0,h_1,\ldots ,h_n,0,\ldots)\in \mathcal{K}_n$, 
		\[Z_{n+1}h=Z(h_0,h_1,\ldots ,h_n,0,\ldots)=(Th_0+D_{T^*}h_1,h_2,h_3,\ldots,h_{n},0,\ldots)=Z_nh\in \mathcal{K}_n.\] Hence clearly $\mathcal{K}_n$ is invariant under $Z_{n+1}$ and with respect to the decomposition $\mathcal{K}_{n+1}=\mathcal{K}_n\oplus \mathcal{L}_{n+1}$, $Z_{n+1}$ can be written as $Z_{n+1}=\begin{bmatrix}
			Z_n & S_n \\
			0 & 0
		\end{bmatrix}$ for an operator $S_n\in \mathcal{B}(\mathcal{L}_{n+1},\mathcal{K}_n)$. Further, as $Z$ is a co-isometry, a simple computation shows that for all $n\in \mathbb{N}$, $Z_nZ_n^* + S_nS_n^*=I_{\mathcal{K}_n}$. If we consider $Y=XQ$, then from the given condition we have $TY=XT$. Since $Q$ is a co-isometry, $\|X\|=\|XQ\|$. Hence for $Z_1=\begin{bmatrix}
			Z_0 & S_0\\
			0 & 0
		\end{bmatrix}=\begin{bmatrix}
			T & D_{T^*}\\
			0 & 0
		\end{bmatrix}$ by Theorem \ref{partialisointertwine}, there is an operator  $Y_1=\begin{bmatrix}
			Y & C_1\\
			0 & D_1
		\end{bmatrix}$ such that for any $X_1=\begin{bmatrix}
			X & A_1\\
			0 & B_1
		\end{bmatrix}$,
		\begin{equation}\label{eqn01'}
			\|Y_1\|=\|Y\| \text{ and }	X_1Z_1=Z_1Y_1.
		\end{equation} 
		Observe that $\mathcal{K}_m $ is invariant subspace for $\overline{Q}$, thus let  $Q_m=\overline{Q}|_{\mathcal{K}_m}=Q\oplus qI_{\mathcal{K}_n\ominus\mathcal{H}}$ for all $m\in \mathbb{N}$. Hence with respect to the decomposition $\mathcal{K}_m=\mathcal{K}_{m-1}\oplus \mathcal{L}_m$, $Q_m=\begin{bmatrix}
			Q_{m-1}&0\\0&qI
		\end{bmatrix}$. In particular, with respect to the decomposition $\mathcal{K}_1=\mathcal{H}\oplus \mathcal{D}_{T^*}$, $Q_1=
		\begin{bmatrix}
			Q&0\\0&qI
		\end{bmatrix}$. Let $A_1=q^{-1}C_1$ and $B_1=q^{-1}D_1$ then $X_1=\begin{bmatrix}
			X&q^{-1}C_1\\0&q^{-1}D_1
		\end{bmatrix}$, $Y_1=X_1Q_1$ and hence from \eqref{eqn01'} we have $X_1Z_1=Z_1X_1Q_1$. Now $\|Y_1\|=\|Y\|$ implies that $\|X_1Q_1\|=\|XQ\|$. Further, as $Q$ is a co-isometry, so is $Q_1$. Hence we have, 
		\begin{equation}\label{norm1} 
			\|X_1\|=\|X_1Q_1Q_1^*\|\leq \|X_1Q_1\|\; \|Q_1^*\|= \|XQ\|\leq \|X\|\leq \|X_1\|,
		\end{equation} where the last inequality follows as $X_1$ is an extension of $X$. Thus $\|X_1\|=\|X\|$. Assume that for $1\leq k \leq n-1$, there is an operator $X_k=\begin{bmatrix}
			X_{k-1} & q^{-1}C_{k}\\
			0 & q^{-1}D_k
		\end{bmatrix}$ such that, 
		\begin{equation}\label{eqn02'}
			X_kZ_k=Z_kX_kQ_k, \quad \|X_k\|=\|X\|.
		\end{equation}
		Let $Y_{n-1}=X_{n-1}Q_{n-1}$. Then from \eqref{eqn02'} we have $Z_{n-1}Y_{n-1}=X_{n-1}Z_{n-1}$. Applying Theorem \ref{partialisointertwine}, there is an operator $Y_n=\begin{bmatrix}
			Y_{n-1} & C_n\\
			0 & D_{n}
		\end{bmatrix}$ such that for any $X_n=\begin{bmatrix}
			X_{n-1} & A_n\\
			0 & B_n
		\end{bmatrix}$ we have $Z_nY_n=X_nZ_n$ and $\|Y_n\|=\|Y_{n-1}\|=\|Y\|$. Let $A_n = q^{-1}C_n$ and $B_n=q^{-1}D_n$. Thus, we have $Y_n=X_nQ_n$ and $X_nZ_n=Z_nX_nQ_n$. A similar argument as in \eqref{norm1} gives us that $\|X_n\|=\|X\|$. Hence by induction, we obtain the sequences $\{Q_n\}_{n=1}^{\infty}$ and $\{X_n\}_{n=1}^{\infty}$ such that $X_nZ_n=Z_nX_nQ_n$ and $\|X_n\|=\|X\|$. We can consider $Q_n, X_n$ to be operators on $\mathcal{K}$ by defining $Q_n=0$ and $ X_n=0$ on $\mathcal{K}\ominus \mathcal{K}_n$. It is clear that $\bigcup\limits_{n=1}^{\infty} \mathcal{K}_n$ is dense in $\mathcal{K}$. Since $X_{n+1}|_{\mathcal{K}_n}=X_n$, for any $x \in \bigcup\limits_{n=1}^{\infty} \mathcal{K}_n$, the sequence $\{X_nx\}$ is cauchy and say $\widetilde{X}$ is the pointwise limit of the sequence $\{X_n\}$. Since, $\|X_n\|=\|X\|$ for all $n\in \mathbb{N}$ so $\widetilde{{X}}\in \mathcal{B}(\mathcal{K})$ and $\|\widetilde{X}\|=\|X\|$. Again by the construction of $Q_n$ it is clear that $\{Q_n\}$ converges strongly to $\overline{Q}$ in $\mathcal{B}(\mathcal{K})$. Therefore, we have a norm preserving extension $\widetilde{{X}}\in \mathcal{B}(\mathcal{K})$ of $X$ such that $$\widetilde{{X}}Z=Z\widetilde{{X}}\overline{Q}.$$
	\end{proof}
	Further removing the minimality condition, we obtain the following.
	\begin{thm}\label{Q_Lcomext}
		Let $T\in \mathcal{B}(\mathcal{H})$ be contractions, $Q$ be any co-isometry on $\mathcal{H}$ and $X\in \mathcal{B}(\mathcal{H})$ be such that $XT=TXQ$. Let $(Z, \mathcal{K})$ be a co-isometric extension of $T$. Then there is a co-isometric extension $\overline{Q}\in \mathcal{B}(\mathcal{H})$ of $Q$ and a norm preserving extension $\widetilde{{X}}\in \mathcal{B}(\mathcal{K})$ of $X$ such that $\widetilde{{X}}Z=Z\widetilde{{X}}\overline{Q}$. 
		Moreover, $\overline{Q}$ can be chosen to be $\overline{Q}=Q\oplus qI$ with respect to the decomposition $\mathcal{K}=\mathcal{H}\oplus \mathcal{H}^{\perp} $, for any complex number $q$ of modulus one.      
	\end{thm}
	\begin{proof}
		The idea of the proof is same as that of Corollary 4.1 of \cite{Dou:Muh:Pea}. Let $\overline{Q}=Q\oplus qI$ for any complex number $|q|=1$. Consider the smallest reducing subspace $\mathcal{K}_1\subseteq \mathcal{K}$ for $Z$ containing $\mathcal{H}$ and let $Z=Z|_{\mathcal{K}_1}$. Therefore, $(Z,\mathcal{K}_1)$ is a  minimal co-isometric extension of $(T,\mathcal{H})$. By Theorem \ref{Q_Lcommminext}, there is a norm preserving extension $(S,\mathcal{K}_1)$ of $X$ such that $SZ=ZS\overline{Q}|_{\mathcal{K}_1}$. Now define $S$ on $\mathcal{K}=\mathcal{K}_1\oplus (\mathcal{K}\ominus \mathcal{K}_1)$ by considering $\widetilde{X}|_{\mathcal{K}_1}=S$ and $\widetilde{X}\equiv 0$ on $\mathcal{K}\ominus \mathcal{K}_1$. Then $\widetilde{X}$ has the desired properties. This completes the proof.      
	\end{proof}
	Again we can restate it as the following lifting theorem. 
	\begin{thm}\label{QminiisoliftDouglass}
		Let $T\in \mathcal{B}(\mathcal{H})$ be contractions, $Q$ be an isometry on $\mathcal{H}$ and $X\in \mathcal{B}(\mathcal{H})$ be such that $TX=QXT$. Let $(V, \mathcal{K})$ be any isometric lift of $T$. Then there is an isometric lift $\overline{Q}\in \mathcal{B}(\mathcal{H})$ of $Q$ and a norm preserving lift $\widetilde{{X}}\in \mathcal{B}(\mathcal{K})$ of $X$ such that $V\widetilde{{X}}=\overline{Q}\widetilde{{X}}V$.
		Moreover, $\overline{Q}$ can be chosen to be $\overline{Q}=Q\oplus qI$ with respect to the decomposition $\mathcal{K}=\mathcal{H}\oplus \mathcal{H}^{\perp} $, for any complex number $q$ of modulus one.            
	\end{thm}
	\begin{proof}
		Clearly $TX=QXT$ implies that $X^*T^*=T^*X^*Q^*$. Since $(V, \mathcal{K})$ is an isometric lift of $T$, we know that $V^*$ will be an co-isometric extension of $T^*$. Also $Q$ is an isometry implies that $Q^*$ is a co-isometry. Let $q$ be any complex number of modulus one. Hence for $\overline{Q}=Q\oplus qI$ on $\mathcal{K}=\mathcal{H}\oplus \mathcal{H}^{\perp}$, by Theorem \ref{Q_Lcomext}, there is a norm preserving extension $Y$ of $X^*$ such that $YV^*=V^*Y\overline{Q}^* $. Taking $\widetilde{X}=Y^*$ proves the desired result.
	\end{proof}
	
	\subsection{The $Q_M$-commutant lifting theorems}
	Suppose $Q,X,T\in \mathcal{B}(\mathcal{H})$ are such that $XT=TQX$ and let $(V,\mathcal{K})$ be an isometric lift of $T$. Then can we find a norm preserving lift $Y\in \mathcal{B}(\mathcal{K})$ of $X$ such that for some lift $\overline{Q}\in \mathcal{B}(\mathcal{K})$ of $Q$, $YV=V\overline{Q}Y$? We address this question bellow. Note that in Theorem \ref{QMT commutant lift}, if in particular $Q$ is assumed to be a contraction and $\overline{Q}$ to be a contraction of form $Q\oplus Q'$ then the hypotheses of $TQ$ and $V\overline{Q}$ being contractions are automatically satisfied. Hence we obtain Theorem \ref{Qcomextsumesh}-(ii) as a corollary of Theorem \ref{QMT commutant lift} bellow.
	\begin{thm}\label{QMT commutant lift}
		Let $Q,X,T\in \mathcal{B}(\mathcal{H})$ be such that $T$ and $TQ$ are contractions and $XT=TQX$. Let $(V,\mathcal{K})$ be an isometric lift of $T$. Let $(\overline{Q},\mathcal{K})$ be any lift of $Q$ such that $V\overline{Q}$ is a contraction. Then there is a norm preserving lift $Y$ of $X$ such that $YV=V\overline{Q}Y$. 
	\end{thm}
	\begin{proof}
		First assume that $\|X\|=1$. As $(V,\mathcal{K})$ is an isometric lift of $(T,\mathcal{H})$ and $\overline{Q}$ is any lift of $Q$ such that $V\overline{Q}$ is a contraction, we obtain that $V\overline{Q}$ is a contractive lift of $TQ$. Therefore, by Theorem \ref{contractive lift},
		there is a contractive lift $Y$ of $X$ such that $YV=V\overline{Q}Y$. Since $\|X\|=1$ and $Y$ is a contractive lift of $X$, so $1\geq \|Y\|\geq \|X\|=1$. Hence $\|Y\|=\|X\|=1$. Now if $0<\|X\|\neq 1$, then $XT=TQX$ implies that $\dfrac{X}{\|X\|}T=TQ\dfrac{X}{\|X\|}$. Thus there is a norm preserving lift $\hat{Y}$ of $\dfrac{X}{\|X\|}$ such that $\hat{Y}V=V\overline{Q}\hat{Y}$. Consequently, $\|X\|\hat{Y}V=V\overline{Q}\|X\|\hat{Y}$. Hence $Y=\|X\|\hat{Y}$ is a lift of $X$ such that $YV=V\overline{Q}Y$ and $\|Y\|=\|X\|$. 
	\end{proof}
	
	Since $(V,\mathcal{K})$ is an isometric lift of $T$ if and only if $(V^*,\mathcal{K})$ is an co-isometric extension of $T^*$, Theorem \ref{QMT commutant lift} can be restated as follows. Note that Theorem \ref{QMT co-iso ext} generalizes Theorem \ref{Qcomextsumesh}-(ii).
	\begin{thm}\label{QMT co-iso ext}
		Let $Q,X,T\in \mathcal{B}(\mathcal{H})$ be such that $T$ and $QT$ are contractions and $TX=XQT$. Let $(Z,\mathcal{K})$ be a co-isometric extension of $T$ and $(\overline{Q},\mathcal{K})$ be any extension of $Q$ such that $\overline{Q}Z$ is a contraction. Then there is a norm preserving extension $Y$ of $X$ such that $ZY=Y\overline{Q}Z$. 
	\end{thm}
\begin{proof}
	Follows by a similar argument as that of the proof of Theorem \ref{QRT co-iso ext}.
\end{proof}
	\subsection{The $Q_R$-commutant lifting theorems}
	Let $Q\in \mathcal{B}(\mathcal{H})$ be arbitrary. Then the commutant lifting analogue for $Q_R$-commuting contractions can be obtained in the following two ways. 
\begin{enumerate}
	\item[I.] If $T\in \mathcal{B}(\mathcal{H})$ is a contraction with $V\in \mathcal{B}(\mathcal{K})$ being an isometric lift of $T$ and $X\in \mathcal{B}(\mathcal{H})$ is any operator such that $TX=XTQ $ then we want to obtain a norm preserving lift $Y\in \mathcal{B}(\mathcal{K})$ of $X$ such that $VY=YV\overline{Q}$, for some lift $\overline{Q}\in \mathcal{B}(\mathcal{K})$ of $Q$. 
	\item[II.] If $T\in \mathcal{B}(\mathcal{H})$ is a contraction with $V\in \mathcal{B}(\mathcal{K})$ being an isometric lift of $T$ and $X\in \mathcal{B}(\mathcal{H})$ is any operator such that $XT=TXQ$ then we want to obtain a lift $Y\in \mathcal{B}(\mathcal{K})$ of $X$ such that $YV=VY\overline{Q}$, for some lift $\overline{Q}\in \mathcal{B}(\mathcal{K})$ of $Q$.
\end{enumerate} We address both these cases separately and obtain the following results. 
	\begin{thm}\label{Q_Riso}
		Let $Q,T\in \mathcal{B}(\mathcal{H})$ be any contractions and $X\in \mathcal{B}(\mathcal{H})$ be such that $TX=XTQ$. Let $V,\overline{Q}$ on $\mathcal{K}$ be any isometric lifts of $T,Q$ respectively. Then there is a norm preserving lift $Y$ of $X$ such that $VY=YV\overline{Q}$.
	\end{thm}
	\begin{proof}
		Since $V,\overline{Q}$ are isometric lifts of $T, Q$ respectively, $V\overline{Q} $ is an isometric lift of $TQ$. Hence by the intertwining lifting theorem, Theorem \ref{intertwining}, there is a norm preserving lift $Y$ of $X$ such that $ VY=YV\overline{Q}$.
	\end{proof}
	Again since $V$ is an  isometric lift of $T$ on $\mathcal{K}$ if and only $V^*$ is a co-isometric extension of $T^*$, we can restate the above theorem in terms of co-isometric extension as follows.
	\begin{thm}\label{Q_Rcoiso}
		Let $Q,T\in \mathcal{B}(\mathcal{H})$ be any contractions and $X\in \mathcal{B}(\mathcal{H})$ be such that $XT=QTX$. Let $\overline{Q},Z$ on $\mathcal{K}$ be any co-isometric extensions of $Q,T$ respectively. Then there is a norm preserving extension $Y$ of $X$ such that $YZ=\overline{Q}ZY$.
	\end{thm}
	
	Now we address the point II. The proof of the following result is based on the techniques similar to that of Douglas, Muhly and Pearcy in \cite{Dou:Muh:Pea}. 
	\begin{thm}\label{QminicoisoextDouglass}
		Let $Q, T\in \mathcal{B}(\mathcal{H})$ be contractions and $X\in \mathcal{B}(\mathcal{H})$ be such that $TX=QXT$. Let $(Z, \mathcal{K})$ be a minimal co-isometric extension of $T$. If for each $n\in \mathbb{N}$, $$\mathcal{K}_n=\mathcal{H} \oplus \underbrace{\mathcal{D}_{T^*}\oplus\dots\oplus \mathcal{D}_{T^*}}_{n \text{ copies}} \oplus \{0\}\oplus \{0\}\oplus\cdots$$ and $\overline{Q}\in \mathcal{B}(\mathcal{H})$ is an extension of $Q$ such that for all $n\in \mathbb{N}$, $\mathcal{K}_n$ is invariant for $\overline{Q}$ . Then there is a norm preserving extension $\widetilde{{X}}\in \mathcal{B}(\mathcal{K})$ of $X$ such that $Z\widetilde{{X}}=\overline{Q}\widetilde{{X}}Z$.      
	\end{thm}
	\begin{proof}
		As the minimal co-isometric extension of a contraction is unique, we consider 
		\[
		Z=\begin{bmatrix}
			T & D_{T^*} & 0 & 0 & \dots\\
			0 & 0 & I_{\mathcal{D}_{T^*}} & 0 & \dots\\
			0 & 0 & 0 & I_{\mathcal{D}_{T^*}} & \dots\\
			\vdots & \vdots & \vdots & \vdots & \ddots
		\end{bmatrix} \text{ on } \mathcal{K}= \mathcal{H}\oplus \mathcal{D}_{T^*}\oplus \mathcal{D}_{T^*}\oplus \mathcal{D}_{T^*}\oplus \cdots,
		\]
		to be the minimal co-isometric extension of $T$. Suppose 
		\[
		\mathcal{K}_n=\mathcal{H} \oplus \underbrace{\mathcal{D}_{T^*}\oplus\dots\oplus \mathcal{D}_{T^*}}_{n \text{ copies}} \oplus \{0\}\oplus \{0\}\oplus\cdots \quad, \text{ for }n \geq 1
		\]
		and 
		\[
		\mathcal{L}_n=\{0\}\oplus \dots\oplus\{0\}\oplus \underbrace{\mathcal{D}_{T^*}}_{(n+1)\text{th position}} \oplus \{0\}\oplus \cdots.
		\]
		Let $\overline{Q} \in \mathcal{B}(\mathcal{K})$ be any contractive extension of $Q$ such that $ \mathcal{K}_n$ is invariant for $\overline{Q}$ for each $n\in \mathbb{N} $. Clearly, for all $n\in \mathbb{N}$, $ \mathcal{K}_n$ is an invariant subspace for $Z$. Define $Z_n=Z|_{\mathcal{K}_n}$. Observe that for $h=(h_0,h_1,\ldots ,h_n,0,\dots)\in \mathcal{K}_n$, 
		\[Z_{n+1}h=Z(h_0,h_1,\ldots ,h_n,0,\dots)=(Th_0+D_{T^*}h_1,h_2,h_3,\ldots,h_{n},0,\ldots)=Z_nh\in \mathcal{K}_n.\] Hence it is clear that $\mathcal{K}_n$ is invariant under $Z_{n+1}$ and with respect to the decomposition $\mathcal{K}_{n+1}=\mathcal{K}_n\oplus \mathcal{L}_{n+1}$, $Z_{n+1}$ can be written as $Z_{n+1}=\begin{bmatrix}
			Z_n & S_n \\
			0 & 0
		\end{bmatrix}$, where $S_n\in \mathcal{B}(\mathcal{L}_{n+1},\mathcal{K}_n)$. Further, as $Z$ is a co-isometry, a simple computation shows that for all $n\in \mathbb{N}$, $Z_nZ_n^* + S_nS_n^*=I_{\mathcal{K}_n}$. If we consider $S=QX$, then clearly $\|S\|= \|QX\|\leq \|X\|$ and $QXT=TX$ implies that $ST=TX$. Hence for $Z_1=\begin{bmatrix}
			Z_0 & S_0\\
			0 & 0
		\end{bmatrix}=\begin{bmatrix}
			T & D_{T^*}\\
			0 & 0
		\end{bmatrix}$ by Theorem \ref{partialisointertwine}, there is a norm preserving lift  $X_1=\begin{bmatrix}
			X & C_1\\
			0 & D_1
		\end{bmatrix}$ of $X$ such that for any $S_1=\begin{bmatrix}
			S & A_1\\
			0 & B_1
		\end{bmatrix}$, 
		\begin{equation}\label{eqn01}
			S_1Z_1=Z_1X_1.
		\end{equation} 
		Since $\mathcal{K}_n$ is an invariant subspace for $\overline{Q}$, define $Q_m=\overline{Q}|_{\mathcal{K}_m}$ for all $m\in \mathbb{N}$. Hence with respect to the decomposition $\mathcal{K}_m=\mathcal{K}_{m-1}\oplus \mathcal{L}_m$, $Q_m=\begin{bmatrix}
			Q_{m-1}&R_{m1}\\0&R_{m2}
		\end{bmatrix}$. In particular, since $\mathcal{K}_1$ and $\mathcal{H}$ are invariant for $\overline{Q}$, with respect to the decomposition $\mathcal{K}_1=\mathcal{H}\oplus \mathcal{D}_{T^*}$, $Q_1=
		\begin{bmatrix}
			Q&R_{11}\\0&R_{12}
		\end{bmatrix}$. Let $A_1=QC_1+R_{11}D_1$ and $B_1=R_{12}D_1$, then from  \eqref{eqn01} we get, $Q_1X_1Z_1=Z_1X_1$. Assume that for $1\leq k \leq n-1$, there is an operator $X_k=\begin{bmatrix}
			X_{k-1} & C_{k}\\
			0 & D_k
		\end{bmatrix}$ such that 
		\begin{equation}\label{eqn02}
			Q_kX_kZ_k=Z_kX_k, \text{ and } \|X_k\|=\|X\|.
		\end{equation}
		Let $S_{n-1}=Q_{n-1}X_{n-1}$. Then from \eqref{eqn02} we have $S_{n-1}Z_{n-1}=Z_{n-1}X_{n-1}$. Applying Theorem \ref{partialisointertwine}, there is an operator $X_n=\begin{bmatrix}
			X_{n-1} & C_n\\
			0 & D_{n}
		\end{bmatrix}$ such that for any $S_n=\begin{bmatrix}
			S_{n-1} & A_n\\
			0 & B_n
		\end{bmatrix}$ we have $S_nZ_n=Z_nX_n$ and $\|X_n\|=\|X_{n-1}\|=\|X\|$. So let $A_n = Q_{n-1}C_n+R_{n1}D_n$ and $B_n=R_{n2}D_n$. Therefore, $Q_nX_nZ_n=Z_nX_n$. So by induction, we obtain sequences $\{Q_n\}_{n=1}^{\infty}$ and $\{X_n\}_{n=1}^{\infty}$ such that $Q_nX_nZ_n=Z_nX_n$ and $\|X_n\|=\|X\|$. We can consider $Q_n, X_n$ to be operators on $\mathcal{K}$ by defining $Q_n=0$ and $ X_n=0$ on $\mathcal{K}\ominus \mathcal{K}_n$ respectively. It is clear that $\bigcup\limits_{n=1}^{\infty} \mathcal{K}_n$ is dense in $\mathcal{K}$. Since $X_{n+1}|_{\mathcal{K}_n}=X_n$, for any $x \in \bigcup\limits_{n=1}^{\infty} \mathcal{K}_n$, the sequence $\{X_nx\}$ is a Cauchy sequence. Since, $\|X_n\|=\|X\|$ for all $n\in \mathbb{N}$ so there is an operator $\widetilde{{X}}\in \mathcal{B}(\mathcal{K})$ which is a pointwise limit of the sequence $\{X_n\}$ and $\|\widetilde{X}\|=\|X\|$. It follows from the definition of $Q_n$ that the sequence $\{Q_n\}$ converges strongly to $\overline{Q}$ in $\mathcal{B}(\mathcal{K})$. Therefore, we have a norm preserving extension $\widetilde{{X}}\in \mathcal{B}(\mathcal{K})$ of $X$ such that $\overline{Q}\widetilde{{X}}Z=Z\widetilde{{X}}$.
	\end{proof}
	
	We give an alternative proof of this result using Sebestyen's approach. 
	\begin{thm}
		Let $Q,T\in \mathcal{B}(\mathcal{H})$ be contractions and $X\in \mathcal{B}(\mathcal{H})$ be such that $TX=QXT$. Let $(Z, \mathcal{K})$ be the minimal co-isometric extension of $T$. If $K_n = \bigvee\limits_{m=0}^n Z^{*m}\mathcal{H}$ and $\overline{Q}\in\mathcal{B}(\mathcal{K})$ is an extension of $Q$ such that, $\overline{Q}(K_n)\subseteq K_n$ for all $n\in \mathbb{N}$, then there is a norm preserving extension $\widetilde{X}\in \mathcal{B}(\mathcal{K})$ of $X$ such that $Z\widetilde{X}=\overline{Q}\widetilde{X}Z$. 
	\end{thm}
	
	\begin{proof}
		The proof follows by exact same technique as in the proof of Theorem 2 in \cite{Sebestyen}.
		Since $(V,\mathcal{K})$ is the minimal co-isometric extension of $T$, we have 
		\[
		\mathcal{K}=\bigvee_{n=0}^{\infty}Z^{*n}\mathcal{H}=\overline{span}\{ Z^{*n}h:h\in \mathcal{H}, n\in \mathbb{N}\cup\{0\} \}.
		\]
		Let $K_n = \bigvee\limits_{m=0}^n Z^{*m}\mathcal{H}$. Then $\bigcup\limits_{n = 0}^{\infty}K_n$ is dense in $\mathcal{K}$. Consider the orthogonal projections $P_n : \mathcal{K}\rightarrow K_n$ for all $n \geq 0$. Clearly 
		\begin{equation}\label{proj01}
			P_{n+1}Z^* = Z^*P_n.
		\end{equation}
		We construct the required operator $\widetilde{X}:\mathcal{K} \rightarrow \mathcal{K}$ inductively by finding operators $S_n : K_n \rightarrow K_n$ for all $n \in \mathbb N$.\\
		
		\noindent\textbf{Step I.} For the existence of $S_1:K_1\to K_1$, due to Theorem \ref{dualparrot}, it suffices to prove the following claim. \\
		
		\noindent\textbf{Claim.} For any $h_1, h_2\in \mathcal{H}$, $\langle Z^*X^*P_0\overline{Q}^*Z(Z^*h_1), h_2 \rangle = \langle Z^*h_1, Xh_2 \rangle$.\\
		
		\noindent \textit{Proof of Claim.} Since $Z$ is the minimal co-isometric extension of $T$, $Z|_{\mathcal{H}}=T$. Hence using this and Equation \eqref{proj01} we have 
		\begingroup
		\allowdisplaybreaks
		\begin{align*}
			\langle Z^*X^*P_0\overline{Q}^*Z(Z^*h_1), h_2 \rangle
			& = \langle X^*P_0\overline{Q}^*Z(Z^*h_1), Zh_2 \rangle\\
			& = \langle X^*P_0\overline{Q}^*h_1, Th_2 \rangle\\
			&=\langle P_0\overline{Q}^*h_1, XTh_2 \rangle\\
			&=\langle \overline{Q}^*h_1, XTh_2 \rangle\\
			&=\langle h_1, \overline{Q}XTh_2 \rangle\\
			& = \langle h_1, QXTh_2 \rangle & [\text{since }\overline{Q}|_{\mathcal{H}}=Q]\\
			& = \langle h_1, TXh_2 \rangle  & [\text{since } TX=QXT]\\
			& = \langle h_1, ZXh_2 \rangle\\
			& = \langle Z^*h_1, Xh_2 \rangle.
		\end{align*}
		\endgroup
		This completes the proof of the claim.\\
		
		\noindent Therefore, by Theorem \ref{dualparrot}, there is an operator $S_1: K_1 \to K_1$ such that $S_1|_{\mathcal{H}}=X$, $S_1^*|_{Z^*\mathcal{H}}=Z^*X^*P_0\overline{Q}^*Z|_{Z^*\mathcal{H}}$ and $\|S_1\|\leq \text{max}\{\|X\|, \|Z^*X^*P_0\overline{Q}^*Z|_{Z^*\mathcal{H}}\|\}=\|X\|$. Since $S_1|_{\mathcal{H}}=X$, $\|S_1\|=\|X\|$.\\
		
		Also we have, 
		\begingroup
		\allowdisplaybreaks
		\begin{align}\label{proj02}
			S_1^*P_1Z^* &= S_1^*Z^*P_0 & [\text{from }\eqref{proj01}]\notag\\
			& = Z^*X^*P_0\overline{Q}^*ZZ^*P_0 & [\text{since }S_1^*|_{Z^*\mathcal{H}}=Z^*X^*P_0\overline{Q}^*Z|_{Z^*\mathcal{H}}]\notag\\
			& = Z^*X^*P_0\overline{Q}^*P_0.
		\end{align}
		\endgroup
		
		\noindent\textbf{Step II.} Assume the existence of $S_n: K_{n}\rightarrow K_{n}$ such that
		\begin{itemize}
			\item[(a)] $S_{n}|_{K_{n-1}}=S_{n-1}$,
			\item[(b)] $S^*_{n}|_{Z^*K_{n-1}}=Z^*S_{n-1}^*P_{n-1}\overline{Q}^*Z|_{Z^*K_{n-1}}$,
			\item[(c)] $S_{n}^*P_{n}Z^* = Z^*S_{n-1}^*P_{n-1}\overline{Q}^*P_{n-1}$,
			\item[(d)] $\|S_{n}\|=\|S_{n-1}\|=\|X\|$.
		\end{itemize}
		\text{ }\\
		\textbf{Step III.} To find $S_{n+1}:K_{n+1}\to K_{n+1}$, it suffices to prove the following claim.\\
		
		\noindent \textbf{Claim.} For any $h_1,h_2 \in K_{n}$, $\langle Z^*S_{n}^*P_n\overline{Q}^*Z(Z^*h_1), h_2 \rangle=\langle Z^*h_1,S_nh_2 \rangle$.\\
		
		\noindent \textit{Proof of Claim.}	Suppose $h_1\in K_{n} \text{ and }h_2 \in K_{n}$. Then
		\begingroup
		\allowdisplaybreaks
		\begin{alignat*}{2}
			\langle Z^*S_{n}^*P_n\overline{Q}^*Z(Z^*h_1), h_2 \rangle 
			& = \langle S_{n}^*P_n\overline{Q}^*h_1, Zh_2 \rangle &\\
			& = \langle P_n\overline{Q}^*h_1, S_{n}Zh_2 \rangle&\\
			&= \langle \overline{Q}^*h_1, S_{n}Zh_2 \rangle&[\text{since } S_nK_n\subseteq K_n ]\\
			&= \langle \overline{Q}^*h_1, S_{n-1}Zh_2 \rangle& [\text{from } (a) \text{ as } Z(K_n)\subseteq K_{n-1}]   \\
			&= \langle h_1, \overline{Q}S_{n-1}Zh_2 \rangle&\\
			&= \langle P_{n-1}h_1, \overline{Q}S_{n-1}Zh_2\rangle& [\text{as }\overline{Q}K_{n-1}\subseteq K_{n-1}] &\\
			&= \langle \overline{Q}^*P_{n-1}h_1, S_{n-1}Zh_2\rangle&\\
			&= \langle P_{n-1}\overline{Q}^*P_{n-1}h_1, S_{n-1}Zh_2 \rangle& [\text{as }S_{n-1}K_{n-1}\subseteq S_{n-1}]\\
			&=\langle S_{n-1}^*P_{n-1}\overline{Q}^*P_{n-1}h_1, Zh_2 \rangle&\\
			&=\langle Z^*S_{n-1}^*P_{n-1}\overline{Q}^*P_{n-1}h_1, h_2 \rangle&\\
			&=\langle S_n^*P_nZ^*h_1, h_2 \rangle &[\text{from }(c)]\\
			&=\langle P_nZ^*h_1, S_nh_2 \rangle&\\
			&=\langle Z^*h_1, S_nh_2 \rangle. &[\text{as } S_nK_n\subseteq K_n]
		\end{alignat*}
		\endgroup
		This completes the proof of the claim.\\
		
		\noindent Hence by Theorem \ref{dualparrot}, there is $S_{n+1}: K_{n+1}\to K_{n+1}$ such that $S_{n+1}|_{K_{n}}=S_{n} $, $S_{n+1}^*|_{Z^*K_{n}}=Z^*S_{n}^*P_n\overline{Q}^*Z|_{Z^*K_{n}}$ and $\|S_{n+1}\| \leq\text{max}\{\|S_{n}\|,\|Z^*S_{n}^*P_n\overline{Q}^*Z|_{Z^*K_{n}}\| \} =\|S_n\|$. Since $S_{n+1}|_{K_n}=S_n$, $\|S_{n+1}\|=\|S_n\|=\|X\|$.
		Again 
		$$
		S_n^*P_{n+1}Z^* = S_{n+1}^*Z^*P_{n}=Z^*S_{n}^*P_n\overline{Q}^*ZZ^*P_n
		= Z^*S_{n}^*P_n\overline{Q}^*P_n.$$
		Therefore, by induction for all $n \in \mathbb N$ there is $S_n:K_n \rightarrow K_n$ such that 
		\begin{itemize}
			\item[(i)] $S_{n}|_{K_{n-1}}=S_{n-1}$,
			\item[(ii)] $S^*_{n}|_{Z^*K_{n-1}}=Z^*S_{n-1}^*P_{n-1}\overline{Q}^*Z|_{Z^*K_{n-1}}$,
			\item[(iii)] $S_{n}^*P_{n}Z^* = Z^*S_{n-1}^*P_{n-1}\overline{Q}^*P_{n-1}$,
			\item[(iv)] $\|S_{n}\|=\|S_{n-1}\|=\|X\|$.
		\end{itemize} 
		Define $S:\bigcup\limits_{n=0}^{\infty} \mathcal{K}_n \rightarrow \bigcup\limits_{n=0}^{\infty} \mathcal{K}_n$ such that $S|_{K_n}=S_n$. This is well defined as $S_{n}|_{K_{n-1}} = S_{n-1}$. Since $\bigcup\limits_{n = 0}^{\infty}K_n$ is dense in $\mathcal{K}$, by continuity $S$ extends to an operator $\widetilde{X}:\mathcal{K}\to \mathcal{K}$. One can easily check that $S_n^*P_n$ and $\overline{Q}^*P_n$ converges to $\widetilde{X}^*$ and $\overline{Q}^*$ respectively in the strong operator topology. Thus 
		\[
		Z\widetilde{X} = \overline{Q}\widetilde{X}Z.
		\]
		Clearly $\|\widetilde{X}\|=\|X\|$.
		Since $Z$ and $\widetilde{X}$ are extensions of $T$ and $X$ respectively, one can easily verify that $T^nX^m=P_{\mathcal{H}}Z^n\widetilde{X}^m|_{\mathcal{H}}$ and $X^nT^m=P_{\mathcal{H}}\widetilde{X}^nZ^m|_{\mathcal{H}}$ for all $n,m\geq 0$.
		This completes the proof.
	\end{proof}
	
	Removing the minimality assumption, we obtain the following general result. Note that Theorem \ref{QcoisoextDouglass} can be thought of as a generalization of Theorem \ref{Qcomextsumesh}-(iii) in the sense that here $Q$ can be any contraction instead of a unitary. 
	
	\begin{thm}\label{QcoisoextDouglass}
		Let $Q,T\in \mathcal{B}(\mathcal{H})$ be contractions and $X\in \mathcal{B}(\mathcal{H})$ be such that $TX=QXT$. Let $(Z, \mathcal{K})$ be an co-isometric extension of $T$. Then there is a contractive extension $\overline{Q}\in \mathcal{B}(\mathcal{K})$ of $Q$ and a norm preserving extension $\widetilde{{X}}\in \mathcal{B}(\mathcal{K})$ of $X$ such that $Z\widetilde{{X}}=\overline{Q}\widetilde{{X}}Z$. 
	\end{thm} 
	\begin{proof}
		Let $\mathcal{K}_0\subseteq\mathcal{K}$ be the smallest reducing subspace for $Z$ containing $\mathcal{H}$. Let $Z|_{\mathcal{K}_0}=Z_0$. Then $(Z_0,\mathcal{K}_0)$ is the minimal co-isometric extension of $T$. Due to the uniqueness of the minimal co-isometric extension of a contraction, Theorem \ref{QminicoisoextDouglass}, gives us a contractive extension $\overline{Q}_0$ of $Q $ and a norm preserving extension $Y$ of $X$ on $\mathcal{K}_0$ such that $Z_0Y=\overline{Q}_0YZ_0$.  Now define $\widetilde{X}$ on $\mathcal{K}$ as $\widetilde{X}=Y\oplus 0$ and $\overline{Q}=\overline{Q}_0\oplus Q'$ with respect to the decomposition $\mathcal{K}_0\oplus \mathcal{K}_0^{\perp}=\mathcal{K}$. Then we obtain $Z\widetilde{{X}}=\overline{Q}\widetilde{{X}}Z$ as required.   
	\end{proof}
	Similarly the next theorem generalizes Theorem \ref{Qcomlift}-(iii). 
	\begin{thm}\label{QisoliftDouglass}
		Let $Q,T\in \mathcal{B}(\mathcal{H})$ be contractions and $X\in \mathcal{B}(\mathcal{H})$ be such that $XT=TXQ$. Let $(V, \mathcal{K})$ be an isometric lift of $T$. Then there is a lift $\overline{Q}\in \mathcal{B}(\mathcal{K})$ of $Q$ and a norm preserving lift $\widetilde{{X}}\in \mathcal{B}(\mathcal{K})$ of $X$ such that $\widetilde{{X}}V=V\widetilde{{X}}\overline{Q}$. 
	\end{thm}
	\begin{proof}
		It is clear from the hypothesis  that $V^*$ is a co-isometric extension of $T^*$ and $T^*X^*=Q^*X^*T^*$. Hence due to Theorem \ref{QcoisoextDouglass}, there is an extension $\overline{Q}^*$ of $Q^*$ and norm preserving extension $Y$ of $X^*$ such that $V^*Y=\overline{Q}^*YV^* $. Taking $\widetilde{X}=Y$ proves the claim. 
	\end{proof}
	
	\vspace{0.2cm}

	\section{Ando-type dilations for $Q$-commuting contractions}\label{QSectionAndo}
	
	\vspace{0.4cm}
	
\noindent In this Section, we intend to find Ando-type dilations for a pair of $Q$-commuting contractions $(T_1,T_2)$. We consider $Q$-commuting contractions of all three kinds namely $Q_L$-commuting, $Q_M$-commuting and $Q_R$-commuting contractions respectively and construct explicitly Ando-type dilations for them.
 
	\subsection{Ando-type dilations for $Q_L$-commuting contractions}
	If $Q,T_1,T_2\in \mathcal{B}(\mathcal{H})$ are contractions such that $T_1T_2=QT_2T_1$ then Theorem \ref{QLT commutant lift} guarantees that corresponding to any contractive lift $\overline{Q}$ of $Q$ and isometric lift $V_2$ of $T_2$ there exists a norm preserving lift $Y_1 $ of $ T_1$ such that $Y_1V_2=\overline{Q}V_2Y_1 $. Since here $T_1$ is also a contraction, it is legitimate to ask whether we can find an isometric lift $V_1$ of $T_1$ instead of $Y_1$ satisfying the condition that $V_1V_2=\overline{Q}V_2V_1$. In fact, given any lift $\overline{Q}$ of $Q$,  one can ask that does there exist isometric lifts $V_1,V_2$ of $T_1,T_2$ respectively, such that $V_1V_2=\overline{Q}V_2V_1 $? But this is not possible in general. For instance suppose $Q=qI_{\mathcal{H}}$ and $\overline{Q}=qI_{\mathcal{K}}$ for some $|q|<1$ and there are isometric lifts $V_1$,$V_2$ of $T_1$, $T_2$ respectively such that $V_1V_2=qV_2V_1$. Since here $\|qV_2V_1\|=|q|<1$ and $\|V_1V_2\|=1$, it contradicts the equation $V_1V_2=qV_2V_1$. Although we can expect any of the following:
	\begin{enumerate}
		\item[I.] If $Q,T_1,T_2\in \mathcal{B}(\mathcal{H})$ are any contractions such that $T_1T_2=QT_2T_1$ then there are isometric lifts $\overline{Q},V_1,V_2$ of $Q,T_1,T_2$ respectively such that $V_1V_2=\overline{Q}V_2V_1$.
		\item[II.] If $Q,T_1,T_2\in \mathcal{B}(\mathcal{H})$ are such that $T_1,T_2,QT_2$ are contractions and $T_1T_2=QT_2T_1$ then there are isometric lifts $V_1,V_2,\overline{Q}V_{Q}$ of $T_1,T_2,QT_2$ respectively, such that $V_1V_2=\overline{Q}V_{Q}V_1$.
	\end{enumerate}
	
	Notice that in Case-II, $Q$ can be any bounded operator as long as it satisfies the condition that $QT_2$ is a contraction. For instance one can consider $Q$ such that $\|Q\|\leq 1/\|T_2\|$. Hence in general, in Case-II, we can not expect to obtain an isometric lift of $Q$ unlike in Case-I. In this section, we obtain the results mentioned in both Case-I and Case-II.
	\subsubsection{Ando-type dilation Case-I}
	Our first result guarantees the existence of $\overline{Q}$-commuting isometric lifts.
		Before going to the main theorem of this section, we note the following result. 
	\begin{thm}\label{QLcoro}
		Let $Q,T_1,T_2\in \mathcal{B}(\mathcal{H})$ be any contractions such that $T_1T_2=QT_2T_1$. Let $\mathcal{K}'$ be any Hilbert space containing $\mathcal{H}$ and $V_2', \overline{Q}_0\in  \mathcal{B}(\mathcal{K}')$ be any isometric lifts of $T_2,Q$ respectively. Then there is a Hilbert space $\mathcal{K}$ containing $\mathcal{K}'$ and isometric lifts $V_1,V_2$ of $T_1,V_2'$ such that $V_1V_2=\overline{Q}V_2V_1$, where for any complex number $q$ of modulus one, $\overline{Q}=\overline{Q}_0\oplus qI$ with respect to the decomposition $\mathcal{K}=\mathcal{K}'\oplus \mathcal{K}'^{\perp}$.  
		Moreover, $\mathcal{K}'$ is a reducing subspace for $V_2$. 
	\end{thm}
	\begin{proof}
		Let $\mathcal{K}'$ be any Hilbert space such that $\overline{Q}_0$ and $V_2'$ are isometric lifts of $Q$ and $T_2$ on $\mathcal{K}'$ respectively as given. Hence from Theorem \ref{QLT commutant lift}, there is a norm preserving lift $Y_1$ of $T_1$ on $\mathcal{K}'$ such that $Y_1V_2'=\overline{Q}_0V_2'Y_1$. Now let $V_1$ on $\mathcal{K}$ containing $\mathcal{K}'$ be the minimal isometric dilation of $Y_1$. Let $q$ be any complex number of modulus one and   $\overline{Q}=\overline{Q}_0\oplus qI$ on $\mathcal{K}'\oplus (\mathcal{K} \ominus \mathcal{K}')$ then by Theorem \ref{QminiisoliftDouglass}, there is a norm preserving lift $V_2$ of $V_2'$ such that $V_1V_2=\overline{Q}V_2V_1 $. We claim that $V_2$ is an isometry. First note that $V_2=\begin{bmatrix}
		V_2'&0\\A&B
	\end{bmatrix}$ with respect to the decomposition $\mathcal{K}'\oplus (\mathcal{K} \ominus \mathcal{K}')$. Now clearly $\overline{Q}$ is an isometry on $\mathcal{K}$ as $\overline{Q}_0$ is so. Further, using the fact that $V_2'^*V_2'=I_{\mathcal{K}'}$, we obtain, \[ 0\leq I+ A^*A\leq  V_2'^*V_2'+A^*A\leq \| V_2'^*V_2'+A^*A \|I \leq 	\|V_2^*V_2\|I\leq I .\] Thus, $A^*A=0$ which gives us that $A=0$. Hence we now have $V_2=\begin{bmatrix}
		V_2'&0\\0&B
	\end{bmatrix}$ with respect to the decomposition $\mathcal{K}'\oplus (\mathcal{K} \ominus \mathcal{K}')$.  For any $k\in \mathcal{K}'$, we obtain, \begin{alignat*}{2}
		\|V_2V_1^nk\|&=\| \overline{Q}V_2V_1^nk \| & [\text{as }\overline{Q} \text{ is an isometry.}]\\
		&=\|V_1V_2V_1^{n-1}k  \| =\|V_2V_1^{n-1}k\|. & \quad[\text{as }V_1 \text{ is an isometry.}]
	\end{alignat*} 
	Hence inductively we obtain, \begin{alignat*}{2}
		\|V_2V_1^nk\|&=\| V_2V_1k\|=\|\overline{Q}V_2V_1k\|\\
		&=\|V_1V_2k\|=\|V_2k\|&[\text{as }V_1 \text{ is an isometry.}]\\
		&=\|V_2'k\|=\|k\|.&[\text{as }V_2|\mathcal{K}'=V_2' \text{ is an isometry}.]
	\end{alignat*}
	Therefore, \begin{align*} \langle (I-V_2^*V_2)V_1^nk, V_1^nk \rangle&=\langle V_1^nk,V_1^nk \rangle- \langle V_2^*V_2V_1^nk, V_1^nk \rangle\\
		&=\langle k,k \rangle- \langle V_2V_1^nk, V_2V_1^nk \rangle\\
		&=\|k\|^2-\| V_2V_1^nk\|^2\\
		&=0. \end{align*} Thus, for all $n\in \mathbb{N}\cup \{0\}$ and $k\in \mathcal{K}'$, \[\|(I-V_2^*V_2)^{1/2}V_1^nk\|=0. \] Since $\mathcal{K}=\overline{span}\{V_1^nk:n\in \mathbb{N}\cup\{0\},\;k\in \mathcal{K}' \}$, we obtain $I-V_2^*V_2=0$ on $\mathcal{K}$. Hence $V_2$ is an isometry.
	
	\end{proof} 
	The following theorem generalizes Theorem \ref{dil-1} for arbitrary contraction $Q$.
	\begin{thm}\label{Andodil1}
		Let $Q,T_1,T_2\in \mathcal{B}(\mathcal{H})$ be any contractions such that $T_1T_2=QT_2T_1$. Then there is a Hilbert space $\mathcal{K}$ containing $\mathcal{H}$, an isometric lift $\overline{Q}\in \mathcal{B}(\mathcal{K})$ of $Q$ and there are isometric lifts $V_1,V_2$ of $T_1,T_2$ respectively, such that $V_1V_2=\overline{Q}V_2V_1 $. Moreover, if $Q$ is an isometry then $\overline{Q}$ can be chosen to be $Q\oplus qI$ with respect to decomposition $\mathcal{H}\oplus (\mathcal{K}\ominus \mathcal{H})$ of $\mathcal{K}$, for any complex number $q$ of modulus one.
	\end{thm} 
	\begin{proof}
		First assume that $Q$ is any contraction. Let $\mathcal{K}'=\mathcal{H}\oplus \mathcal{H}\oplus \dots  $. Define for any $k=(h_0,h_1,h_2,\dots )\in \mathcal{K}'$, 
		\[ \overline{Q}_0(k):=(Qh_0,D_{Q}h_0,h_1,h_2,\dots)\quad \text{ and } \quad V_2'(k):=(T_2h_0,D_{T_2}h_0,h_1,h_2,\dots  ). \]
		Clearly, $\overline{Q}_0$ and $V_2'$ are isometric lifts of $Q$ and $T_2$. Hence, we get the desired result by Theorem \ref{QLcoro}.
		
		Further, if $Q$ is in particular an isometry, let $\overline{Q}_1=Q\oplus qI$ on $\mathcal{K}'=\mathcal{H}\oplus (\mathcal{K}'\ominus \mathcal{H})$ for some complex number $q$ of modulus one. Then clearly $\overline{Q}_1$ and $V_2'$ are isometric lifts of $Q$ and $T_2$ respectively. Hence Theorem \ref{QLcoro} guarantees existence of a Hilbert space $\mathcal{K}$ containing $\mathcal{K}'$ and isometric lifts $V_1,V_2$ of $T_1,V_2'$ such that $V_1V_2=\overline{Q}V_2V_1$, where $\overline{Q}=\overline{Q}_1\oplus qI$ on $\mathcal{K}'\oplus (\mathcal{K}\ominus\mathcal{K}')$ which is same as $Q\oplus qI $ on $\mathcal{H}\oplus (\mathcal{K}\ominus\mathcal{H})$.     
	\end{proof}
Similarly, we also have a generalized form of Theorem \ref{dil-2}.
	\begin{thm}
		Let $Q,T_1,T_2\in \mathcal{B}(\mathcal{H})$ be any contractions such that $T_1T_2=T_2T_1Q$. Then there is a Hilbert space $\mathcal{K}$ containing $\mathcal{H}$, a co-isometric extensions $\overline{Q}\in \mathcal{B}(\mathcal{K})$ of $Q$ and there are co-isometric extensions $Z_1,Z_2$ of $T_1,T_2$ respectively, such that $Z_1Z_2=Z_2Z_1\overline{Q} $. 
		Moreover, if $Q$ is a co-isometry then $\overline{Q}$ can be chosen to be $Q\oplus qI$ with respect to decomposition $\mathcal{H}\oplus (\mathcal{K}\ominus \mathcal{H})$ of $\mathcal{K}$, for any complex number $q$ of modulus one.
	\end{thm} 
	\begin{proof}
		Since $T_1T_2=T_2T_1Q$ implies that $T_2^*T_1^*=Q^*T_1^*T_2^*$, by Theorem \ref{Andodil1}, there is a Hilbert space $\mathcal{K}$ containing $\mathcal{H}$ and isometric lifts $\overline{Q}_1$, $ V_1$, $V_2$ of $Q^*$, $T_1^*$ and $T_2^*$ respectively such that $V_2V_1=\overline{Q}_1V_1V_2$. Hence taking $Z_i=V_i^*$ for $i=1,2$ and $\overline{Q}=\overline{Q}_1^*$ we obtain the desired result. 
	\end{proof}
	\subsubsection{Ando-type dilation Case 2.}
	Now we look at those $Q$-commuting contractions for which $Q$ is not necessarily a contraction. We obtain Ando-type dilation for such $Q$-commuting contractions. 
	\begin{thm}\label{andoisodilL}
		Let $T_1,T_2\in \mathcal{B}(\mathcal{H})$ be any contractions and $Q\in \mathcal{B}(\mathcal{H})$ be such that $ QT_2 $ is a contraction and $T_1T_2=QT_2T_1$. Then there is a Hilbert space $\mathcal{K}$ containing $\mathcal{H}$, a bounded linear operator $\overline{Q}=Q\oplus Q'$ on $\mathcal{K}=\mathcal{H}\oplus (\mathcal{K}\ominus \mathcal{H})$ and there are isometric lifts $V_1,V_2,\overline{Q}V_{2Q}$ of $T_1,T_2,QT_2$ respectively on Hilbert space $\mathcal{K}$ such that $V_{1}V_2=\overline{Q}V_{2Q}V_1$.
	\end{thm} 
	\begin{proof}  
		The proof is based on the Ando's technique for obtaining isometric dilation of a commuting tuple of contractions. Define $\mathcal{K}=\oplus_{0}^{\infty} \mathcal{H}$, 
		\[W_{1}(h_0,h_1,\ldots ) = \left(T_1h_{0},D_{T_1}h_0,0, h_{1},\ldots \right),\]
		\[W_{2}(h_0,h_1,\ldots ) = \left(T_2h_{0},D_{T_2}h_0,0, h_{1},\ldots \right),\]
		\[W(h_0,h_1,\ldots ) = \left(T_2h_{0},D_{QT_2}h_0,0, h_{1},\ldots \right).\] Let $\overline{Q}=Q\oplus I$. Thus, \[\overline{Q}W(h_0,h_1,\ldots ) = \left(QT_2h_{0},D_{QT_2}h_0,0, h_{1},\ldots \right). \] Then $W_{1}, W_{2}, \overline{Q}W$ are isometries as $\| T_ih_0 \|^2+\| D_{T_i}h_0 \|^2=\|h_0\|^2$ for $i=1,2$ and \[ \| QT_2h_0 \|^2+\| D_{QT_2}h_0 \|^2=\|h_0\|^2. \] Moreover, $\mathcal{H}$ is an invariant subspace for $W_1^*,W_2^*,(\overline{Q}W)^*$ and \[ W_1^*|_{\mathcal{H}}=T_1^*,\,W_2^*|_{\mathcal{H}}=T_2^*,\,(\overline{Q}W)^*|_{\mathcal{H}}=(QT_2)^*.  \] Now let $$\mathcal{B}=\mathcal{H}\oplus \mathcal{H}\oplus \mathcal{H}\oplus \mathcal{H} .$$
		Hence $\mathcal{K}$ can be identified with $\mathcal{H}\oplus \mathcal{B}\oplus \mathcal{B}\oplus \cdots$ via the identification \[ (h_0,h_1,\ldots )=(h_0, (h_1,h_2,h_3,h_4),(h_5,h_6,h_7,h_8),\ldots ). \] 
		Now let, $$\mathcal{L}_1= \{(D_{T_1}T_2h,0,D_{T_2}h,0)\in \mathcal{B}: h\in \mathcal{H}\} $$ and $$\mathcal{L}_2= \{(D_{QT_2}T_1h,0,D_{T_1}h,0)\in \mathcal{B}: h\in \mathcal{H}\} .$$ Let $\mathcal{M}_i=\overline{\mathcal{L}}_i$ and $\mathcal{M}_i^{\perp}=\mathcal{B}\ominus\mathcal{M}_i$ for $i=1,2$. Define an operator $G:\mathcal{L}_1 \to \mathcal{L}_2$ as \[ G(D_{T_1}T_2h,0,D_{T_2}h,0)=(D_{QT_2}T_1h,0,D_{T_1}h,0) .\] A simple computation shows that \begin{align*}
			&\|D_{T_1}T_2h \|^2+\| D_{T_2}h \|^2\\
			=&\langle D_{T_1}T_2h, D_{T_1}T_2h \rangle + \langle D_{T_2}h,D_{T_2}h \rangle\\
			=&\langle (I-T_1^*T_1)T_2h, T_2h \rangle +\langle (I-T_2^*T_2)h, h \rangle\\
			=& \|T_2h\|^2-\|T_1T_2h\|^2+\|h\|^2-\|T_2h\|^2\\
			=&\| h \|^2-\|T_1T_2h \|^2. 
		\end{align*} Similarly, \[ \|D_{QT_2}T_1h \|^2+\| D_{T_1}h \|^2=\| h \|^2-\|QT_2T_1h \|^2  .\]  Since $T_1T_2=QT_2T_1$ the above computation proves that
		\[\|D_{T_1}T_2h \|^2+\| D_{T_2}h \|^2=\|D_{QT_2}T_1h \|^2+\| D_{T_1}h \|^2.  \] Hence $G$ defines an isometry from $\mathcal{L}_1$ onto $\mathcal{L}_2$ and extends continuously as an isometry from $\mathcal{M}_1$ onto $\mathcal{M}_2$. It  remains to show that $G$ can be extended to an isometry from the whole space $\mathcal{B}$ onto itself. For this purpose,  it suffices to prove that $\dim \mathcal{M}_1^{\perp}=\dim \mathcal{M}_2^{\perp}. $ This is clearly true when $\mathcal{H}$ and hence $\mathcal{B}$ are finite dimensional. Now  suppose $\mathcal{H}$ is infinite dimensional. Then we have, \[ \dim(\mathcal{H})=\dim\mathcal{B}\geq \dim\mathcal{M}_i^{\perp}\geq \dim \mathcal{H}. \] Here first inequality follows as $\mathcal{M}_i^{\perp}\subset \mathcal{B}$ and second inequality follows as \[\{ (0,h,0,0)\in \mathcal{B}:h\in \mathcal{H}\}\subset \mathcal{M}_i^{\perp}.\] Hence we can extend the map $G$ isometrically to a map from $\mathcal{B}$ onto $\mathcal{B}$. Therefore $G$ will be a unitary. Now define $\widetilde{G}$ on $\mathcal{K}$ as \[ \widetilde{G}(h_0,h_1,\ldots )=(h_0,G(h_1,h_2,h_3,h_4),G(h_5,h_6,h_7,h_8),\ldots). \] Then clearly $\widetilde{G}$ is a unitary with inverse given by \[ \widetilde{G}^{-1}(h_0,h_1,\ldots )=(h_0,G^{-1}(h_1,h_2,h_3,h_4),G^{-1}(h_5,h_6,h_7,h_8),\ldots) .\] Further let $V_1=\widetilde{G}W_1$, $V_2=W_2\widetilde{G}^{-1}$ and $V_{2Q}=W\widetilde{G}^{-1}$. Clearly $V_1$, $V_2$ and $\overline{Q}V_{2Q}$ are isometries. Moreover, since $\widetilde{G}|_{\mathcal{H}}=I$ and $\widetilde{G}^*|_{\mathcal{H}}=\widetilde{G}^{-1}|_{\mathcal{H}}=I$ we have, 
		\[ V_1^*|_{\mathcal{H}}=W_1^*\widetilde{G}^*|_{\mathcal{H}}=T_1^*,\, V_2^*|_{\mathcal{H}}=T_2^* \text{ and }  (\overline{Q}V_{2Q})^*|_{\mathcal{H}}=(\overline{Q}W\widetilde{G}^{-1})^*|_{\mathcal{H}}=\widetilde{G}(\overline{Q}W)^*|_{\mathcal{H}}=(QT_2)^*.\]  
		Now,
		\begin{align*}
			V_1V_2(h_0,h_1,\ldots)&=\widetilde{G}W_1W_2\widetilde{G}^{-1}(h_0,h_1,\ldots)\\
			&=\widetilde{G}W_1W_2(h_0,G^{-1}(h_1,h_2,h_3,h_4),G^{-1}(h_5,h_6,h_7,h_8),\ldots)\\
			&=\widetilde{G}W_1(T_2h_0,D_{T_2}h_0,0,G^{-1}(h_1,h_2,h_3,h_4),G^{-1}(h_5,h_6,h_7,h_8),\ldots)\\
			&=\widetilde{G}(T_1T_2h_0,D_{T_1}T_2h_0,0,D_{T_2}h_0,0,G^{-1}(h_1,h_2,h_3,h_4),\ldots)\\
			&=(T_1T_2h_0,G(D_{T_1}T_2h_0,0,D_{T_2}h_0,0),(h_1,h_2,h_3,h_4),(h_5,h_6,h_7,h_8),\ldots)
		\end{align*} 
		Similarly, 
		\begin{align*}
			\overline{Q}V_{2Q}V_1(h_0,h_1,\ldots )&=\overline{Q}W\widetilde{G}^{-1}\widetilde{G}W_1(h_0,h_1,\ldots )\\
			&=\overline{Q}WW_{1}(h_0,h_1,\ldots)\\
			&=\overline{Q}W\left(T_1h_{0},D_{T_1}h_0,0, h_{1},\ldots \right)\\
			&=(QT_2T_1h_0,D_{QT_2}T_1h_0,0,D_{T_1}h_0,0, h_{1},\ldots)
		\end{align*}
		Since $T_1T_2=QT_2T_1$ and $G(D_{T_1}T_2h_0,0,D_{T_2}h_0,0)=(D_{QT_2}T_1h_0,0,D_{T_1}h_0,0)$, we get $V_1V_2=\overline{Q}V_{2Q}V_1$ as required. 		
	\end{proof}
	\begin{note}
		If $Q\in \mathcal{B}(\mathcal{H})$ above is an isometry then $D_{QT_2}=D_{T_2}$. Therefore, in the above proof we will have $W=W_2$ and hence $V_2=V_{2Q}$. Thus for an isometry $Q$ we obtain $V_1V_2=\overline{Q}V_2V_1$.   
	\end{note}
	\begin{note}
		If we consider $\overline{Q}=Q\oplus qI$ for some non zero complex number $q$, then by taking $W(h_0,h_1,\ldots)=(T_2h_0,D_{QT_2}h_0,0,q^{-1}h_1,q^{-1}h_2,\ldots)$, we would get the same result for such $\overline{Q}$.  
	\end{note}
	We obtain a similar result in terms of co-isometric extensions. 
	\begin{thm}\label{AndocoisoextL}
		Let $T_1,T_2\in \mathcal{B}(\mathcal{H})$ be any contractions and $Q\in \mathcal{B}(\mathcal{H})$ be such that $ T_1Q $ is a contraction and $T_1T_2=T_2T_1Q$. Then there is a Hilbert space $\mathcal{K}$ containing $\mathcal{H}$, a bounded linear operator $\overline{Q}=Q\oplus Q'$ on $\mathcal{K}=\mathcal{H}\oplus (\mathcal{K}\ominus \mathcal{H})$ and there are co-isometric extensions $Z_1,Z_2,Z_{1Q}\overline{Q}$ of $T_1,T_2,T_1Q$ respectively on Hilbert space $\mathcal{K}$ such that $Z_1Z_{2}=Z_{2}Z_{1Q}\overline{Q}$.
	\end{thm}
	\begin{proof}
		
		We have $T_1T_2=T_2T_1Q$ which gives us $T_2^*T_1^*=Q^*T_1^*T_2^*$. Since $T_1Q$ is a contraction, so is $Q^*T_1^* $. Hence due to Theorem \ref{andoisodilL}, there is a Hilbert space $\mathcal{K}$, a bounded linear operator $\overline{Q}^*=Q^*\oplus Q'^*$ on $\mathcal{K}=\mathcal{H}\oplus (\mathcal{K}\ominus \mathcal{H})$ and bounded operators $V_1,V_2$, $V_Q$ on $\mathcal{K}$ such that $V_1,V_2$, $\overline{Q}^*V_Q$ are isometric lifts of $T_1^*,T_2^*$, $Q^*T_1^*$ and $V_2V_1=\overline{Q}^*V_QV_2 $. Taking $Z_1=V_1^*$, $Z_2=V_2^*$ and $Z_{2Q}=V_Q^* $ we obtain the desired result.  
		
	\end{proof}
	
	We give an Alternative proof of the above theorem for when $T_1$ is a strict contraction. 
	\begin{thm}
		Let $T_1,T_2\in \mathcal{B}(\mathcal{H})$ be any two contractions with $\|T_2\|<1$ and $Q\in \mathcal{B}(\mathcal{H})$ be such that $T_1Q$ is a contraction and $T_1T_2=T_2T_1Q$. Then there is a Hilbert space $\mathcal{K}$ containing $\mathcal{H}$, a bounded linear operator $\overline{Q}=Q\oplus Q'$ with respect to the decomposition $\mathcal{H}\oplus \mathcal{H}^{\perp} $, an operator $Z_{1Q}\in \mathcal{B}(\mathcal{H})$ and there are co-isometric extensions $Z_1,Z_2$ of $T_1,T_2$ respectively, such that $Z_{1Q}\overline{Q}$ is a co-isometric extension of $T_1Q$ and $Z_{1}Z_{2}=Z_{2}Z_{1Q}\overline{Q}$.	
		
		Infact for any non zero complex number $q$, $\overline{Q}$ can be chosen to be $Q\oplus qI$ with respect to decomposition $\mathcal{H}\oplus \mathcal{H}^{\perp}$.
	\end{thm}
	\begin{proof}
		Let $\mathcal{K}'=\mathcal{H}\oplus \mathcal{H}\oplus \cdots$ and for a nonzero complex number $q$, $\overline{Q}_0=Q\oplus qI$ with respect to the decomposition $\mathcal{K}'=\mathcal{H}\oplus \mathcal{H}^{\perp}$. Let  
		\[W_1(h_0,h_1,\ldots ) = \left(T_1h_{0}+D_{T_1^*}h_1, h_{2},\dots \right),\]
		\[ W_{1Q}(h_0,h_1,\ldots ) = \left(T_1h_{0}+q^{-1}D_{(T_1Q)^*}h_1, q^{-1}h_{2},\dots \right) .\]
		Therefore, $W_1$ and $W_{1Q}\overline{Q}_0$ are co-isometric extensions of $T_1$ and $T_1Q$ respectively. Then by intertwining extension theorem, Theorem \ref{douglasintertwining}, there is a norm preserving extension $Y_{2}$ of $T_2$ such that $W_1Y_{2}=Y_2W_{1Q}\overline{Q}_0$. Now since $T_2$ is a strict contraction, so will be $Y_2$. Hence $D_{Y_2}$ will be invertible on $\mathcal{K}'$. Thus with respect to the decomposition $\mathcal{K}=\mathcal{K}'\oplus \mathcal{K}'\oplus \cdots$, let $\overline{Q}=\overline{Q}_0\oplus qI$ and let $Z_1,Z_{1Q},Z_2$ be defined as follows.  	\begingroup
		\allowdisplaybreaks
		\begin{gather*}
				Z_1=\begin{bmatrix}
				W_1&0&0&\cdots \\
				0&D_{Y_2^*}^{-1}W_1D_{Y_2^*}&0&\cdots \\
				0&0&D_{Y_2^*}^{-1}W_1D_{Y_2^*}&\cdots \\
				\vdots&\vdots&\vdots&\ddots
			\end{bmatrix},\;
				Z_{1Q}=\begin{bmatrix}
				W_{1Q}&0&0&\cdots \\
				0&q^{-1}D_{Y_2^*}^{-1}W_1D_{Y_2^*}&0&\cdots \\
				0&0&q^{-1}D_{Y_2^*}^{-1}W_1D_{Y_2^*}&\cdots \\
				\vdots&\vdots&\vdots&\ddots
			\end{bmatrix},\\
			Z_2=\begin{bmatrix}
				Y_2&D_{Y_2^*}&0&0&\cdots\\
				0&0&I_{\mathcal{K}'}&0&\cdots\\
				0&0&0&I_{\mathcal{K}'}&\cdots\\
				\vdots&\vdots&\vdots&\vdots&\ddots
			\end{bmatrix}.
		\end{gather*}
		\endgroup

		Clearly, $Z_1,Z_{1Q}$ and $Z_{2}$ are bounded linear operators on $\mathcal{K}$ and by construction $Z_2$ is a pure co-isometry. Now we prove that $Z_1$ and $Z_{1Q}\overline{Q}$ are co-isometries. Since $W_1$ and $W_{1Q}\overline{Q}_0$ are co-isometries, it suffices to show that, 
		\begin{equation}\label{T1}
			(D_{Y_2^*}^{-1}W_1D_{Y_2^*})(D_{Y_2^*}^{-1}W_1D_{Y_2^*})^*=I_{\mathcal{K}'}.
		\end{equation}
		Using $W_1Y_2=Y_2W_{1Q}\overline{Q}_0$, $W_1W_1^*=I_{\mathcal{K}'}$ and $(W_{1Q}\overline{Q}_0)(W_{1Q}\overline{Q}_0)^*=I_{\mathcal{K}'}$ we have
		\begingroup
		\allowdisplaybreaks
		\begin{align*}
			&(D_{Y_2^*}^{-1}W_1D_{Y_2^*})(D_{Y_2^*}^{-1}W_1D_{Y_2^*})^*\\
			&=D_{Y_2^*}^{-1}W_1D_{Y_2^*}^2W_1^*D_{Y_2^*}^{-1}\\
			&=D_{Y_2^*}^{-1}W_1(I-Y_2Y_2^*)W_1^*D_{Y_2^*}^{-1}\\
			&=D_{Y_2^*}^{-1}(I_{\mathcal{K}'}-W_1Y_2Y_2^*W_1^*)D_{Y_2^*}^{-1}\\
			&=D_{Y_2^*}^{-1}(I_{\mathcal{K}'}-Y_2W_{1Q}\overline{Q}_0(W_{1Q}\overline{Q}_0)^*Y_2^*)D_{Y_2^*}^{-1}\\
			&=D_{Y_2^*}^{-1}(I_{\mathcal{K}'}-Y_2Y_2^*)D_{Y_2^*}^{-1}\\
			&=I_{\mathcal{K}'}.
		\end{align*}
		\endgroup
		Thus $Z_1, Z_2$ and $Z_{1Q}\overline{Q}$ are co-isometries. Since $W_1, W_{1Q}\overline{Q}_0$ and $Y_2$ are extensions of $T_1, T_{1}Q,T_2 $, so will be $ Z_1,Z_{1Q}\overline{Q},Z_2$ respectively. 
		By direct computation and using $W_1Y_2=Y_2W_{1Q}\overline{Q}_0$, we get $Z_1Z_2=Z_2Z_{1Q}\overline{Q}$.  
		
	\end{proof}
	\subsection{Ando-type dilations for $Q_M$-commuting contractions}
	
	Similarly, we obtain the $Q_M$ commuting analogue of Ando's dilation theorem in the form of following two cases depending on whether $Q$ is a contraction or not.\\
	
	\noindent\textbf{Case I.} If $Q,T_1,T_2\in \mathcal{B}(\mathcal{H})$ are any contractions such that $T_1T_2=T_2QT_1$ then there are isometric lifts $\overline{Q},V_1,V_2$ of $Q,T_1,T_2$ respectively such that $V_1V_2=V_2QV_1$.
	
	\noindent\textbf{Case II.} If $Q,T_1,T_2\in \mathcal{B}(\mathcal{H})$ such that $T_1,T_2,QT_1$ are contractions and $T_1T_2=T_2QT_1$ then there are isometric lifts $V_1,V_2,\overline{Q}V_{1Q}$ of $T_1,T_2,QT_1$ respectively, such that $V_1V_2=V_2\overline{Q}V_{1Q}$.	\\
	
	Notice that in Case II, $Q$ can be any bounded operator as long as it satisfies the condition that $TQ$ is a contraction. For instance one can consider $Q$ such that $\|Q\|\leq 1/\|T\|$. Hence in general, in Case II we can not expect to obtain an isometric lift of $Q$ unlike in Case I. In this section, we obtain the results mentioned in both Case I and Case II.
	\subsubsection{Ando type dilation Case I.}
	First we consider Case I.
	\begin{thm}\label{QMcoro}
		Let $Q,T_1,T_2\in \mathcal{B}(\mathcal{H})$ be any contractions such that $T_1T_2=T_2QT_1$. Let $\mathcal{K}'$ be any Hilbert space containing $\mathcal{H}$ and $\overline{Q}_0, V_2' \in  \mathcal{B}(\mathcal{K}')$ be any isometric lifts of $Q,T_2$ respectively. Then there is a Hilbert space $\mathcal{K}$ containing $\mathcal{K}'$ and isometric lifts $V_1,V_2$ of $T_1,V_2'$ such that $V_1V_2=V_2\overline{Q}V_1$, where for any complex number $q$ of modulus one, $\overline{Q}=\overline{Q}_0\oplus qI$ with respect to the decomposition $\mathcal{K}=\mathcal{K}'\oplus \mathcal{K}'^{\perp}$.  
		Moreover, $\mathcal{K}'$ is a reducing subspace for $V_2$. 
	\end{thm}
	\begin{proof}
		Let $\mathcal{K}'$ be any Hilbert space such that $\overline{Q}_0$ and $V_2'$ are isometric lifts of $Q$ and $T_2$ on $\mathcal{K}'$ respectively. Hence from Theorem \ref{QMT commutant lift}, there is a norm preserving lift $Y_1$ of $T_1$ on $\mathcal{K}'$ such that $Y_1V_2'=V_2'\overline{Q}_0Y_1$. Now let $V_1$ on $\mathcal{K}=\mathcal{K}'\oplus \mathcal{D}_{Y_1}\oplus \mathcal{D}_{Y_1}\oplus \dots$ be the Sch$\ddot{a}$ffer's minimal isometric dilation of $Y_1$. Let $\overline{Q}=\overline{Q}_0\oplus I$ on $\mathcal{K}'\oplus (\mathcal{K} \ominus \mathcal{K}')$. Then by the intertwining lifting theorem, there is a norm preserving lift $V_2$ of $V_2'$ such that $V_1V_2=V_2\overline{Q}V_1 $. We claim that $V_2$ is an isometry. First note that $V_2=\begin{bmatrix}
			V_2'&0\\A&B
		\end{bmatrix}$ with respect to the decomposition $\mathcal{K}'\oplus (\mathcal{K} \ominus \mathcal{K}')$. Now clearly $\overline{Q}$ is an isometry on $\mathcal{K}$ as $\overline{Q}_0$ is so. Further, using the fact that $V_2'^*V_2'=I_{\mathcal{K}'}$, we obtain, \[ 0\leq I+ A^*A\leq  V_2'^*V_2'+A^*A\leq \| V_2'^*V_2'+A^*A \|I \leq 	\|V_2^*V_2\|I\leq I .\] Thus, $A^*A=0$ which gives us that $A=0$. Hence we now have $V_2=\begin{bmatrix}
			V_2'&0\\0&B
		\end{bmatrix}$ with respect to the decomposition $\mathcal{K}'\oplus (\mathcal{K} \ominus \mathcal{K}')$. Since $\overline{Q}_0$ is an isometry, \[ I-(\overline{Q}_0Y_1)^*(\overline{Q}_0Y_1)=I-Y_1^*Y_1. \] Hence $D_{Y_1}=D_{\overline{Q}_0Y_1}$ and $\mathcal{D}_{Y_1}=\mathcal{D}_{\overline{Q}_0Y_1}$. Further, by definitions of $\overline{Q}$ and $V_1 $ it is clear that 
		\[ \overline{Q}V_1=\begin{bmatrix}
			\overline{Q}_0Y_1&0&0&\dots \\
			D_{Y_1}&0&0&\dots \\
			0&I_{\mathcal{D}_{Y_1}}&0&\dots\\
			0&0&I_{\mathcal{D}_{Y_1}}&\dots
		\end{bmatrix}=\begin{bmatrix}
			\overline{Q}_0Y_1&0&0&\dots \\
			D_{\overline{Q}_0Y_1}&0&0&\dots \\
			0&I_{\mathcal{D}_{\overline{Q}_0Y_1}}&0&\dots\\
			0&0&I_{\mathcal{D}_{\overline{Q}_0Y_1}}&\dots
		\end{bmatrix} \] on $\mathcal{K}=\mathcal{K}'\oplus \mathcal{D}_{Y_1}\oplus \mathcal{D}_{Y_1}\oplus \dots=\mathcal{K}'\oplus \mathcal{D}_{\overline{Q}_0Y_1}\oplus \mathcal{D}_{\overline{Q}_0Y_1}\oplus \dots$. Hence $\overline{Q}V_1$ is the minimal isomeric dilation of $\overline{Q}_0Y_1 $ on $\mathcal{K}$. Therefore, \[ \mathcal{K}=\overline{span}\{(\overline{Q}V_1)^nk:n\in \mathbb{N}\cup\{0\},\;k\in \mathcal{K}' \} .\] Now observe that for any $k\in \mathcal{K}'$, we obtain, \begin{alignat*}{2}
			\|V_2(\overline{Q}V_1)^nk\|&=\|V_2\overline{Q}V_1(\overline{Q}V_1)^{n-1}k \| & \\
			&=\|V_1V_2(\overline{Q}V_1)^{n-1}k \|& [\text{From }V_2\overline{Q}V_1=V_1V_2 ]\\
			&=\|V_2(\overline{Q}V_1)^{n-1}k  \|. & \quad[\text{As }V_1 \text{ is an isometry}]
		\end{alignat*} 
		Hence inductively we obtain \begin{alignat*}{2}
			\|V_2(\overline{Q}V_1)^nk\|&=\| V_2\overline{Q}V_1k\|\\
			&=\|V_1V_2k\|=\|V_2k\|&[\text{As }V_1 \text{ is an isometry}]\\
			&=\|V_2'k\|=\|k\|.&[\text{As }V_2|\mathcal{K}'=V_2' \text{ is an isometry on }\mathcal{K}']
		\end{alignat*}
		Therefore, \begin{align*} \langle (I-V_2^*V_2)(\overline{Q}V_1)^nk, (\overline{Q}V_1)^nk \rangle&=\langle (\overline{Q}V_1)^nk,(\overline{Q}V_1)^nk \rangle- \langle V_2^*V_2(\overline{Q}V_1)^nk, (\overline{Q}V_1)^nk \rangle\\
			&=\langle k,k \rangle- \langle V_2(\overline{Q}V_1)^nk, V_2(\overline{Q}V_1)^nk \rangle\\
			&=\|k\|^2-\| V_2(\overline{Q}V_1)^nk\|^2\\
			&=0. \end{align*} Thus, for all $n\in \mathbb{N}\cup \{0\}$ and $k\in \mathcal{K}'$, \[\|(I-V_2^*V_2)^{1/2}(\overline{Q}V_1)^nk\|=0. \] Therefore, $I-V_2^*V_2=0$ on $\mathcal{K}$. Hence $V_2$ is an isometry.
	\end{proof}
	\begin{thm}\label{AndodilM} 
		Let $Q,T_1,T_2\in \mathcal{B}(\mathcal{H})$ be any contractions such that $T_1T_2=T_2QT_1$. Then there is a Hilbert space $\mathcal{K}$ containing $\mathcal{H}$, an isometric lift $\overline{Q}\in \mathcal{B}(\mathcal{K})$ of $Q$ and there are isometric lifts $V_1,V_2$ of $T_1,T_2$ respectively, such that $V_1V_2=V_2\overline{Q}V_1$. Moreover, if $Q$ is an isometry then $\overline{Q}$ can be chosen to be $Q\oplus qI$ with respect to decomposition $\mathcal{H}\oplus (\mathcal{K}\ominus \mathcal{H})$ of $\mathcal{K}$, for any complex number $q$ of modulus one.
	\end{thm} 
	\begin{proof}
First assume that $Q$ is any contraction. Let $\mathcal{K}'=\mathcal{H}\oplus \mathcal{H}\oplus \dots  $. Define for any $k=(h_0,h_1,h_2,\dots )\in \mathcal{K}'$, 
\[ \overline{Q}_0(k):=(Qh_0,D_{Q}h_0,h_1,h_2,\dots)\quad \text{ and } \quad V_2'(k):=(T_2h_0,D_{T_2}h_0,h_1,h_2,\dots  ). \]
Clearly, $\overline{Q}_0$ and $V_2'$ are isometric lifts of $Q$ and $T_2$. Hence, we get the desired result by Theorem \ref{QLcoro}.

Further, if $Q$ is in particular an isometry, let $\overline{Q}_1=Q\oplus qI$ on $\mathcal{K}'=\mathcal{H}\oplus (\mathcal{K}'\ominus \mathcal{H})$ for some complex number $q$ of modulus one. Then clearly $\overline{Q}_1$ and $V_2'$ are isometric lifts of $Q$ and $T_2$ respectively. Hence Theorem \ref{QMcoro} guarantees existence of a Hilbert space $\mathcal{K}$ containing $\mathcal{K}'$ and isometric lifts $V_1,V_2$ of $T_1,V_2'$ such that $V_1V_2=V_2\overline{Q}V_1$, where $\overline{Q}=\overline{Q}_1\oplus qI$ on $\mathcal{K}'\oplus (\mathcal{K}\ominus\mathcal{K}')$ which is same as $Q\oplus qI $ on $\mathcal{H}\oplus (\mathcal{K}\ominus\mathcal{H})$.     		
	\end{proof}
	
	As observed before, $(V, \mathcal{K})$ is an isometric lift of $(T,\mathcal{H})$, if and only if $V^*$ is a co-isometric extension of $T^*$. Hence we can restate Theorem \ref{AndodilM} as follows.
	\begin{thm}\label{AndodilMcoiso} 
		Let $Q,T_1,T_2\in \mathcal{B}(\mathcal{H})$ be any contractions such that $T_1T_2=T_2QT_1$. Then there is a Hilbert space $\mathcal{K}$ containing $\mathcal{H}$, a co-isometric extension $\overline{Q}\in \mathcal{B}(\mathcal{K})$ of $Q$ and there are co-isometric extensions $Z_1,Z_2$ of $T_1,T_2$ respectively, such that $Z_1Z_2=Z_2\overline{Q}Z_1$. 
	\end{thm} 
	
	We consider a special case of it for when $T_2$ is a strict contraction and obtain an alternative proof of Theorem \ref{AndodilMcoiso}.
	\begin{thm}
		Let $Q,T_1,T_2\in \mathcal{B}(\mathcal{H})$ be any two contractions with $\|T_1\|<1$ and $T_1T_2=T_2QT_1$. Then there is a Hilbert space $\mathcal{K}$ containing $\mathcal{H}$, a co-isometric extension $\overline{Q}\in \mathcal{B}(\mathcal{K})$ and co-isometric extensions $Z_1,Z_2$ of $T_1,T_2$ respectively, such that $Z_{1}Z_{2}=Z_{2}\overline{Q}Z_1$.
	\end{thm}
	\begin{proof}
		Let $\mathcal{K}'=\mathcal{H}\oplus \mathcal{H}\oplus \cdots$,
		\[\overline{Q}_0(h_0,h_1,\dots )=(Qh_0+D_{Q^*}h_1, h_2,h_3,\dots),\] and 
		\[W_2(h_0,h_1,\ldots ) = \left(T_2h_{0}+D_{T_2^*}h_1, h_{2},\dots \right).\]
		Hence $\overline{Q}_0$ is a co-isometric extension of $Q$ and $W_2$ is a co-isometric extension of $T_2$. Therefore, $W_2\overline{Q}_0$ is a co-isometric extension of $T_2Q$. Then by intertwining lifting theorem there is a norm preserving extension $Y_{1}$ of $T_1$ such that $Y_{1}W_2=W_2\overline{Q}_0Y_1$. Since $\|Y_1\|=\|T_1\|$ and $T_1$ is a strict contraction, so is $Y_1$. Hence $D_{Y_1}$ is invertible on $\mathcal{K}'$. Thus with respect to the decomposition $\mathcal{K}=\mathcal{K}'\oplus \mathcal{K}'\oplus \cdots$, let $\overline{Q}=\overline{Q}_0\oplus I$ and $Z_1,Z_2\in \mathcal{B}(\mathcal{K})$ be defined as follows.  	\begingroup
		\allowdisplaybreaks
		\begin{gather*}
			Z_1=\begin{bmatrix}
			Y_1&D_{Y_1^*}&0&0&\cdots\\
			0&0&I_{\mathcal{K}'}&0&\cdots\\
			0&0&0&I_{\mathcal{K}'}&\cdots\\
			\vdots&\vdots&\vdots&\vdots&\ddots
		\end{bmatrix},	Z_2=\begin{bmatrix}
				W_2&0&0&\cdots \\
				0&D_{Y_1^*}^{-1}W_2\overline{Q}_0D_{Y_1^*}&0&\cdots \\
				0&0&D_{Y_1^*}^{-1}W_2\overline{Q}_0D_{Y_1^*}&\cdots \\
				\vdots&\vdots&\vdots&\ddots
			\end{bmatrix}.		\end{gather*}
		\endgroup
		Clearly, $V_1,Z_2$ are bounded linear operators on $\mathcal{K}$ and by construction $Z_1$ is a pure co-isometry. Now we prove that $Z_2$ is a co-isometry. Since $W_2$ is a co-isometry, it suffices to show that, 
		\begin{equation}\label{T1}
			(D_{Y_1^*}^{-1}W_2\overline{Q}_0D_{Y_1^*})(D_{Y_1^*}^{-1}W_2\overline{Q}_0D_{Y_1^*})^*=I_{\mathcal{K}'}.
		\end{equation}
		Using $W_{2}\overline{Q}_0Y_1=Y_1W_2$ and $W_2W_2^*=I_{\mathcal{K}'}$ we have
		\begingroup
		\allowdisplaybreaks
		\begin{align*}
			&(D_{Y_1^*}^{-1}W_2\overline{Q}_0D_{Y_1^*})(D_{Y_1^*}^{-1}W_2\overline{Q}_0D_{Y_1^*})^*\\
			&=D_{Y_1^*}^{-1}W_2\overline{Q}_0D_{Y_1^*}^2\overline{Q}_0^*W_2^*D_{Y_1^*}^{-1}\\
			&=D_{Y_1^*}^{-1}W_2\overline{Q}_0(I-Y_1Y_1^*)\overline{Q}_0^*W_2^*D_{Y_1^*}^{-1}\\
			&=D_{Y_1^*}^{-1}(I_{\mathcal{K}'}-W_2\overline{Q}_0Y_1Y_1^*\overline{Q}_0^*W_2^*)D_{Y_1^*}^{-1}\\
			&=D_{Y_1^*}^{-1}(I_{\mathcal{K}'}-Y_1W_{2}W_{2}^*Y_1^*)D_{Y_1^*}^{-1}\\
			&=D_{Y_1^*}^{-1}(I_{\widetilde{\mathcal{K}}}-Y_1Y_1^*)D_{Y_1^*}^{-1}\\
			&=I_{\mathcal{K}'}.
		\end{align*}
		\endgroup
		Thus, $Z_1, Z_2$ are co-isometries. Further it is clear from the construction that $Z_1$ extends $Y_1$ which is an extension of $T_1$ and $Z_2$ extends $W_2$ which is an extension of $T_2$. Hence $Z_1,Z_2$ extend $T_1,T_2$ respectively. 
		By direct computation and using $Y_{1}W_2=W_2\overline{Q}_0Y_1$, we get $Z_1Z_{2}=Z_2\overline{Q}Z_1$.  
	\end{proof}
	\subsubsection{Ando type dilation Case II}
	Now we consider Case II in which $Q$ is not necessarily a contraction.
	\begin{thm}\label{andodilII}
		Let $T_1,T_2\in \mathcal{B}(\mathcal{H})$ be any contractions and $Q\in \mathcal{B}(\mathcal{H})$ be such that $ QT_1 $ is a contraction and $T_1T_2=T_2QT_1$. Then there are isometric lifts $V_1,V_2$ of $T_1,T_2$, a lift $\overline{Q}$ of $Q$, a bounded operator $V_{1Q}$ such that $\overline{Q}V_{1Q}$ is an isometric lift of $QT_1$ on Hilbert space $\mathcal{K}$ containing $\mathcal{H}$ and $V_{1}V_2=V_2\overline{Q}V_{1Q}$.
	\end{thm}
	\begin{proof}
		Define $\mathcal{K}=\oplus_{0}^{\infty} \mathcal{H}$, 
		\[W_{1}(h_0,h_1,\ldots ) = \left(T_1h_{0},D_{T_1}h_0,0, h_{1},\ldots \right),\]
		\[W_{2}(h_0,h_1,\ldots ) = \left(T_2h_{0},D_{T_2}h_0,0, h_{1},\ldots \right),\]
		\[W(h_0,h_1,\ldots ) = \left(T_1h_{0},D_{QT_1}h_0,0, h_{1},\ldots \right).\] Let $\overline{Q}=Q\oplus I$ with respect to the decomposition $\mathcal{K}=\mathcal{H}\oplus \mathcal{H}^{\perp}$. Thus note that \[\overline{Q}W(h_0,h_1,\ldots ) = \left(QT_1h_{0},D_{QT_1}h_0,0, h_{1},\ldots \right). \] Then $W_{1}, W_{2}, \overline{Q}W$ are isometries as $\| T_ih_0 \|^2+\| D_{T_i}h_0 \|^2=\|h_0\|^2$ for $i=1,2$ and \[ \| QT_1h_0 \|^2+\| D_{QT_1}h_0 \|^2=\|h_0\|^2. \] Moreover, $\mathcal{H}$ is an invariant subspace for $W_1^*,W_2^*,(\overline{Q}W)^*$ and \[ W_1^*|_{\mathcal{H}}=T_1^*,\,W_2^*|_{\mathcal{H}}=T_2^*,\,(\overline{Q}W)^*|_{\mathcal{H}}=(QT_1)^*.  \] Now let $$\mathcal{B}=\mathcal{H}\oplus \mathcal{H}\oplus \mathcal{H}\oplus \mathcal{H} .$$
		Hence $\mathcal{K}$ can be identified with $\mathcal{H}\oplus \mathcal{B}\oplus \mathcal{B}\oplus \cdots$ via the identification \[ (h_0,h_1,\ldots )=(h_0, (h_1,h_2,h_3,h_4),(h_5,h_6,h_7,h_8),\ldots ). \] 
		Now let $$\mathcal{L}_1= \{(D_{T_2}QT_1h,0,D_{QT_1}h,0)\in \mathcal{B}: h\in \mathcal{H}\} $$ and $$\mathcal{L}_2= \{(D_{T_1}T_2h,0,D_{T_2}h,0)\in  \mathcal{B}: h\in \mathcal{H}\} .$$ Let $\mathcal{M}_i=\overline{\mathcal{L}}_i$ and $\mathcal{M}_i^{\perp}=\mathcal{B}\ominus\mathcal{M}_i$ for $i=1,2$. Define an operator $G:\mathcal{L}_1 \to \mathcal{L}_2$ as \[ G(D_{T_2}QT_1h_0,0,D_{QT_1}h_0,0)=(D_{T_1}T_2h_0,0,D_{T_2}h_0,0) .\] A simple computation shows that \begin{align*}
			&\|D_{T_2}QT_1h \|^2+\| D_{QT_1}h \|^2\\
			=&\langle D_{T_2}QT_1h, D_{T_2}QT_1h \rangle + \langle D_{QT_1}h,D_{QT_1}h \rangle\\
			=&\langle (I-T_2^*T_2)QT_1h, QT_1h \rangle +\langle (I-(QT_1)^*QT_1)h, h \rangle\\
			=& \|QT_1h\|^2-\|T_2QT_1h\|^2+\|h\|^2-\|QT_1h\|^2\\
			=&\| h \|^2-\|T_2QT_1h \|^2. 
		\end{align*} Similarly, \[ \|D_{T_1}T_2h \|^2+\| D_{T_2}h \|^2=\| h \|^2-\|T_1T_2h \|^2  .\]  Since $T_1T_2=T_2QT_1$ the above computation proves that
		\[\|D_{T_1}T_2h \|^2+\| D_{T_2}h \|^2=\|D_{T_2}QT_1h \|^2+\| D_{QT_1}h \|^2.  \] Hence $G$ defines an isometry from $\mathcal{L}_1$ onto $\mathcal{L}_2$ and extends continuously as an isometry from $\mathcal{M}_1$ onto $\mathcal{M}_2$. It  remains to show that $G$ can be extended to an isometry from the whole space $\mathcal{B}$ onto itself. For this purpose,  it suffices to prove that $\dim \mathcal{M}_1^{\perp}=\dim \mathcal{M}_2^{\perp}. $ This is clearly true when $\mathcal{H}$ and hence $\mathcal{B}$ are finite dimensional. Now  suppose $\mathcal{H}$ is infinite dimensional. Then we have, \[ \dim(\mathcal{H})=\dim\mathcal{B}\geq \dim\mathcal{M}_i^{\perp}\geq \dim \mathcal{H}. \] Here first inequality follows as $\mathcal{M}_i^{\perp}\subset \mathcal{B}$ and second inequality follows as \[\{ (0,h,0,0)\in \mathcal{B}:h\in \mathcal{H}\}\subset \mathcal{M}_i^{\perp}.\] Hence we can extend the map $G$ isometrically to a map from $\mathcal{B}$ onto $\mathcal{B}$. Therefore $G$ will be a unitary. Now define $\widetilde{G}$ on $\mathcal{K}$ as \[ \widetilde{G}(h_0,h_1,\ldots )=(h_0,G(h_1,h_2,h_3,h_4),G(h_5,h_6,h_7,h_8),\ldots). \] Then clearly $\widetilde{G}$ is a unitary with inverse given by, \[ \widetilde{G}^{-1}(h_0,h_1,\ldots )=(h_0,G^{-1}(h_1,h_2,h_3,h_4),G^{-1}(h_5,h_6,h_7,h_8),\ldots) .\] Further let $V_1=W_1\widetilde{G}^{-1}$, $V_2=\widetilde{G}W_2$ and $V_{1Q}=W\widetilde{G}^{-1}$. Clearly $V_1$, $V_2$ and $\overline{Q}V_{1Q}$ are isometries. Moreover, since $\widetilde{G}|_{\mathcal{H}}=I$ and $\widetilde{G}^*|_{\mathcal{H}}=\widetilde{G}^{-1}|_{\mathcal{H}}=I$ we have, 
		\[ V_1^*|_{\mathcal{H}}=\widetilde{G}W_1^*|_{\mathcal{H}}=T_1^*,\, V_2^*|_{\mathcal{H}}=T_2^* \text{ and }  (\overline{Q}V_{1Q})^*|_{\mathcal{H}}=(\overline{Q}W\widetilde{G}^{-1})^*|_{\mathcal{H}}=\widetilde{G}(\overline{Q}W)^*|_{\mathcal{H}}=(QT_1)^*.\]    Now 
		\begin{align*}
			V_2\overline{Q}V_{1Q}(h_0,h_1,\ldots)&=\widetilde{G}W_2\overline{Q}W\widetilde{G}^{-1}(h_0,h_1,\ldots)\\
			&=\widetilde{G}W_2\overline{Q}W(h_0,G^{-1}(h_1,h_2,h_3,h_4),G^{-1}(h_5,h_6,h_7,h_8),\ldots)\\
			&=\widetilde{G}W_2(QT_1h_0,D_{QT_1}h_0,0,G^{-1}(h_1,h_2,h_3,h_4),G^{-1}(h_5,h_6,h_7,h_8),\ldots)\\
			&=\widetilde{G}(T_2QT_1h_0,D_{T_2}QT_1h_0,0,D_{QT_1}h_0,0,G^{-1}(h_1,h_2,h_3,h_4),\ldots)\\
			&=(T_2QT_1h_0,G(D_{T_2}QT_1h_0,0,D_{QT_1}h_0,0),(h_1,h_2,h_3,h_4),(h_5,h_6,h_7,h_8),\ldots).
		\end{align*} 
		Similarly, 
		\begin{align*}
			V_{1}V_2(h_0,h_1,\ldots )&=W_1\widetilde{G}^{-1}\widetilde{G}W_2(h_0,h_1,\ldots )\\
			&=W_1W_{2}(h_0,h_1,\ldots)\\
			&=W_1\left(T_2h_{0},D_{T_2}h_0,0, h_{1},\ldots \right)\\
			&=(T_1T_2h_0,D_{T_1}T_2h_0,0,D_{T_2}h_0,0, h_{1},\ldots).
		\end{align*}
		Since $T_1T_2=T_2QT_1$ and $G(D_{T_2}QT_1h_0,0,D_{QT_1}h_0,0)=(D_{T_1}T_2h_0,0,D_{T_2}h_0,0)$, we get $V_1V_2=V_2\overline{Q}V_{1Q}$ as required. 		
	\end{proof}
	The result corresponding to co-isometric extensions is as bellow.
	\begin{thm}
		Let $T_1,T_2\in \mathcal{B}(\mathcal{H})$ be any contractions and $Q\in \mathcal{B}(\mathcal{H})$ be such that $ T_2Q $ is a contraction and $T_1T_2=T_2QT_1$. Then there are co-isometric extensions $Z_1,Z_2$ of $T_1,T_2$, an extension $\overline{Q}$ of $Q$, a bounded operator $Z_{2Q}$ such that $Z_{2Q}\overline{Q}$ is a co-isometric extension of $T_2Q$ on Hilbert space $\mathcal{K}$ containing $\mathcal{H}$ and $Z_{1}Z_2=Z_{2Q}\overline{Q}Z_{1}$.
	\end{thm}
	\begin{proof}
		We have $T_1T_2=T_2QT_1$ which gives us $T_1^*Q^*T_2^*=T_2^*T_1^*$. Since $T_2Q$ is a contraction, so is $Q^*T_2^* $. Hence due to Theorem \ref{andodilII}, there is a Hilbert space $\mathcal{K}$, a lift $\overline{Q}^*$ of $Q^*$, bounded operators $Y_1,Y_2$, $Y_{2Q}$ on $\mathcal{K}$ such that $Y_1,Y_2$, $\overline{Q}^*Y_{2Q}$ are isometric lifts of $T_1^*,T_2^*$, $Q^*T_2^*$ and $Y_1\overline{Q}^*Y_{2Q}=Y_2Y_1 $. Taking $Z_1=Y_1^*$, $Z_2=Y_2^*$ and $Z_{2Q}=Y_{2Q}^* $ we obtain the desired result.  	
	\end{proof}
	\subsection{Ando-type dilations for $Q_R$-commuting contractions}
	Similarly, we can obtain the $Q_R$ commuting analogue of Ando's dilation theorem in the form of following two cases.\\
	
	\noindent\textbf{Case I.} If $Q,T_1,T_2\in \mathcal{B}(\mathcal{H})$ are any contractions such that $T_1T_2=T_2T_1Q$ then there are isometric lifts $\overline{Q},V_1,V_2$ of $Q,T_1,T_2$ respectively such that $V_1V_2=V_2V_1\overline{Q}$.
	
	\noindent\textbf{Case II.} If $Q,T_1,T_2\in \mathcal{B}(\mathcal{H})$ such that $T_1,T_2,T_1Q$ are contractions and $T_1T_2=T_2T_1Q$ then there are isometric liftings $V_1,V_2,V_{1Q}\overline{Q}$ of $T_1,T_2,T_1Q$ respectively, such that $V_1V_2=V_2V_{1Q}\overline{Q}$.	\\
	
	Notice that in Case II, $Q$ can be any bounded operator as long as it satisfies the condition that $TQ$ is a contraction. For instance one can consider $Q$ such that $\|Q\|\leq 1/\|T\|$. Hence in general, in Case II we can not expect to obtain an isometric lift of $Q$ unlike in Case II. In this section, we obtain the results mentioned in both Case I and Case II.
	\subsubsection{Ando-type dilation Case I.}
	We will prove the case I under the additional assumption that $\|T_1\|<1$.  
		\begin{thm}\label{QLcoisopurecoro}
		Let $T_1,T_2,Q\in \mathcal{B}(\mathcal{H})$ be any two contractions with $\|T_1\|<1$ and $T_1T_2=QT_2T_1$.Let $\mathcal{K}'$ be any Hilbert space containing $\mathcal{H}$ and $Z_2', \overline{Q}_0\in  \mathcal{B}(\mathcal{K}')$ be any co-isometric extensions of $T_2,Q$ respectively. Then there is a Hilbert space $\mathcal{K}$ containing $\mathcal{K}'$ and co-isometric extensions $Z_1,Z_2$ of $T_1,Z_2'$ such that $Z_1Z_2=\overline{Q}Z_2Z_1$, where for any complex number $q$ of modulus one, $\overline{Q}=\overline{Q}_0\oplus qI$ with respect to the decomposition $\mathcal{K}=\mathcal{K}'\oplus \mathcal{K}'^{\perp}$.  
		Moreover, $\mathcal{K}'$ is a reducing subspace for $Z_2$. 
	\end{thm}
	\begin{proof}
		Let $\mathcal{K}'$, $\overline{Q}_0$ and $Z_2'$ be as in the hypothesis. Then by Theorem \ref{Q_Rcoiso} there is a norm preserving extension $Y_{1}$ of $T_1$ such that $Y_1Z_2'=\overline{Q}_0Z_2'Y_1$. Since $\|Y_1\|=\|T_1\|$ and $T_1$ is a strict contraction, so will be $Y_1$. Hence $D_{Y_1}$ will be invertible on $\mathcal{K}'$. Thus with respect to the decomposition $\mathcal{K}=\mathcal{K}'\oplus \mathcal{K}'\oplus \cdots$, let $\overline{Q}=\overline{Q}_0\oplus I$ and $Z_1,Z_2$ be defined as follows.  	\begingroup
		\allowdisplaybreaks
		\begin{gather*}
			Z_1=\begin{bmatrix}
				Y_1&D_{Y_1^*}&0&0&\cdots\\
				0&0&I_{\mathcal{K}'}&0&\cdots\\
				0&0&0&I_{\mathcal{K}'}&\cdots\\
				\vdots&\vdots&\vdots&\vdots&\ddots
			\end{bmatrix},\;		Z_2=\begin{bmatrix}
				Z_2'&0&0&\cdots \\
				0&D_{Y_1^*}^{-1}\overline{Q}_0Z_2'D_{Y_1^*}&0&\cdots \\
				0&0&D_{Y_1^*}^{-1}\overline{Q}_0Z_2'D_{Y_1^*}&\cdots \\
				\vdots&\vdots&\vdots&\ddots
			\end{bmatrix}.	\end{gather*}
		\endgroup
		
		Clearly, $Z_1,Z_2$ are bounded linear operators on $\mathcal{K}$ and by construction $Z_1$ is a pure co-isometry. Now we prove that $Z_2$ is a co-isometry. Since $Z_2'$ is a co-isometry, it suffices to show that, 
		\begin{equation}\label{T1}
			(D_{Y_1^*}^{-1}\overline{Q}_0Z_2'D_{Y_1^*})(D_{Y_1^*}^{-1}\overline{Q}_0Z_2'D_{Y_1^*})^*=I_{\mathcal{K}'}.
		\end{equation}
		Using $Y_1Z_2'=\overline{Q}_0Z_2'Y_1$ and $Z_2'Z_2'^*=I_{\mathcal{K}'}$ we have
		\begingroup
		\allowdisplaybreaks
		\begin{align*}
			&(D_{Y_1^*}^{-1}\overline{Q}_0Z_2'D_{Y_1^*})(D_{Y_1^*}^{-1}\overline{Q}_0Z_2'D_{Y_1^*})^*\\
			&=D_{Y_1^*}^{-1}\overline{Q}_0Z_2'D_{Y_1^*}^2Z_2'^*\overline{Q}_0^*D_{Y_1^*}^{-1}\\
			&=D_{Y_1^*}^{-1}\overline{Q}_0Z_2'(I-Y_1Y_1^*)Z_2'^*\overline{Q}_0^*D_{Y_1^*}^{-1}\\
			&=D_{Y_1^*}^{-1}(I_{\mathcal{K}'}-\overline{Q}_0Z_2'Y_1Y_1^*Z_2'^*\overline{Q}_0^*)D_{Y_1^*}^{-1}\\
			&=D_{Y_1^*}^{-1}(I_{\mathcal{K}'}-Y_1Z_2'Z_2'^*Y_1^*)D_{Y_1^*}^{-1}\\
			&=D_{Y_1^*}^{-1}(I_{\widetilde{\mathcal{K}}}-Y_1Y_1^*)D_{Y_1^*}^{-1}\\
			&=I_{\mathcal{K}'}.
		\end{align*}
		\endgroup
		Thus, $Z_1, Z_2$ are co-isometries. Further it is clear from the construction that $Z_1$ extends $Y_1$ which is an extension of $T_1$ and $Z_2$ extends $Z_2'$ which is an extension of $T_2$. Hence $Z_1,Z_2$ extend $T_1,T_2$ respectively. 
		By direct computation and using $Y_1Z_2'=\overline{Q}_0Z_2'Y_1$, we get $Z_1Z_{2}=\overline{Q}Z_2Z_1$.  
	\end{proof}
Hence clearly we can rewrite the above result as follows.
	\begin{thm}\label{QRisopurecoro}
	Let $Q,T_1,T_2\in \mathcal{B}(\mathcal{H})$ be any two contractions with $\|T_2\|<1$ and $T_1T_2=T_2T_1Q$. Let $\mathcal{K}'$ be any Hilbert space containing $\mathcal{H}$ and $W_1, \overline{Q}_0\in  \mathcal{B}(\mathcal{K}')$ be any isometric lifts of $T_1,Q$ respectively. Then there is a Hilbert space $\mathcal{K}$ containing $\mathcal{K}'$ and isometric lifts $V_1,V_2$ of $W_1,T_2$ such that $V_1V_2=V_2V_1\overline{Q}$, where for any complex number $q$ of modulus one, $\overline{Q}=\overline{Q}_0\oplus qI$ with respect to the decomposition $\mathcal{K}=\mathcal{K}'\oplus \mathcal{K}'^{\perp}$.  
	Moreover, $\mathcal{K}'$ is a reducing subspace for $V_1$. 
\end{thm}
	\begin{thm}\label{Andodil2}
		Let $Q,T_1,T_2\in \mathcal{B}(\mathcal{H})$ be any contractions such that $\|T_2\|<1$, $T_1T_2=T_2T_1Q$. Then there is a Hilbert space $\mathcal{K}$ containing $\mathcal{H}$, an isometric lift $\overline{Q}\in \mathcal{B}(\mathcal{K})$ of $Q$ and isometric lifts $V_1,V_2$ of $T_1,T_2$ respectively, such that $V_1V_2=V_2V_1\overline{Q} $. 
	\end{thm} 
	As observed before, $(V, \mathcal{K})$ is an isometric lift of $(T,\mathcal{H})$, if and only if $V^*$ is a co-isometric extension of $T^*$. Hence we obtain a proof of Theorem \ref{Andodil2} by restating it as follows. 
	\begin{thm}\label{QLcoisopure}
		Let $Q,T_1,T_2\in \mathcal{B}(\mathcal{H})$ be any two contractions with $\|T_1\|<1$ and $T_1T_2=QT_2T_1$. Then there is a Hilbert space $\mathcal{K}$ containing $\mathcal{H}$, a co-isometric extension $\overline{Q}\in \mathcal{B}(\mathcal{K})$ and co-isometric extensions $Z_1,Z_2$ of $T_1,T_2$ respectively, such that $Z_{1}Z_{2}=\overline{Q}Z_{2}Z_1$.
	\end{thm}
	\begin{proof}
		Let $\mathcal{K}'=\mathcal{H}\oplus \mathcal{H}\oplus \cdots$,
		\[\overline{Q}_0(h_0,h_1,\dots )=(Qh_0+D_{Q^*}h_1, h_2,h_3,\dots),\] and 
		\[Z_2'(h_0,h_1,\ldots ) = \left(T_2h_{0}+D_{T_2^*}h_1, h_{2},\dots \right).\]
		Hence $\overline{Q}_0$ is a co-isometric extension of $Q$ and $Z_2'$ is a co-isometric extension of $T_2$. Then by Theorem \ref{QLcoisopurecoro} we get the desired result.  
	\end{proof}
	In particular, in above theorem, if $Q$ is a co-isometry, then we can take $\overline{Q}_0=Q\oplus I$ on $\mathcal{H}\oplus \mathcal{H}^{\perp}$ and obtain the following result.
	\begin{thm}\label{QLcoiso}
		Let $T_1,T_2\in \mathcal{B}(\mathcal{H})$ be any two contractions with $\|T_1\|<1$ and $Q\in \mathcal{B}(\mathcal{H})$ be any co-isometry such that $T_1T_2=QT_2T_1$. Then there is a Hilbert space $\mathcal{K}$ containing $\mathcal{H}$, a co-isometric extension $\overline{Q}\in \mathcal{B}(\mathcal{K})$ of form $\overline{Q}=Q\oplus I$ with respect to the decomposition $\mathcal{H}\oplus \mathcal{H}^{\perp}$ of $\mathcal{K}$ and co-isometric extensions $Z_1,Z_2$ of $T_1,T_2$ respectively, such that $Z_{1}Z_{2}=\overline{Q}Z_{2}Z_1$.
	\end{thm}
	\begin{proof}
		Let $\mathcal{K}'=\mathcal{H}\oplus \mathcal{H}\oplus \cdots$,
		\[\overline{Q}_0(h_0,h_1,\dots )=(Qh_0,h_1, h_2,h_3,\dots),\] and 
		\[Z_2'(h_0,h_1,\ldots ) = \left(T_2h_{0}+D_{T_2^*}h_1, h_{2},\dots \right).\]
		Since $ Q$ is a co-isometry, it follows that $\overline{Q}_0$ is a co-isometry. Hence it is a co-isometric extension of $Q$ and $Z_2'$ is a co-isometric extension of $T_2$. Thus, $\overline{Q}_0Z_2'$ is a co-isometric extension of $ QT_2$. Hence by Theorem \ref{Q_Rcoiso}, there is a norm preserving extension $Y_{1}$ of $T_1$ such that $Y_1Z_{2}'=\overline{Q}_0Z_2'Y_1$. Since $\|Y_1\|=\|T_1\|$ and $T_1$ is a strict contraction, so will be $Y_1$. Hence $D_{Y_1}$ will be invertible on $\mathcal{K}'$. Thus with respect to the decomposition $\mathcal{K}=\mathcal{K}'\oplus \mathcal{K}'\oplus \cdots$, let $\overline{Q}=\overline{Q}_0\oplus I$ and $Z_1,Z_2$ be defined as follows.  	\begingroup
		\allowdisplaybreaks
		\begin{gather*}
			Z_1=\begin{bmatrix}
				Y_1&D_{Y_1^*}&0&0&\cdots\\
				0&0&I_{\mathcal{K}'}&0&\cdots\\
				0&0&0&I_{\mathcal{K}'}&\cdots\\
				\vdots&\vdots&\vdots&\vdots&\ddots
			\end{bmatrix},\;		Z_2=\begin{bmatrix}
				W_2&0&0&\cdots \\
				0&D_{Y_1^*}^{-1}\overline{Q}_0Z_2'D_{Y_1^*}&0&\cdots \\
				0&0&D_{Y_1^*}^{-1}\overline{Q}_0Z_2'D_{Y_1^*}&\cdots \\
				\vdots&\vdots&\vdots&\ddots
			\end{bmatrix}.	\end{gather*}
		\endgroup
		Then $Z_1,Z_2$ are co-isometric extensions of $T_1,T_2$ such that  $Z_1Z_2=\overline{Q}Z_2Z_1$. (Proof of this part is similar to that of Theorem \ref{QLcoisopure}.)
	\end{proof}
	Rewriting it in terms of isometric liftings we obtain the following. 
	\begin{thm}\label{QRisoforgraph}
		Let $T_1,T_2\in \mathcal{B}(\mathcal{H})$ be any two contractions with $\|T_1\|<1$ and $Q\in \mathcal{B}(\mathcal{H})$ be any isometry such that $T_1T_2=T_2T_1Q$. Then there is a Hilbert space $\mathcal{K}$ containing $\mathcal{H}$, an isometric lift $\overline{Q}\in \mathcal{B}(\mathcal{K})$ of form $\overline{Q}=Q\oplus I$ with respect to the decomposition $\mathcal{H}\oplus \mathcal{H}^{\perp}=\mathcal{K}$ of $Q$ and isometric lifts $V_1,V_2$ of $T_1,T_2$ respectively, such that $V_{1}V_{2}=V_{2}V_1\overline{Q}$.
	\end{thm}
	\subsubsection{Ando-type dilation Case II.}
	Now we look into case II.
	\begin{thm}\label{andoisodilR}
		Let $T_1,T_2\in \mathcal{B}(\mathcal{H})$ be any contractions and $Q\in \mathcal{B}(\mathcal{H})$ be such that $ T_1Q $ is a contraction and $T_1T_2=T_2T_1Q$. Then there is a Hilbert space $\mathcal{K}$ containing $\mathcal{H}$, a lift $\overline{Q}$ of $Q$, a lift $V_{1Q}$ of $T_1$ and isometric lifts $V_1,V_2$ of $T_1,T_2$ respectively, such that $V_{1Q}\overline{Q}$ is an isometric lift of $T_1Q$ and $V_{1}V_{2}=V_{2}V_{1Q}\overline{Q}$.
	\end{thm} 
	\begin{proof}
		Define $\mathcal{K}=\oplus_{0}^{\infty} \mathcal{H}$, 
		\[W_{1}(h_0,h_1,\ldots ) = \left(T_1h_{0},D_{T_1}h_0,0, h_{1},\dots \right),\]
		\[W_{2}(h_0,h_1,\ldots ) = \left(T_2h_{0},D_{T_2}h_0,0, h_{1},\dots \right),\]
		\[W(h_0,h_1,\ldots ) = \left(T_1h_{0},h_1,0, h_{2},h_3,\dots \right).\] Let \[\overline{Q}(h_0,h_1,\dots )=(Qh_0,D_{T_1Q}h_0,h_1,h_2,\dots ).\] Therefore, 
		\[W\overline{Q}(h_0,h_1,\ldots ) = (T_1Qh_{0},D_{T_1Q}h_0,0, h_{1},h_2,\dots ) .\]  Then $W_{1}, W_{2}, W\overline{Q}$ are isometries as $\| T_ih_0 \|^2+\| D_{T_i}h_0 \|^2=\|h_0\|^2$ for $i=1,2$ and \[ \| T_1Qh_0 \|^2+\| D_{T_1Q}h_0 \|^2=\|h_0\|^2. \] Moreover, it is clear from the definition that $\mathcal{H}$ is invariant for $W_1^*,W_2^* ,(W\overline{Q})^*$ and for $i=1,2$, 
		\begin{equation}\label{invarianceR}
			W_i^*|_{\mathcal{H}}=T_i^*  \text{ and } (W\overline{Q})^*|_{\mathcal{H}}=(T_1Q)^* .
		\end{equation}  
		Now let $$\mathcal{B}=\mathcal{H}\oplus \mathcal{H}\oplus \mathcal{H}\oplus \mathcal{H} .$$
		Hence $\mathcal{K}$ can be identified with $\mathcal{H}\oplus \mathcal{B}\oplus \mathcal{B}\oplus \cdots$ via the identification \[ (h_0,h_1,\ldots )=(h_0, (h_1,h_2,h_3,h_4),(h_5,h_6,h_7,h_8),\ldots ). \] 
		Let $$\mathcal{L}_1= \{(D_{T_1}T_2h,0,D_{T_2}h,0)\in \mathcal{B}: h\in \mathcal{H}\} $$ and $$\mathcal{L}_2= \{(D_{T_2}T_1Qh,0,D_{T_1Q}h,0)\in \mathcal{B}: h\in \mathcal{H}\} .$$ Let $\mathcal{M}_i=\overline{\mathcal{L}}_i$ and $\mathcal{M}_i^{\perp}=\mathcal{B}\ominus\mathcal{M}_i$ for $i=1,2$. Define an operator $G:\mathcal{L}_1 \to \mathcal{L}_2$ as \[ G(D_{T_1}T_2h,0,D_{T_2}h,0)=(D_{T_2}T_1Qh,0,D_{T_1Q}h,0).\]
		A simple computation shows that \begin{align*}
			&\|D_{T_2}T_1Qh \|^2+\| D_{T_1Q}h \|^2\\
			=&\langle D_{T_2}T_1Qh, D_{T_2}T_1Qh \rangle + \langle D_{T_1Q}h,D_{T_1Q}h \rangle\\
			=&\langle (I-T_2^*T_2)T_1Qh, T_1Qh \rangle +\langle (I-(T_1Q)^*T_1Q)h, h \rangle\\
			=& \|T_1Qh\|^2-\|T_2T_1Qh\|^2+\|h\|^2-\|T_1Qh\|^2\\
			=&\| h \|^2-\|T_2T_1Qh \|^2. 
		\end{align*}
		Similarly, 
		\[ \|D_{T_1}T_2h \|^2+\| D_{T_2}h \|^2=\| h \|^2-\|T_1T_2h \|^2 \] Since $T_1T_2=T_2T_1Q$ the above computation proves that
		\[\|D_{T_2}T_1Qh \|^2+\| D_{T_1Q}h \|^2= \|D_{T_1}T_2h \|^2+\| D_{T_2}h \|^2.  \] Hence $G$ defines an isometry from $\mathcal{L}_1$ onto $\mathcal{L}_2$ and extends continuously as an isometry from $\mathcal{M}_1$ onto $\mathcal{M}_2$. It remains to show that $G$ can be extended to an isometry from the whole space $\mathcal{B}$ onto itself. For this purpose,  it suffices to prove that $\dim \mathcal{M}_1^{\perp}=\dim \mathcal{M}_2^{\perp}. $ This is clearly true when $\mathcal{H}$ and hence $\mathcal{B}$ are finite dimensional. Now  suppose $\mathcal{H}$ is infinite dimensional. Then we have, \[ \dim(\mathcal{H})=\dim\mathcal{B}\geq \dim\mathcal{M}_i^{\perp}\geq \dim \mathcal{H}. \] Here first inequality follows as $\mathcal{M}_i^{\perp}\subset \mathcal{B}$ and second inequality follows as \[\{ (0,h,0,0)\in \mathcal{B}:h\in \mathcal{H}\}\subset \mathcal{M}_i^{\perp}.\] Hence we can extend the map $G$ isometrically to a map from $\mathcal{B}$ onto $\mathcal{B}$. Therefore $G$ is a unitary. Now define $\widetilde{G}$ on $\mathcal{K}$ as \[ \widetilde{G}(h_0,h_1,\ldots )=(h_0,G(h_1,h_2,h_3,h_4),G(h_5,h_6,h_7,h_8),\ldots). \] Then clearly $\widetilde{G}$ is a unitary with inverse given by \[ \widetilde{G}^{-1}(h_0,h_1,\ldots )=(h_0,G^{-1}(h_1,h_2,h_3,h_4),G^{-1}(h_5,h_6,h_7,h_8),\ldots) .\] Further, let $V_{1}=\widetilde{G}W_1$, $V_{1Q}=\widetilde{G}W$ and $V_2=W_2\widetilde{G}^{-1}$. Since $\mathcal{H}$ is reducing for $\widetilde{G}$ and both $\widetilde{G}$ and $\widetilde{G}^*=\widetilde{G}^{-1}$ are identity on $\mathcal{H}$, from \eqref{invarianceR} we get, \[ V_2^*|_{\mathcal{H}}=\widetilde{G}W_2^*|_{\mathcal{H}}=T_2^*,\; V_1^*|_{\mathcal{H}}=W_1^*\widetilde{G}^*|_{\mathcal{H}}=T_1^*\]  and \[(V_{1Q}\overline{Q})^*|_{\mathcal{H}}=(\widetilde{G}W\overline{Q})^*|_{\mathcal{H}}=(W\overline{Q})^*\widetilde{G}^*|_{\mathcal{H}}=(T_1Q)^* .\]  Clearly $V_1$, $V_2$ and $V_{1Q}\overline{Q}$ are isometries. Therefore, $V_1$, $V_2$ and $V_{1Q}\overline{Q}$ are isometric lifts of $T_1$, $T_2$ and $T_1Q$ respectively. Now 
		\begin{align*}
			V_2V_{1Q}\overline{Q}(h_0,h_1,\dots)&=W_2\widetilde{G}^{-1}\widetilde{G}W\overline{Q}(h_0,h_1,\dots)\\
			&=W_2W\overline{Q}(h_0,h_1,\dots)\\
			&=W_2(T_1Qh_{0},D_{T_1Q}h_0,0, h_{1},h_2,\dots )\\
			&=(T_2T_1Qh_0,D_{T_2}T_1Qh_0,0,D_{T_1Q}h_0,0,h_1,h_2\dots).
		\end{align*} 
		Similarly, 
		\begin{align*}
			&V_1V_2(h_0,h_1,\dots )\\
			=&\widetilde{G}W_1W_2\widetilde{G}^{-1}(h_0,h_1,\dots )\\
			=&\widetilde{G}W_1W_2(h_0,G^{-1}(h_1,h_2,h_3,h_4), G^{-1}(h_5,h_6,h_7,h_8),\dots)\\
			=&\widetilde{G}W_1(T_2h_{0},D_{T_2}h_0,0, G^{-1}(h_1,h_2,h_3,h_4), G^{-1}(h_5,h_6,h_7,h_8),\dots )\\
			=&\widetilde{G}(T_1T_2h_0,D_{T_1}T_2h_0,0,D_{T_2}h_0,0,G^{-1}(h_1,h_2,h_3,h_4), G^{-1}(h_5,h_6,h_7,h_8),\dots)\\
			=&(T_1T_2h_0,G(D_{T_1}T_2h_0,0,D_{T_2}h_0,0),(h_1,h_2,h_3,h_4),(h_5,h_6,h_7,h_8),\dots).
		\end{align*}
		Since $T_1T_2=T_2T_1Q$ and \[G(D_{T_1}T_2h_0,0,D_{T_2}h_0,0)=(D_{T_2}T_1Qh_0,0,D_{T_1Q}h_0,0),\] we get $V_1V_{2}=V_{2}V_{1Q}\overline{Q}$ as required. 		
	\end{proof}
	Correspondingly, we obtain $Q$-commuting co-isometric extension result.
	\begin{thm}\label{AndocoisoextR}
		Let $T_1,T_2\in \mathcal{B}(\mathcal{H})$ be any contractions and $Q\in \mathcal{B}(\mathcal{H})$ be such that $ QT_2 $ is a contraction and $T_1T_2=QT_2T_1$. Then there is a Hilbert space $\mathcal{K}$ containing $\mathcal{H}$, an extension $\overline{Q}$ of $Q$, co-isometric extensions $Z_1,Z_2$ of $T_1,T_2$ and an extension $Z_{2Q}$ of $T_2$ such that $\overline{Q}Z_{2Q}$ is a co-isometric extension of $QT_2$ on Hilbert space $\mathcal{K}$ and $Z_{1}Z_2=\overline{Q}Z_{2Q}Z_1$.
	\end{thm}
	\begin{proof}
		We have $T_1T_2=QT_2T_1$ which gives us $T_2^*T_1^*=T_1^*T_2^*Q^*$. Since $QT_2$ is a contraction, so is $T_2^*Q^* $. Hence due to Theorem \ref{andoisodilR}, there is a Hilbert space $\mathcal{K}$, a lift $\overline{Q}^*$ of $Q^*$, bounded operators $Y_1,Y_2$, $Y_{2Q}$ on $\mathcal{K}$ such that $Y_1,Y_2$, $Y_{2Q}\overline{Q}^*$ are isometric lifts of $T_1^*,T_2^*$, $T_2^*Q^*$ and $Y_2Y_1=Y_1Y_{2Q}\overline{Q}^* $. Taking $Z_1=Y_1^*$, $Z_2=Y_2^*$ and $Z_{2Q}=Y_{2Q}^* $ we obtain the desired result.  
	\end{proof}
	
	\vspace{0.4cm}
	
	\section{The $Q$-intertwining liftings and $Q$-intertwining extensions}\label{QSectioninter}
	
	\vspace{0.4cm}
	
\noindent In this Section, we apply similar techniques as in Section \ref{QSection1} to obtain a few lifting results for $Q$-intertwining operators. These results are generalizations of a few results of Section \ref{QSection1}.
  \begin{thm}\label{QLTinterlift}
		Let $Q,T_1,T_2,X\in \mathcal{B}(\mathcal{H})$ be such that $T_1,T_2$ and $QT_2$ are contractions and $XT_1=QT_2X$. Let $(V_1,\mathcal{K}_1)$ and $(V_2,\mathcal{K}_2)$ be any isometric lifts of $T_1$ and $T_2$ respectively. Let $\overline{Q}$ on $\mathcal{K}_2$ be any lift of $Q$ such that $\overline{Q}V_2$ is a contraction. Then there is a norm preserving lift $Y:\mathcal{K}_1\to \mathcal{K}_2$ of $X$ such that $YV_1=\overline{Q}V_2Y$. 
	\end{thm}
	\begin{proof}
		First assume that $\|X\|=1$. Given an isometric lift $(V_1,\mathcal{K}_1)$ of $(T_1,\mathcal{H})$, $(V_2,\mathcal{K}_2)$ of $(T_2,\mathcal{H})$ and a lift $\overline{Q}$ of $Q$ such that $\overline{Q}V_2$ is a contraction we get that $\overline{Q}V_2$ is a contractive lift of $QT_2$. Therefore, by Theorem \ref{contractive lift} there is a contractive lift $Y$ of $X$ such that $YV_1=\overline{Q}V_2Y$. Since $\|X\|=1$ and $Y$ is a contractive lift of $X$, so $1\geq \|Y\|\geq \|X\|=1$. Hence $\|Y\|=\|X\|=1$. Now if $0<\|X\|\neq 1$, then $XT_1=QT_2X$ implies that $\dfrac{X}{\|X\|}T_1=QT_2\dfrac{X}{\|X\|} $. Hence again there is a norm preserving lift $\hat{Y}$ of $\dfrac{X}{\|X\|}$ such that $\hat{Y}V_1=\overline{Q}V_2\hat{Y}$. Therefore, $\|X\|\hat{Y}V_1=\overline{Q}V_2\|X\|\hat{Y}$. Hence $Y=\|X\|\hat{Y}$ is a lift of $X$ such that $YV_1=\overline{Q}V_2Y$ and $\|Y\|=\|X\|$. This proves the claim.     
	\end{proof}
	
	\begin{thm}\label{QMinterlift}
		Let $Q,T_1,T_2,X\in \mathcal{B}(\mathcal{H})$ be such that $T_1,T_2$ and $T_2Q$ are contractions and $XT_1=T_2QX$. Let $(V_1,\mathcal{K}_1)$ and $(V_2,\mathcal{K}_2)$ be any isometric lifts of $T_1$ and $T_2$ respectively. Let $\overline{Q}$ on $\mathcal{K}_2$ be any lift of $Q$ such that $V_2\overline{Q}$ is a contraction. Then there is a norm preserving lift $Y:\mathcal{K}_1\to \mathcal{K}_2$ of $X$ such that $YV_1=V_2\overline{Q}Y$. 
	\end{thm}
	\begin{proof}
		First assume that $\|X\|=1$. Given the minimal isometric dilation $(V_1,\mathcal{K}_1)$ of $(T_1,\mathcal{H})$, $(V_2,\mathcal{K}_2)$ of $(T_2,\mathcal{H})$ and a lift $\overline{Q}$ of $Q$ such that $V_2\overline{Q}$ is a contraction we get that $V_2\overline{Q}$ is a contractive lift of $T_2Q$. Therefore, by Theorem \ref{contractive lift}
		there is a contractive lift $Y$ of $X$ such that $YV_1=V_2\overline{Q}Y$. Since $\|X\|=1$ and $Y$ is a contractive lift of $X$, so $1\geq \|Y\|\geq \|X\|=1$. Hence $\|Y\|=\|X\|=1$. Now if $0<\|X\|\neq 1$, then $XT_1=T_2QX$ implies that $\dfrac{X}{\|X\|}T_1=T_2Q\dfrac{X}{\|X\|} $. Hence again there is a norm preserving lift $\hat{Y}$ of $\dfrac{X}{\|X\|}$ such that $\hat{Y}V_1=V_2\overline{Q}\hat{Y}$. Hence $\|X\|\hat{Y}V_1=V_2\overline{Q}\|X\|\hat{Y}$. Hence $Y=\|X\|\hat{Y}$ is a lift of $X$ such that $YV_1=V_2\overline{Q}Y$ and $\|Y\|=\|X\|$. This proves the claim.     
	\end{proof}
	
	We can rewrite the results in terms of extensions as follows.
	
	\begin{thm}\label{QLTinterext}
		Let $Q,T_1,T_2,X\in \mathcal{B}(\mathcal{H})$ be such that $T_1,T_2$ and $T_2Q$ are contractions and $T_1X=XT_2Q$. Let $(V_1,\mathcal{K}_1)$ and $(V_2,\mathcal{K}_2)$ be any co-isometric extensions of $T_1$ and $T_2$ respectively. Let $\overline{Q}$ on $\mathcal{K}_2$ be any extension of $Q$ such that $V_2\overline{Q}$ is a contraction. Then there is a norm preserving extension $Y:\mathcal{K}_2\to \mathcal{K}_1$ of $X$ such that $V_1Y=YV_2\overline{Q}$. 
	\end{thm}
	\begin{proof}
		Follows from Theorem \ref{QLTinterlift}.
	\end{proof}
	
	\begin{thm}\label{QMinterext}
		Let $Q,T_1,T_2,X\in \mathcal{B}(\mathcal{H})$ be such that $T_1,T_2$ and $QT_2$ are contractions and $T_1X=XQT_2$. Let $(V_1,\mathcal{K}_1)$ and $(V_2,\mathcal{K}_2)$ be any co-isometric extensions of $T_1$ and $T_2$ respectively. Let $\overline{Q}$ on $\mathcal{K}_1$ be any extension of $Q$ such that $\overline{Q}V_2$ is a contraction. Then there is a norm preserving extension $Y:\mathcal{K}_2\to \mathcal{K}_1$ of $X$ such that $V_1Y=Y\overline{Q}V_2$. 
\end{thm}
	\begin{proof}
		Follows from Theorem \ref{QMinterlift}.
	\end{proof}
	
	\vspace{0.2cm}
	
	\section{A few intertwining lifting theorems}\label{Sectioninter}
	
	\vspace{0.4cm}
	
\noindent For an operator $X$ that intertwines with a pair of contractions $T,T'$, whenever we try to find lifting or co-isometric extension, we expect to find a norm preserving lift (or extension) of $X$ that intertwines with isometric lifts (or co-isometric extensions) $V$ and $V'$ of $T$ and $T'$ respectively. So, one can ask the following question: given an isometric lift (or a co-isometric extension) $W$ of $X$, can we find norm preserving lifts of $T$ and $T'$ that intertwine with $W$ ? We address this question bellow.

	\begin{thm}\label{Finter}
		Let $T_1,T_2,X\in \mathcal{B}(\mathcal{H})$ be contractions such that $T_1X=XT_2$. Let $(Z,\mathcal{K})$ be the minimal co-isometric extension of $X$. Then there is a norm preserving extensions $Y_1$ of $T_1$ and a contractive extension $Y_2$ of $T_2$ respectively on $\mathcal{K}$ such that $Y_1Z=ZY_2$ 
	\end{thm}
	\begin{proof}
		As the minimal co-isometric extension of a contraction is unique and $Z$ is the minimal co-isometric extension of $X$, let
		\[Z=\begin{bmatrix}
			X&D_{X^*}&0&0&\cdots\\
			0&0&I_{\mathcal{K}}&0&\cdots\\
			0&0&0&I_{\mathcal{K}}&\cdots\\
			\vdots&\vdots&\vdots&\vdots&\ddots
		\end{bmatrix},\] with respect to the decomposition $\mathcal{K}=\mathcal{H}\oplus \mathcal{D}_{X^*}\oplus \cdots$.
		We observe that, 
		\begin{alignat*}{2}
			&XD_{T_2^*}(XD_{T_2^*})^*+D_{X^*}(D_{X^*})-T_1D_{X^*}(T_1D_{X^*})^* \\
			=&X(I_{\mathcal{H}}-T_2T_2^*)X^*+	I_{\mathcal{H}}-XX^*-T_1(I_{\mathcal{H}}-XX^*)T_1^*\\
			=&-XT_2T_2^*X^*+I_{\mathcal{H}}-T_1T_1^*+T_1XX^*T_1^*\\
			=&I_{\mathcal{H}}-T_1X(T_1X)^*-T_1T_1^*+T_1XX^*T_1^* & [\text{as } T_1X=XT_2]\\
			=&I-T_1T_1^*\\
			\geq&0.  & [\text{as } T_1 \text{ is a contraction}]
		\end{alignat*}
		Hence we have, \[ XD_{T_2^*}(XD_{T_2^*})^*+D_{X^*}(D_{X^*})\geq T_1D_{X^*}(T_1D_{X^*})^* .\]
		Therefore, by Theorem \ref{Dthm1}, there are contractions $Z_1:\mathcal{H}\to \mathcal{H}$, $Z_2:\mathcal{H}\to \mathcal{H}$ satisfying \begin{equation}\label{T12}
			XD_{T_2^*}Z_1+D_{X^*}Z_2=T_1D_{X^*}
		\end{equation} and 
		\begin{equation}\label{T13} 
			Z_1^*Z_1+Z_2^*Z_2\leq I_{\mathcal{H}}.
		\end{equation}
		Further, let $Z_2=\begin{bmatrix}
			Z_{11}&Z_{12}\\
			Z_{21}&Z_{22}
		\end{bmatrix}$ with respect to the decomposition $\mathcal{H}=\mathcal{D}_{X^*}\oplus \mathcal{D}_{X^*}^{\perp}$. Clearly, $D_{X^*}Z_2'=\begin{bmatrix}
			D_{X^*}Z_{11}&D_{X^*}Z_{12}\\0&0
		\end{bmatrix}$. By defining $Z_2'=\begin{bmatrix}
			Z_{11}&Z_{12}\\0&0
		\end{bmatrix}$, we observe that \eqref{T12} holds if we replace $Z_2$ by $Z_2'$ and from \eqref{T13} we have 
		\begin{equation}\label{T14} 
			Z_1^*Z_1+Z_2'^*Z_2'\leq I_{\mathcal{H}}.
		\end{equation}
		Clearly, $Z_2'$ is a map from $\mathcal{H}$ to $\mathcal{D}_{X^*}$ having $\mathcal{D}_{X^*}$ as an invariant subspace. Hence let $S_{22}=Z_2'|_{\mathcal{D}_{X^*}}=Z_{11}$. Let $Z_1'=Z_1|_{\mathcal{D}_{X^*}}:\mathcal{D}_{X^*}\to \mathcal{H}$. Hence, it is clear from \eqref{T12} that 
		\begin{equation}\label{T15}
			XD_{T_2^*}Z_1'+D_{X^*}Z_{11}=T_1D_{X^*}|_{\mathcal{D}_{X^*}} :\mathcal{D}_{X^*}\to \mathcal{H}.
		\end{equation} and from \eqref{T14} that 
		\begin{equation}\label{T16} 
			Z_1'^*Z_1'+Z_{11}^*Z_{11}\leq I_{\mathcal{D}_{X^*}}.
		\end{equation}    
		Now \eqref{T16} implies that $Z_1'^*Z_1'\leq I-Z_{11}^*Z_{11}$, hence by Lemma \ref{Dlemma}, there is a contraction $C:\mathcal{D}_{X^*}\to \mathcal{H}$ such that $Z_1'=CD_{Z_{11}} $. Substituting this in \eqref{T15}, we have  
		\begin{equation}\label{T17}	
			XD_{T_2^*}CD_{Z_{11}}+D_{X^*}Z_{11}=T_1D_{X^*},
		\end{equation} 
		for some contractions $C:\mathcal{D}_{X^*}\to \mathcal{H}$ and $Z_{11}:\mathcal{D}_{X^*}\to \mathcal{D}_{X^*}\subseteq \mathcal{H}$. Now we define $Y_1$ and $Y_2$ on $\mathcal{K}=\mathcal{H}\oplus \mathcal{D}_{X^*}\oplus \cdots$ as, 
		\begingroup
		\allowdisplaybreaks
		\begin{gather*}
			Y_1=\begin{bmatrix}
				T_1&0&0&\cdots \\
				0&0&0&\cdots \\
				0&0&0&\cdots \\
				\vdots&\vdots&\vdots&\ddots
			\end{bmatrix},\;
		Y_2=\begin{bmatrix}
				T_2&D_{T_2^*}CD_{Z_{11}}&0&\cdots \\
				0&Z_{11}&0&\cdots \\
				0&0&0&\cdots \\
				\vdots&\vdots&\vdots&\ddots
			\end{bmatrix}.
		\end{gather*}
		\endgroup
		Then clearly, $Y_1$ is a norm preserving extension of $T_1$ and $Y_2$ is a contractive extensions of $T_2$ such that $Y_1Z=ZY_2$. 
	\end{proof}
	The minimality restriction can be removed and hence we obtain the following.
	\begin{thm}\label{Finter1}
		Let $T_1,T_2,X\in \mathcal{B}(\mathcal{H})$ be contractions such that $T_1X=XT_2$. Let $Z\in \mathcal{B}(\mathcal{K})$ be a co-isometric extension of $X$. Then there is a norm preserving extensions $Y_2$ of $T_2$ and a contractive extension $Y_1$ of $T_1$ respectively on $\mathcal{K}$ such that $Y_1Z=ZY_2$. 
	\end{thm}
	\begin{proof}
		The proof follows from an argument similar to that of the Corollary 4.1 in \cite{Dou:Muh:Pea}. Let $\mathcal{K}'\subseteq \mathcal{K}$ be the smallest reducing subspace for $Z$ containing $\mathcal{H}$. Hence, $Z'=Z|_{\mathcal{K}'}$ is the minimal co-isometric extension of $X$. Then by Theorem \ref{Finter}, there is a norm preserving extension $Y_1'$ of $T_1$ and a contractive extension $ Y_2'$ of $T_2$ such that $Y_1'Z'=Z'Y_2'$. Now if we define $Y_1=Y_1'\oplus 0$ and $Y_2=Y_2'\oplus 0$ with respect to the decomposition $\mathcal{K}=\mathcal{K}'\oplus \mathcal{K}'^{\perp}$ then we obtain the desired result.
	\end{proof}
	Now we obtain a similar intertwining lifting result. 
	\begin{thm}\label{Finterlift}
		Let $X,T_1,T_2\in \mathcal{B}(\mathcal{H})$ be contractions such that $T_1X=XT_2$. Let $V\in \mathcal{B}(\mathcal{K})$ be an isometric lift of $X$. Then there is a contractive lift $S_1$ of $T_1$ and a norm preserving lift $S_2$ of $T_2$ on $\mathcal{K}$ such that $S_1V=VS_2$. 
	\end{thm}
	\begin{proof}
		Since, $T_1X=XT_2$, we have $T_2^*X^*=X^*T_1^*$. As $V$ is an isometric lift of $X$, $V^*$ is a co-isometric extension of $X^*$. Hence by Theorem \ref{Finter1}, there is a contractive extension $Y_1$ of $T_1^*$ and a norm preserving extension $ Y_2$ of $T_2^*$ such that $Y_2V^*=V^*Y_1$. Taking $S_1=Y_1^*$ and $S_2=Y_2^*$ proves the desired result. 
	\end{proof}
	As an application of Theorem \ref{Finter} we obtain the following.
	\begin{thm}\label{interdil}
		Let $T_1,T_2,X\in \mathcal{B}(\mathcal{H})$ be any contractions such that $T_1X=XT_2$. Then there are co-isometric extensions $Z_i\in \mathcal{B}(\mathcal{K}_i)$ of $T_i$ for $i=1,2$ and a co-isometric extension $R\in \mathcal{B}(\mathcal{K}_2,\mathcal{K}_1)$ of $(X,\mathcal{H})$ such that $Z_1R=RZ_2$.   
	\end{thm}
	\begin{proof}
		Let $(Z,\mathcal{K}')$ be the minimal co-isometric extension of $X$. Then due to Theorem \ref{Finter}, there is a norm preserving extension $Y_1$ of $T_1$ and a contractive extension $Y_2$ of $T_2$ on $\mathcal{K}'$, such that $Y_1Z=ZY_2$. Let $(Z_i,\mathcal{K}_i)$ be the minimal co-isometric extensions of $(Y_i,\mathcal{K}')$ and hence of $(T_i,\mathcal{H})$, for $i=1,2$. Now by Theorem \ref{douglasintertwining}, we know that there is a norm preserving co-isometric extension $R$ of $Z$ from $\mathcal{K}_2$ to $\mathcal{K}_1$ such that $Z_1R=RZ_2$. Now it remains to prove that $R$ is a co-isometry. Let \[ R= \begin{bmatrix}
			Z&A\\
			0&B
		\end{bmatrix}:\mathcal{K}'\oplus (\mathcal{K}_1\ominus \mathcal{K}')\to \mathcal{K}'\oplus (\mathcal{K}_2\ominus \mathcal{K}') .\] 
		We first observe that, \begin{align*}
			&0\leq ZZ^*+AA^*\leq \|ZZ^*+AA^* \|I\leq \|RR^* \|I\leq I\\
			\implies & 0\leq I+AA^*\leq I. \hspace{5cm} [\text{as }Z \text{ is a co-isometry }]
		\end{align*}
		Thus, $A=0$. Hence $\mathcal{K}'$ is a reducing subspace for $R$ and in fact $R^*|_{\mathcal{K}'}=Z^*$. Now observe that for $k\in \mathcal{K}'$ and for $n\geq 2$, \begin{align*}
			&R^*(Z_1^{*n}k)=R^*Z_1^*(Z_1^{*n-1})=Z_2^*R^*(Z_1^{*n-1})\\
			\implies& \|R^*(Z_1^{*n}k) \|=\|Z_2^*R^*(Z_1^{*n-1}) \| =\| R^*(Z_1^{*n-1})\|.
		\end{align*}
		Therefore, inductively for all $n\geq 1$, \[\|R^*(Z_1^{*n}k) \|= \|R^*Z_1^* k\|=\|Z_2^*R^*k\|=\|R^*k\|=\|Z^*k \|=\|k\|=\|Z_1^{*n}k\|. \]	Since, $\mathcal{K}_1=\overline{Span}\{Z_1^{*n}k:n\in \mathbb{N}\cup \{0\},\, k\in \mathcal{K}'  \}$ we get $I_{\mathcal{K}_1}-RR^*=0$. Hence $R$ is a co-isometry. 
	\end{proof}
	In particular, if $X$ is a strict contraction intertwining $T_1,T_2$ then we obtain a pure co-isometric extension $Z$ of $X$ intertwining the minimal co-isometric extensions of $T_1,T_2$ as follows. 
	\begin{thm}
		Let $T_1,T_2,X\in \mathcal{B}(\mathcal{H})$ be any contractions with $\|X\|<1$ and $T_1X=XT_2$. Then there is a Hilbert space $\mathcal{K}$ containing $\mathcal{H}$ and there are co-isometric extensions $Z_1,Z_2,R$ of $T_1,T_2,X$ respectively such that $Z_1R=RZ_2$. Moreover this $R$ is a pure co-isometry.	
	\end{thm}
	\begin{proof}
		Let $\mathcal{K}'=\mathcal{H}\oplus \mathcal{H}\oplus \dots$ and for any $h=(h_0,h_1,\dots )\in \mathcal{H}\oplus \mathcal{H}\oplus \dots$, 
		\[ Z_1'(h_0,h_1,\dots)=(T_1h_0+D_{T_1^*}h_1,h_2,\dots ) \] and 
		\[ Z_2'(h_0,h_1,\dots)=(T_2h_0+D_{T_2^*}h_1,h_2,\dots ).\]
		Hence clearly, $Z_1',Z_2'$ are co-isometric extensions of $T_1,T_2$ respectively. Then by Theorem \ref{douglasintertwining}, we know that there is a norm preserving extension $Y$ of $X$ such that $Z_1'Y=YZ_2'$. Now since $X$ is a strict contraction, so will be $Y$. Hence $D_{Y^*}$ will be invertible on $\mathcal{K}'$. Thus, we define,	 
		\begingroup
		\allowdisplaybreaks
		\begin{gather*}
			R=\begin{bmatrix}
				Y&D_{Y^*}&0&0&\cdots\\
				0&0&I_{\widetilde{\mathcal{K}}}&0&\cdots\\
				0&0&0&I_{\widetilde{\mathcal{K}}}&\cdots\\
				\vdots&\vdots&\vdots&\vdots&\ddots
			\end{bmatrix},\;Z_1=\begin{bmatrix}
				Z_1'&0&0&\cdots \\
				0&D_{Y^*}^{-1}Z_1'D_{Y^*}&0&\cdots \\
				0&0&D_{Y^*}^{-1}Z_1'D_{Y^*}&\cdots \\
				\vdots&\vdots&\vdots&\ddots
			\end{bmatrix},\\
			Z_2=\begin{bmatrix}
				Z_2'&0&0&\cdots \\
				0&D_{Y^*}^{-1}Z_1'D_{Y^*}&0&\cdots \\
				0&0&D_{Y^*}^{-1}Z_1'D_{Y^*}&\cdots \\
				\vdots&\vdots&\vdots&\ddots
			\end{bmatrix},
		\end{gather*}
		\endgroup
		on $\mathcal{K}=\mathcal{K}'\oplus \mathcal{K}'\oplus \cdots$.
		It can easily be observed that $Z_1R=RZ_2$ and $R|_{\mathcal{H}}=X, Z_1|_{\mathcal{H}}=T_1$ and $Z_2|_{\mathcal{H}}=T_2$. Also $Z_1,Z_2$ are co-isometries and $R$ is a strict co-isometry. 
	\end{proof}
	In the remainder of this section, we give some alternative proofs of intertwining lifting and intertwining dilation theorems using different methods.
	\subsection{An approach due to Douglas, Muhly and Pearcy}
	In this section, we give an Alternative proof of intertwining lifting theorem by Douglas's approach.
	First we observe that in Proposition \ref{prop:21}, if we consider $T_i\in \mathcal{B}(\mathcal{H}_i,\mathcal{H}_i')$ instead of $T_i\in\mathcal{B}(\mathcal{H}_i)$ then we obtain a similar result which is as follows. The proof follows by same method as that of the Proposition \ref{prop:21}.    
	\begin{prop}\label{Dgenprop}
		Let $T_1\in \mathcal{B}(\mathcal{H}_1,\mathcal{H}_1')$ and $T_2\in \mathcal{B}(\mathcal{H}_2,\mathcal{H}_2')$ be contractions. Suppose $X$ is an operator from $\mathcal{H}_1$ to $\mathcal{H}_2'$. Then \[
		Y = \begin{bmatrix}
			T_1 & 0\\
			X & T_2
		\end{bmatrix}:\mathcal{H}_1\oplus \mathcal{H}_2 \to \mathcal{H}_1'\oplus \mathcal{H}_2'\]
		is a contraction if and only if $X = D_{T_2^*}CD_{T_1}$ for some contraction $C : \mathcal{H}_1 \to \mathcal{H}_2'$.
	\end{prop}
	\begin{proof}
		The operator $Y = \begin{bmatrix}
			T_1 & 0\\
			X & T_2
		\end{bmatrix}$ is a contraction if and only if 
		$
		Y^*Y \leq \begin{bmatrix}
			I_{\mathcal{H}_1} & 0\\
			0 & I_{\mathcal{H}_2}
		\end{bmatrix}
		$, which is  equivalent to,  
		\begingroup
		\allowdisplaybreaks
		\begin{align}\label{eqmatrix}
			&\left(
			\begin{bmatrix}
				T_1 & 0\\
				0 & 0
			\end{bmatrix} + \begin{bmatrix}
				0 & 0\\
				X & T_2
			\end{bmatrix}
			\right)^*\left( \begin{bmatrix}
				T_1 & 0\\
				0 & 0
			\end{bmatrix} + \begin{bmatrix}
				0 & 0\\
				X & T_2
			\end{bmatrix} \right) \leq \begin{bmatrix}
				I_{\mathcal{H}_1} & 0\\
				0 & I_{\mathcal{H}_2}
			\end{bmatrix}\notag\\
			\iff & \begin{bmatrix}
				T_1^* & 0\\
				0 & 0
			\end{bmatrix}\begin{bmatrix}
				T_1 & 0\\
				0 & 0
			\end{bmatrix} + \begin{bmatrix}
				0 & 0\\
				X & T_2
			\end{bmatrix}^*\begin{bmatrix}
				0 & 0\\
				X & T_2
			\end{bmatrix} \leq \begin{bmatrix}
				I_{\mathcal{H}_1} & 0\\
				0 & I_{\mathcal{H}_2}
			\end{bmatrix}\notag\\
			\iff & \begin{bmatrix}
				T_1^*T_1 & 0\\
				0 & 0
			\end{bmatrix}+ \begin{bmatrix}
				0 & 0\\
				X & T_2
			\end{bmatrix}^*\begin{bmatrix}
				0 & 0\\
				X & T_2
			\end{bmatrix} \leq \begin{bmatrix}
				I_{\mathcal{H}_1} & 0\\
				0 & I_{\mathcal{H}_2}
			\end{bmatrix}\notag\\
			\iff & \begin{bmatrix}
				0 & 0\\
				X & T_2
			\end{bmatrix}^*\begin{bmatrix}
				0 & 0\\
				X & T_2
			\end{bmatrix} \leq \begin{bmatrix}
				I_{\mathcal{H}_1}-T_1^*T_1 & 0\\
				0 & I_{\mathcal{H}_2}
			\end{bmatrix}.
		\end{align}
		\endgroup
		By Lemma \ref{Dlemma}, Equation \eqref{eqmatrix} holds if and only if there is a contraction \[Z = \begin{bmatrix}
			Z_{11} & Z_{12}\\
			Z_{21} & Z_{22}
		\end{bmatrix}: \mathcal{H}_1 \oplus \mathcal{H}_2 \to  \mathcal{H}_1' \oplus \mathcal{H}_2'
		\]
		such that, 
		\begingroup
		\allowdisplaybreaks
		\begin{align}\label{eqmatrix1}
			\begin{bmatrix}
				0 & 0\\
				X & T_2
			\end{bmatrix}^* &= \begin{bmatrix}
				D_{T_1} & 0\\
				0 & I_{\mathcal{H}_2}
			\end{bmatrix}\begin{bmatrix}
				Z_{11} & Z_{12}\\
				Z_{21} & Z_{22}
			\end{bmatrix}^*\notag\\
			& = \begin{bmatrix}
				D_{T_1}Z_{11}^* & D_{T_1}Z_{21}^*\\
				Z_{12}^* & Z_{22}^*
			\end{bmatrix}.
		\end{align}
		\endgroup
		Equation \eqref{eqmatrix1} holds if and only if there is a contraction $Z$ such that $Z_{12} =0$, $Z_{22}=T_2$, $X=Z_{21}D_{T_1}$ and $D_{T_1}Z_{11}^*=0$. This is equivalent to existence of a contraction $Z' =\begin{bmatrix}
			0 & 0\\
			Z_{21} & T_2
		\end{bmatrix}$ satisfying $X = Z_{21}D_{T_1}$. Now $Z'$ is a contraction if and only if $Z'Z'^*\leq I_{\mathcal{H}_1'\oplus \mathcal{H}_2'}$ if and only if $Z_{21}Z_{21}^*+T_2T_2^*\leq I_{\mathcal{H}_2'}$, that is, $Z_{21}Z_{21}^*\leq I_{\mathcal{H}_2'}-T_2T_2^*$. Hence so far we have that, $Y = \begin{bmatrix}
			T_1 & 0\\
			X & T_2
		\end{bmatrix}$ is a contraction if and only if there is \[Z' =\begin{bmatrix}
			0 & 0\\
			Z_{21} & T_2
		\end{bmatrix}:\mathcal{H}_1\oplus \mathcal{H}_2\to \mathcal{H}_1'\oplus \mathcal{H}_2'\] satisfying $Z_{21}Z_{21}^*\leq I_{\mathcal{H}_2'}-T_2T_2^*$ and $X = Z_{21}D_{T_1}$. 
		By Lemma \ref{Dlemma}, $Z_{21}Z_{21}^*\leq I_{\mathcal{H}_2'}-T_2T_2^*$ holds if and only if there is a contraction $C:\mathcal{H}_1\to \mathcal{H}_2'$ such that $Z_{21}=D_{T_2^*}C$. Therefore, $Y$ is a contraction if and only if $X=D_{T_2^*}CD_{T_1}$ for some contraction $C$. 	
	\end{proof}
	Following the same technique as that of the proof of Theorem 3 in \cite{Dou:Muh:Pea}, we obtain a generalized version of it for intertwining operators instead of commuting operators.
	\begin{thm}\label{qpart}
		Let $T_1\in\mathcal{B}(\mathcal{H}_1)$, $T_2\in\mathcal{B}(\mathcal{H}_2)$ and $X\in\mathcal{B}(\mathcal{H}_2,\mathcal{H}_1)$ be contractions such that $T_1X=XT_2$. Suppose $V_1=\begin{bmatrix}
			T_1&0\\
			S_1&0
		\end{bmatrix}$ on $\mathcal{H}_1\oplus \mathcal{H}_1'$, $V_2=\begin{bmatrix}
			T_2&0\\
			S_2&0
		\end{bmatrix}$ on $\mathcal{H}_2\oplus \mathcal{H}_2'$ such that $T_1^*T_1+S_1^*S_1=I_{\mathcal{H}_1}$ and $T_2^*T_2+S_2^*S_2=I_{\mathcal{H}_2}$. Then there is a map $Y=\begin{bmatrix}
			X&0\\
			A&B
		\end{bmatrix}$ from $\mathcal{H}_2\oplus \mathcal{H}_2'$ to $\mathcal{H}_1\oplus \mathcal{H}_1'$ such that $V_1Y=YV_2$ and $\|Y\|=\|X\|$.  
	\end{thm}
	\begin{proof}
		Without loss of generality assume $\|X\|=1$. We want to find a matrix $Y=\begin{bmatrix}
			X&0\\A&B
		\end{bmatrix}$ from $\mathcal{H}_2\oplus \mathcal{H}_2'$ to $\mathcal{H}_1\oplus \mathcal{H}_1'$ such that $V_1Y=YV_2$, that is, to find $A$ and $B$ such that \begin{equation}\label{qgen1}
			AT_2+BS_2=S_1X.
		\end{equation} For such a $Y$ to be a contraction, due to Proposition \ref{Dgenprop} it suffices to show that $A$, $B$ are contractions and \begin{equation}\label{qgen2}
			A=D_{B^*}CD_{X},
		\end{equation} for some contraction $C\in \mathcal{B}(\mathcal{H}_2,\mathcal{H}_1')$. We first construct $K$ from $\mathcal{H}_2$ to $\mathcal{H}_1'$ such that $A=KD_{X}$. Thus Equation \eqref{qgen1} gives us, 
		\begin{equation*} 
			KD_{X}T_2+BS_2=S_1X. 
		\end{equation*} This implies, 
		\begin{equation}\label{qgen3}
			T_2^*D_{X}K^*+S_2^*B^*=X^*S_1^*. 
		\end{equation} In order to find $K^*$ and $B^*$ satisfying \eqref{qgen3}, due to Theorem \ref{Dthm1}, it suffices to show that $$(T_2^*D_{X})(T_2^*D_{X})^*+S_2^*S_2\geq X^*S_1^*S_1X.$$ 
		Since $X^*X\leq I_{\mathcal{H}_2}$, $T_2^*T_2+S_2^*S_2=I_{\mathcal{H}_2}$ and $T_1^*T_1+S_1^*S_1=I_{\mathcal{H}_1}$, we have 
		\begingroup
		\allowdisplaybreaks
		\begin{align*}
			(T_2^*D_{X})(T_2^*D_{X})^*+S_2^*S_2&=T_2^*(I-X^*X)T_2+S_2^*S_2\\
			&\geq T_2^*(I-X^*X)T_2+S_2^*S_2-(I_{\mathcal{H}_2}-X^*X)\\
			& = X^*X - T_2^*X^*XT_2\\
			& = X^*X - X^*T_1^*T_1X\\
			& = X^*(I_{\mathcal{H}_1}-T_1^*T_1)X\\
			& = X^*S_1^*S_1X.
		\end{align*}
		\endgroup
		Hence by Theorem \ref{Dthm1}, there are operators $K:  \mathcal{H}_2 \to \mathcal{H}_1'$ and $B : \mathcal{H}_2' \to \mathcal{H}_1'$ satisfying \eqref{qgen3} and $KK^* + BB^*\leq I_{\mathcal{H}_1'}$. This implies that $B$ is a contraction and $KK^* \leq I_{\mathcal{H}_1'} - BB^*$. Again by applying Lemma \ref{Dlemma}, there is a contraction $C$ from $\mathcal{H}_2$ to $\mathcal{H}_1'$ such that $K = D_{B^*}C$. 
		
		Clearly $\|Y\|\geq \|X\|=1$. Since $X, B$ are contractions and $A=D_{B^*}CD_{X}$ for some contraction $C$, it follows from Proposition \ref{Dgenprop} that $\|Y\|\leq 1$.
		This completes the proof.
	\end{proof}
	Now we are ready to give an alternative proof of the intertwining lifting theorem by means of a similar technique as that in the proof of Theorem 4 in \cite{Dou:Muh:Pea}.
	\begin{thm}[Theorem \ref{intertwining}]\label{maininter}
		Let $T,T',X\in \mathcal{B}(\mathcal{H})$ be such that $T$ and $T'$ are contraction and $TX=XT'$. Let $V$ on $\mathcal{K}$ be the minimal isometric dilation of $T$ and $V'$ on $\mathcal{K}'$ be the minimal isometric dilation of $T'$. Then there is a bounded linear operator $Y:\mathcal{K}' \rightarrow \mathcal{K}$ such that 
		\begin{itemize}
			\item[(1)] $\mathcal{H}$ is invariant under $Y^*$ and $Y^*|_{\mathcal{H}}=X^*$;
			\item[(2)] $\|Y\|=\|X\|$ and
			\item[(3)] $VY=YV'$.
		\end{itemize}
	\end{thm}
	\begin{proof}
		As $\mathcal{K}$ and $\mathcal{K}'$ are the minimal isometric dilation spaces for $T$ and $T'$ respectively, let
		\[
		\mathcal{K} =  \mathcal{H}\oplus \mathcal D_{T}\oplus \mathcal D_{T}\oplus \dots
		\]
		and
		\[
		\mathcal{K}'=  \mathcal{H}\oplus \mathcal D_{T'}\oplus \mathcal D_{T'}\oplus \dots.
		\]
		Consider for all $n\in \mathbb{N}\cup \{0 \}$, \[ \mathcal{K}_n=\mathcal{H}\oplus\underbrace{\mathcal{D}_{T}\oplus\cdots \oplus\mathcal{D}_{T}}_{\text{n copies}}\oplus \{ 0\}\oplus \cdots \] and
		\[ \mathcal{K}_n'=\mathcal{H}\oplus\underbrace{\mathcal{D}_{T'}\oplus\cdots \oplus\mathcal{D}_{T'}}_{\text{n copies}}\oplus \{ 0\}\oplus \cdots .\] 
		One can easily check that $\mathcal{K}_n$ and $\mathcal{K}_n'$ are invariant subspaces for $V^*$ and $V'^*$ respectively. Hence define $V_{n}^*=V^*|_{\mathcal{K}_n}$ and $V_{n}'^*=V'^*|_{\mathcal{K}_n'}$.
		Since $\mathcal{K}_n$ is an invariant subspace for $V_{(n+1)}^*$, with respect to the decomposition $\mathcal{K}_{(n+1)}=\mathcal{K}_n\oplus \mathcal{D}_{T}$ the operator $V_{(n+1)}$ has the block matrix form $V_{(n+1)}=\begin{bmatrix}
			V_{n}&0\\
			S_n&0
		\end{bmatrix}$. Similarly, $\mathcal{K}_n'$ is an invariant subspace for $V_{(n+1)}'^*$ hence with respect to the decomposition $\mathcal{K}_{n+1}'=\mathcal{K}_n'\oplus \mathcal{D}_{T'}$ the operator $V'_{(n+1)}$ has the block matrix form $V'_{(n+1)}=\begin{bmatrix}
			V'_{n}&0\\
			S'_n&0
		\end{bmatrix}$. It can easily be observed that $V_{n}^*V_{n}+S_{n}^*S_{n}=I_{\mathcal{K}_n}$ and $V_{n}'^*V'_{n}+S_n'^*S_n'=I_{\mathcal{K}_n'}$. Now we find a sequence $\{Y_n\}$ inductively using Theorem $\ref{qpart}$ as follows. \\
		\textbf{Step I.} We have $V_{1}=\begin{bmatrix}
			T&0\\
			D_{T}&0
		\end{bmatrix}$ and $V'_{1}=\begin{bmatrix}
			T'&0\\
			D_{T'}&0
		\end{bmatrix}$ 
		satisfying $TX=XT'$, $T^*T+D_{T}^2=I_{\mathcal{H}}$ and $(T')^*(T')+D_{T'}^2=I_{\mathcal{H}}$. Therefore, by Theorem \ref{qpart}, there is a map $Y_1=\begin{bmatrix}
			X&0\\
			A_1&B_1
		\end{bmatrix}$ from $\mathcal{K}_1'$ to $\mathcal{K}_1$ such that $V_1Y_1=Y_1V_1'$, $Y_1^*|_{\mathcal{H}}=X^*$ and $\|Y_1\|=\|X\|$. \\
		\textbf{Step II.} Assume that for $1\leq m \leq n-1$ there is a map $Y_{m}=\begin{bmatrix}
			Y_{m-1}&0\\
			A_{m}&B_{m}
		\end{bmatrix}$ from $\mathcal{K}_{m}'$ to $\mathcal{K}_{m}$ such that $V_{m}Y_{m}=Y_mV_{m}'$, $Y_{m}^*|_{\mathcal{K}_{m-1}}=Y_{m-1}^*$ and $\|Y_{m}\|=\|X\|$. By Theorem \ref{qpart}, there is a map $Y_{n}=\begin{bmatrix}
			Y_{n-1}&0\\
			A_{n}&B_{n}
		\end{bmatrix}$ from $\mathcal{K}_{n}'$ to $\mathcal{K}_{n}$ such that $V_{n}Y_n=Y_{n}V_{n}'$, $Y_{n}^*|_{\mathcal{K}_{n-1}}=Y_{n-1}^*$ and $\|Y_{n}\|=\|X\|$.\\
		
		Hence by induction we have a sequence $\{Y_n\}$ such that for each $n\in\mathbb{N}$, $V_{n}Y_{n}=Y_{n}V_{n}'$, $Y_{n}^*|_{\mathcal{K}_{n-1}}=Y_{n-1}^*$ and $\|Y_{n}\|=\|X\|$.
		We may consider $Y_n^*$ as an operator from $\mathcal{K}$ to $\mathcal{K}'$ by defining $Y_n^*(k)=0$ for $k\in \mathcal{K}\ominus \mathcal{K}_n$. Similarly $V_n^*$ and $V'^*_n$ can be thought of as operators on $\mathcal{K}$ and $\mathcal{K}'$ by defining it to be $0$ on $\mathcal{K}_n^{\perp}$ and $\mathcal{K}_n'^{\perp}$ respectively. Hence for all $n\in \mathbb{N}$, $Y_n^*V_n^*=V_n'^*Y_n^*$ from $\mathcal{K}$ to $\mathcal{K}'$. For each $x \in \bigcup\limits_{n=0}^{\infty}\mathcal{K}_n$, which is dense in $\mathcal{K}$, the sequence $\{Y_n^*x\}$ is a cauchy sequence and since $\|Y_n^*\|=\|X^*\|$, by uniform boundedness principle the sequence $\{Y_n^*\}$ converges strongly, say to the operator $Y^*\text{ from }\mathcal{K}$ to $\mathcal{K}'$. Clearly $\|Y\|=\|X\|$ and $Y^*|_{\mathcal{H}}=X^*$. By construction $\{V_n^*\}$ and $\{V_n'^*\}$ converge strongly to $V^*$ and $V'^*$ respectively. Hence $VY=YV'$. This completes the proof.
	\end{proof}
	Since, $(V,\mathcal{K})$ is the minimal co-isometric extension of $T$ if and only if $V^*$ is the minimal isometric lift of $T^*$, we obtain the following result by restating the Theorem \ref{maininter}.
	\begin{thm}[Theorem \ref{douglasintertwining}]\label{mainintercoiso}
		Let $T,T',X\in \mathcal{B}(\mathcal{H})$ be such that $T$ and $T'$ are contraction and $TX=XT'$. Let $Z$ on $\mathcal{K}$ be the minimal co-isometric extension of $T$ and $Z'$ on $\mathcal{K}'$ be the minimal co-isometric extension of $T'$. Then there is a bounded linear transformation $Y:\mathcal{K}' \rightarrow \mathcal{K}$ such that 
		\begin{itemize}
			\item[(1)] $\mathcal{H}$ is invariant under $Y$ and $Y|_{\mathcal{H}}=X$;
			\item[(2)] $\|Y\|=\|X\|$ and
			\item[(3)] $ZY=YZ'$.
		\end{itemize}
	\end{thm}
	\begin{proof}
		Follows from Theorem \ref{maininter}.
	\end{proof}
	Now we obtain intertwining lift of $X$ to the minimal unitary dilation spaces of $T$ and $T'$ as follows.
	\begin{thm}\label{minunidil}
		Let $T$ and $T'$ be contractions and $X$ be any operator on a Hilbert space $\mathcal{H}$ such that $TX=XT'$. Let $U$ on $\mathcal{K}$ and $U'$ on $\mathcal{K}'$ be the minimal unitary dilations of $T$ and $T'$ respectively. Then there is a bounded linear operator $S:\mathcal{K}' \rightarrow \mathcal{K}$ such that 
		\begin{itemize}
			\item[(1)] $\|S\|=\|X\|$;
			\item[(2)] $US = SU'$ and
			\item[(3)] $T^nX = P_{\mathcal{H}}U^nS|_{\mathcal{H}}$ and  $XT'^n = P_{\mathcal{H}}SU'^n|_{\mathcal{H}}$ for all $n\geq 0$.
		\end{itemize}
	\end{thm}
	\begin{proof}
		Let $(U_+,\mathcal{K}_+)$ and $\left(U'_+,\mathcal{K}'_+\right)$ be the minimal isometric dilations of $(T,\mathcal{H})$ and $(T', \mathcal{H})$ respectively. Then by Theorem \ref{maininter}, there is $S_+:\mathcal{K}_+' \rightarrow \mathcal{K}_+$ such that 
		\begin{itemize}
			\item[(i)] $\mathcal{H}$ is invariant under $S_+^*$ and $S_+^*|_{\mathcal{H}}=X^*$;
			\item[(ii)] $\|S_+\| = \|X\|$ and 
			\item[(iii)] $U_+S_+ =S_+ U'_+$.
		\end{itemize}
		From (iii) we have that 
		\[
		S_+^*U_+^*=U_+'^*S_+^* .
		\]
		Since $(U^*, \mathcal{K})\text{ and }(U'^*, \mathcal{K}')$ are the minimal isometric dilations of $(U_+^*,\\ \mathcal{K}_+)\text{ and }(U_+'^*, \mathcal{K}_+')$ respectively, so again by Theorem \ref{maininter}, there is $S^*: \mathcal{K} \rightarrow \mathcal{K}'$ such that
		\begin{itemize}
			\item[(a)] $S \text{ maps }\mathcal{K}_+' \text{ to } \mathcal{K}_+ \text{ and }S|_{\mathcal{K}_+'} = S_+$;
			\item[(b)] $\|S\|=\|S_+\|=\|X\|$ and
			\item[(c)] $S^*U^*=U'^*S^* $.
		\end{itemize}
		From (c) we have that 
		\[
		US=SU'.
		\]
		By (i) the block matrix of $S_+^*$ with respect to the decompositions $\mathcal{K}_+= \mathcal{H} \oplus (\mathcal{K}_+\ominus \mathcal{H})$ and $\mathcal{K}_{+}'=\mathcal{H}\oplus (\mathcal{K}_+'\ominus \mathcal{H})$ is 
		\[
		\begin{bmatrix}
			X^* & *\\
			0 & *
		\end{bmatrix}, \text{ that is, } S_+=
		\begin{bmatrix}
			X & 0\\
			* & *
		\end{bmatrix}:\mathcal{H}\oplus (\mathcal{K}_+'\ominus \mathcal{H})\rightarrow \mathcal{H} \oplus (\mathcal{K}_+\ominus \mathcal{H}).
		\]
		Again by (a) the block matrix of $S$ with respect to the decompositions $\mathcal{K}=\mathcal{K}_+\oplus \mathcal{K}_+^\perp$ and $\mathcal{K}'= \mathcal{K}_+' \oplus \mathcal{K}_+'^\perp$ is 
		\[
		\begin{bmatrix}
			S_+ & *\\
			0 & *
		\end{bmatrix}.
		\]
		Therefore, the block matrix of $S$ with respect to the decompositions $\mathcal{K}=\mathcal{H}\oplus (\mathcal{K}_+\ominus \mathcal{H})\oplus \mathcal{K}_+^\perp$ and $\mathcal{K}'=\mathcal{H}\oplus (\mathcal{K}_+'\ominus \mathcal{H})\oplus \mathcal{K}_+'^\perp$ is 
		$\begin{bmatrix}
			X & 0 & *\\
			* & * & *\\
			0 & 0 & *
		\end{bmatrix}$.
		Again, since $(U, \mathcal{K})$ is the minimal unitary dilation of $(T, \mathcal{H})$, the block matrix of $U$ with respect to the decompositions $\mathcal{K}=\mathcal{H}\oplus (\mathcal{K}_+\ominus \mathcal{H})\oplus \mathcal{K}_+^\perp$ and $\mathcal{K}'=\mathcal{H}\oplus (\mathcal{K}_+'\ominus \mathcal{H})\oplus \mathcal{K}_+'^\perp$ is 
		$\begin{bmatrix}
			T & 0 & *\\
			* & * & *\\
			0 & 0 & *
		\end{bmatrix}$.
		Similarly, the block matrix of $U'$ with respect to the decompositions $\mathcal{K}=\mathcal{H}\oplus (\mathcal{K}_+\ominus \mathcal{H})\oplus \mathcal{K}_+^\perp$ and $\mathcal{K}'=\mathcal{H}\oplus (\mathcal{K}_+'\ominus \mathcal{H})\oplus \mathcal{K}_+'^\perp$ is 
		$\begin{bmatrix}
			T' & 0 & *\\
			* & * & *\\
			0 & 0 & *
		\end{bmatrix}$.
		Therefore,
		\[
		U^nS  = \begin{bmatrix}
			T^nX & 0 & *\\
			* & * & *\\
			0 & 0 & *
		\end{bmatrix}
		\] and consequently $T^nX = P_{\mathcal{H}}U^nS|_{\mathcal{H}}$ for all $n\geq 0$. Similarly, $XT'^n = P_{\mathcal{H}}SU'^n|_{\mathcal{H}}$ for all $n \geq 0$.
	\end{proof}
	\subsection{An approach due to Sebesty$\acute{e}$n}
	
	\begin{thm}[Theorem \ref{douglasintertwining}]
		Let $T,T',X\in \mathcal{B}(\mathcal{H})$ be such that $T$ and $T'$ are contraction and $TX=XT'$. Let $Z$ on $\mathcal{K}$ be the minimal co-isometric extension of $T$ and $Z'$ on $\mathcal{K}'$ be the minimal co-isometric extension of $T'$. Then there is a bounded linear operator $Y:\mathcal{K}' \rightarrow \mathcal{K}$ such that 
		\begin{itemize}
			\item[(1)] $\mathcal{H}$ is invariant under $Y$ and $Y|_{\mathcal{H}}=X$;
			\item[(2)] $\|Y\|=\|X\|$ and
			\item[(3)] $ZY=YZ'$.
		\end{itemize}
	\end{thm}
	\begin{proof}
		Since $(Z,\mathcal{K})$ and $(Z', \mathcal{K}')$ are the minimal co-isometric extensions of $T$ and $T'$ respectively, we have 
		\[
		\mathcal{K}=\bigvee_{n=0}^{\infty}Z^{*n}\mathcal{H} \;\text{  and  }\; \mathcal{K}'=\bigvee_{n=0}^{\infty}Z'^{*n}\mathcal{H}.
		\]
		Let $K_n = \bigvee\limits_{\ell = 0}^{n}Z^{*\ell}\mathcal{H}$ and $K_n' = \bigvee\limits_{m=0}^n Z'^{*m}\mathcal{H}$. Then $\bigcup\limits_{n = 0}^{\infty} K_n$ is dense in $\mathcal{K}$ and $\bigcup\limits_{n = 0}^{\infty}K_n'$ is dense in $\mathcal{K}'$. Consider the orthogonal projections $P_n : \mathcal{K}\rightarrow K_n$ and $P_n': \mathcal{K}' \rightarrow K_n'$ for all $n \geq 0$. Clearly 
		\begin{equation}\label{eq1'}
			P_{n+1}'Z'^* = Z'^*P_{n} \;\text{ and }\; P_{n+1}Z^* = Z^*P_n.
		\end{equation}
		We construct the required operator $Y:\mathcal{K}' \rightarrow \mathcal{K}$ inductively by finding operators $Y_n : K_n' \rightarrow K_n$ for all $n \in \mathbb N$.\\
		
		\noindent\textbf{Claim.} For any $h_1,h_2\in \mathcal{H}$, $\langle Z'^*X^*Z(Z^*h_1), h_2 \rangle = \langle (Z^*h_1), Xh_2 \rangle$.\\
		
		\noindent \textbf{Proof of Claim.} Since $Z$ and $Z'$ are the minimal co-isometric extensions of $T$ and $T'$ respectively, then for $h_1,h_2\in \mathcal{H}$ we have 
		\begingroup
		\allowdisplaybreaks
		\begin{align*}
			\langle Z'^*X^*Z(Z^*h_1), h_2 \rangle & = \langle Z'^*X^*h_1,h_2 \rangle \\
			& = \langle X^*h_1, Z'h_2 \rangle \\
			& = \langle X^*h_1, T'h_2 \rangle \\
			& = \langle h_1, XT'h_2 \rangle\\
			& = \langle h_1, TXh_2 \rangle \hspace{2.3cm} [\text{since } TX=XT']\\
			& = \langle h_1, ZXh_2 \rangle\\
			& = \langle Z^*h_1, Xh_2 \rangle.
		\end{align*}
		\endgroup
		This completes the proof of the claim.\\
		
		\noindent Therefore, by Theorem \ref{dualparrot}, there is an operator $Y_1\in \mathcal{B}(K_1',K_1)$ such that $Y_1|_{\mathcal{H}}=X$, $Y_1^*|_{Z^*\mathcal{H}}=Z'^*X^*Z|_{Z^*\mathcal{H}}$ and $\|Y_1\|\leq \text{max}\{\|X\|,\|Z'^*X^*Z|_{Z^*\mathcal{H}}\| \}=\|X\|$. Since $Y_1$ is an extension of $X$, $\|Y_1\|=\|X\|$.\\
		
		\noindent	Also we have 
		\begingroup
		\allowdisplaybreaks
		\begin{align}\label{eq2'S}
			Y_1^*P_1Z^* &= Y_1^*Z^*P_0 \hspace{7.05cm} [\text{from }\eqref{eq1'}]\notag\\
			& =Z'^*X^*ZZ^*P_0 \hspace{3.2cm}  \left[\text{since }Y_1^*|_{Z^*\mathcal{H}}=Z'^*X^*Z|_{Z^*\mathcal{H}}\right]\notag\\
			& = Z'^*X^*P_0.
		\end{align}
		\endgroup
		
		\noindent Assume the existence of $Y_{n-1}: K_{n-1}'\rightarrow K_{n-1}$ such that
		\begin{itemize}
			\item[(a)] $Y_{n-1}|_{K_{n-2}}=Y_{n-2}$,
			\item[(b)] $Y^*_{n-1}|_{Z^*K_{n-2}}=Z'^*Y_{n-2}^*Z|_{Z^*K_{n-2}}$,
			\item[(c)] $Y_{n-1}^*P_{n-1}Z^* = Z'^*Y_{n-2}^*P_{n-2}$,
			\item[(d)] $\|Y_{n-1}\|=\|Y_{n-2}\|$.
		\end{itemize}
		\text{ }\\
		\noindent\textbf{Claim.}	For any $h_1\in K_{n-1} \text{ and }h_2 \in K_{n-1}'$, $$\langle Z'^*Y_{n-1}^*Z(Z^*h_1), h_2 \rangle = \langle Z^*h_1, Y_{n-1}h_2 \rangle.$$
		\noindent \textbf{Proof of Claim.} Suppose $h_1\in K_{n-1} \text{ and }h_2 \in K_{n-1}'$. Then
		\begingroup
		\allowdisplaybreaks
		\begin{align*}
			\langle Z'^*Y_{n-1}^*Z(Z^*h_1), h_2 \rangle & = \langle Z'^*Y_{n-1}^*h_1,h_2 \rangle \\
			& = \langle P_{n-1}'Z'^*Y_{n-1}^*h_1, h_2 \rangle\hspace{2cm} [\text{since }Z'^*K_{n-1}'\subseteq K_{n}'] \\
			& = \langle Z'^*P_{n-2}'Y_{n-1}^*h_1, h_2 \rangle \hspace{2cm}[\text{using }\eqref{eq1'}]\\
			& = \langle h_1, Y_{n-1}P_{n-2}'Z'h_2 \rangle\\
			& = \langle P_{n-2}h_1, Y_{n-1}P_{n-2}'Z'h_2 \rangle\hspace{1.2cm}[\text{since }Y_{n-1}K_{n-2}'\subseteq K_{n-2}]\\
			& = \langle P_{n-2}h_1, Y_{n-2}P_{n-2}'Z'h_2 \rangle \hspace{1.2cm} [\text{using (a)}]\\
			& = \langle Y_{n-2}^*P_{n-2}'h_1, P_{n-2}Z'h_2 \rangle \\
			& = \langle Y_{n-2}^*P_{n-2}'h_1, Z'h_2 \rangle\hspace{2cm} [\text{since }Y_{n-2}^*K_{n-2}\subseteq K_{n-2}'] \\
			& = \langle Z'^*Y_{n-2}^*P_{n-2}h_1, h_2 \rangle\\
			& = \langle Y_{n-1}^*P_{n-1}Z^*h_1, h_2 \rangle \hspace{2cm} [\text{using (c)}]\\
			& = \langle Z^*h_1, Y_{n-1}h_2 \rangle.
		\end{align*}
		\endgroup
		The last inequality follows from the fact that $Y_{n-1}h_2\in K_{n-1}$, for any $h_2\in K_{n-1}'$. This completes the proof of claim.\\
		
		\noindent Hence by Theorem \ref{dualparrot}, there is a map $Y_n: K_n'\to K_n$ such that $Y_n|_{K_{n-1}'}=Y_{n-1} $,\\$Y_n^*|_{Z^*K_{n-1}}=Z'^*Y_{n-1}^*Z|_{Z^*K_{n-1}}$ and $\|Y_n\| \leq \text{max}\{\|Y_{n-1}\|,\|Z'^*Y_{n-1}^*Z|_{Z^*K_{n-1}}\|\}=\|Y_{n-1}\|$. Since $Y_n|_{K_{n-1}}=Y_{n-1}$, so $\|Y_n\|=\|Y_{n-1}\|=\|X\|$.
		Again 
		$$
		Y_n^*P_nZ^*  = Y_{n}^*Z^*P_{n-1} = Z'^*Y_{n-1}^*ZZ^*P_{n-1} = Z'^*Y_{n-1}^*P_{n-1}.
		$$
		The second last equality follows from (c). Therefore, by induction for all $n \in \mathbb N$, there is a map $Y_n:K_n' \rightarrow K_n$ such that 
		\begin{itemize}
			\item[(i)] $Y_{n}|_{K_{n-1}'}=Y_{n-1}$,
			\item[(ii)] $Y^*_{n}|_{Z^*K_{n-1}}=Z'^*Y_{n-1}^*Z|_{Z^*K_{n-1}}$,
			\item[(iii)] $Y_{n}^*P_{n}Z^* = Z'^*Y_{n-1}^*P_{n-1}$,
			\item[(iv)] $\|Y_{n}\|=\|Y_{n-1}\|=\|X\|$.
		\end{itemize} 
		Define $Y_0:\bigcup\limits_{n=0}^{\infty} \mathcal{K}_n' \rightarrow \bigcup\limits_{n=0}^{\infty} \mathcal{K}_n$ such that $Y_0|_{K_n'}=Y_n$. This is well defined as $Y_{n}|_{K_{n-1}'} = Y_{n-1}$. Since $\bigcup\limits_{n = 0}^{\infty}K_n'$ is dense in $\mathcal{K}'$, by continuity $Y_0$ extends to an operator $Y:\mathcal{K}'\to \mathcal{K}$.  From (iii) we have $ZY_nP_n'=Y_{n-1}P_{n-1}'Z'$ and one can easily check that $Y_nP_n'$ converges to $Y$ in the strong operator topology. Thus 
		\[
		ZY = YZ'.
		\]
		Clearly $\|Y\|=\|X\|$.
		This completes the proof of the theorem.
	\end{proof}
	The idea of the proof of above theorem is similar to that of Theorem 2 in \cite{Sebestyen}.\\
	\subsection{An approach due to Ando}
	Now we give an alternative proof of Theorem \ref{interdil}, using Ando's technique. Since $V$ is an isometric lift of $T$ if and only if $V^*$ is a co-isometric extension of $T^*$, we rewrite the Theorem \ref{interdil} as follows.
	\begin{thm}
		Let $T_1,T_2,X\in \mathcal{B}(\mathcal{H})$ be any contractions such that $T_1X=XT_2$. Then there are isometric lifts $V_1,V_2,V$ of $T_1,T_2,X$ respectively, on Hilbert space $\mathcal{K}$ containing $\mathcal{H}$ such that $V_{1}V=VV_2$.	
	\end{thm}
	\begin{proof}
		Define $\mathcal{K}=\oplus_{0}^{\infty} \mathcal{H}$, 
		\[W_{1}(h_0,h_1,\ldots ) = \left(T_1h_{0},D_{T_1}h_0,0, h_{1},\ldots \right),\]
		\[W_{2}(h_0,h_1,\ldots ) = \left(T_2h_{0},D_{T_2}h_0,0, h_{1},\ldots \right),\]
		\[W(h_0,h_1,\ldots ) = \left(Xh_{0},D_{X}h_0,0, h_{1},\ldots \right).\] Then $W_{1}, W_{2}, W$ all are isometries as $\| T_ih_0 \|^2+\| D_{T_i}h_0 \|^2=\|h_0\|^2$ for $i=1,2$ and \[ \| Xh_0 \|^2+\| D_{X}h_0 \|^2=\|h_0\|^2. \] Moreover, it is clear from the definition that $\mathcal{H}$ is invariant for $W_1^*,W_2^* ,W^*$ and for $i=1,2$, 
		\begin{equation}\label{invariance}
			W_i^*|_{\mathcal{H}}=T_i^*  \text{ and } W^*|_{\mathcal{H}}=X^* .
		\end{equation}  
		Now let $$\mathcal{B}=\mathcal{H}\oplus \mathcal{H}\oplus \mathcal{H}\oplus \mathcal{H} .$$
		Hence $\mathcal{K}$ can be identified with $\mathcal{H}\oplus \mathcal{B}\oplus \mathcal{B}\oplus \cdots$ via the identification, \[ (h_0,h_1,\ldots )=(h_0, (h_1,h_2,h_3,h_4),(h_5,h_6,h_7,h_8),\ldots ). \] 
		Let $$\mathcal{L}_1= \{(D_{X}T_2h,0,D_{T_2}h,0)\in\mathcal{B}: h\in \mathcal{H}\} $$ and $$\mathcal{L}_2= \{(D_{T_1}Xh,0,D_{X}h,0)\in\mathcal{B}: h\in \mathcal{H}\} .$$ Let $\mathcal{M}_i=\overline{\mathcal{L}}_i$ and $\mathcal{M}_i^{\perp}=\mathcal{B}\ominus\mathcal{M}_i$ for $i=1,2$. Define an operator $G:\mathcal{L}_1 \to \mathcal{L}_2$ by \[ G(D_{X}T_2h,0,D_{T_2}h,0)=(D_{T_1}Xh,0,D_{X}h,0) .\] A simple computation shows us that, \[ \|D_{X}T_2h \|^2+\| D_{T_2}h \|^2=\| h \|^2-\|XT_2h \|^2 \] and \[ \|D_{T_1}Xh \|^2+\| D_{X}h \|^2=\| h \|^2-\|T_1Xh \|^2  .\]  Since $T_1X=XT_2$, $G$ defines an isometry from $\mathcal{L}_1$ onto $\mathcal{L}_2$ and extends continuously as an isometry from $\mathcal{M}_1$ onto $\mathcal{M}_2$. It remains to show that $G$ can be extended to an isometry from the whole space $\mathcal{B}$ onto itself. For this purpose,  it suffices to prove that $\dim \mathcal{M}_1^{\perp}=\dim \mathcal{M}_2^{\perp}. $ This is clearly true when $\mathcal{H}$ and hence $\mathcal{B}$ are finite dimensional. Now  suppose $\mathcal{H}$ is infinite dimensional. Then we have, \[ \dim(\mathcal{H})=\dim\mathcal{B}\geq \dim\mathcal{M}_i^{\perp}\geq \dim \mathcal{H}. \] Here first inequality follows as $\mathcal{M}_i^{\perp}\subset \mathcal{B}$ and second inequality follows as \[\{ (0,h,0,0)\in \mathcal{B}:h\in \mathcal{H}\}\subset \mathcal{M}_i^{\perp}.\] Hence we can extend the map $G$ isometrically to a map from $\mathcal{B}$ onto $\mathcal{B}$. Therefore $G$ is a unitary. Now define $\widetilde{G}$ on $\mathcal{K}$ as \[ \widetilde{G}(h_0,h_1,\ldots )=(h_0,G(h_1,h_2,h_3,h_4),G(h_5,h_6,h_7,h_8),\ldots). \] Then clearly $\widetilde{G}$ is a unitary with inverse given by \[ \widetilde{G}^{-1}(h_0,h_1,\ldots )=(h_0,G^{-1}(h_1,h_2,h_3,h_4),G^{-1}(h_5,h_6,h_7,h_8),\ldots) .\] Further, let $V=\widetilde{G}W$, $V_1=W_1\widetilde{G}^{-1}$ and $V_2=W_{2}\widetilde{G}^{-1}$. Since $\mathcal{H}$ is reducing for $\widetilde{G}$ and both $\widetilde{G}$ and $\widetilde{G}^*=\widetilde{G}^{-1}$ are identity on $\mathcal{H}$, from \ref{invariance} we get, 
		\[ V_1^*|_{\mathcal{H}}=T_1^*, V_2^*|_{\mathcal{H}}=T_2^* \text{ and } V^*|_{\mathcal{H}}=X^*.\]  Now, 
		\begin{align*}
			VV_2(h_0,h_1,\ldots)&=\widetilde{G}WW_2\widetilde{G}^{-1}(h_0,h_1,\ldots)\\
			&=\widetilde{G}WW_2(h_0,G^{-1}(h_1,h_2,h_3,h_4),G^{-1}(h_5,h_6,h_7,h_8),\ldots)\\
			&=\widetilde{G}W(T_2h_0,D_{T_2}h_0,0,G^{-1}(h_1,h_2,h_3,h_4),G^{-1}(h_5,h_6,h_7,h_8),\ldots)\\
			&=\widetilde{G}(XT_2h_0,D_{X}T_2h_0,0,D_{T_2}h_0,0,G^{-1}(h_1,h_2,h_3,h_4),\ldots)\\
			&=(XT_2h_0,G(D_{X}T_2h_0,0,D_{T_2}h_0,0),(h_1,h_2,h_3,h_4),(h_5,h_6,h_7,h_8),\ldots).
		\end{align*} 
		Similarly, 
		\begin{align*}
			V_1V(h_0,h_1,\ldots )&=W_{1}\widetilde{G}^{-1}\widetilde{G}W(h_0,h_1,\ldots )\\
			&=W_{1}W(h_0,h_1,\ldots)\\
			&=W_{1}\left(Xh_{0},D_{X}h_0,0, h_{1},\ldots \right)\\
			&=(T_1Xh_0,D_{T_1}Xh_0,0,D_{X}h_0,0, h_{1},\ldots).
		\end{align*}
		Since $T_1X=XT_2$ and $G(D_{X}T_2h_0,0,D_{T_2}h_0,0)=(D_{T_1}Xh_0,0,D_{X}h_0,0)$, we get $V_1V=VV_2$ as required. 		
	\end{proof}
	
	\vspace{0.4cm}

	\section{System of $(G,Q)$-commuting contractions and applications}\label{QSectiongraph}
	
	\vspace{0.4cm}
	
	Op$\check{e}$la \cite{Ope} generalized both Ando's dilation for a pair of commuting contractions and Parrott's counter example of triple of commuting contractions that do not dilate simultaneously to commuting unitaries. Any $n$-tuple of contractions that commute according to a graph without a cycle can be dilated to an $n$-tuple of unitaries that commute according  to that graph. Conversely, if the graph contains a cycle, then Op$\check{e}$la has constructed a counterexample in \cite{Ope}. In \cite{K.M.} further generalized Op$\check{e}$la's notion of commuting with respect to a graph for $q$ commuting operators with $|q|=1$. They proved that if a $n$-tuple of strict contractions, $(T_1,\dots ,T_n)$ is $(G, q)$-commuting tuple with respect to acyclic simple graph $G$, then it possesses an isometric dilation $(V_1,\dots,V_n )$ which is $(G,q)$-commuting, where $|q|=1$.
	In this section, we give a sufficient condition for some $Q$-commuting $n$-tuple of contractions to possess $\overline{Q}$-commutant isometric lifts where $Q$, $\overline{Q}$ are some unitaries. First we recall some basic concepts from graph theory. 
	\begin{itemize}
		\item A simple graph $G=(V,E)$ consists of a vertex set $V$ and an edge set $E$, where an edge is an unordered pair of distinct vertices of $G$.
		\item A graph is called connected if for any two distinct vertices $u,v\in V$, there is a set of vertices, $\{u=u_1,u_2,\ldots u_{n-1},u_n=v\}$ in $V$ such that $\{u_i,u_{i+1}\}\in E$ for each $i=1,2,\ldots ,n-1$.
		\item For a given vertex $u$, the number of vertices connected to $u$ by an edge is called the degree of $u$.
		\item A connected graph in which every vertex has degree $2$ is called a cycle. 
		\item A graph $G'=(V',E')$ is called a subgraph of $G=(V,E)$ if $V'\subseteq V$ and $E' \subseteq E$.
		\item A graph is called acyclic if it does not contain any cycle as a subgraph.
		\item A connected acyclic graph is called Tree. 
		\item Every tree has atleast two vertices of degree 1.   
	\end{itemize} In this section, we generalize the Lemma 2.2 in \cite{Ope} and Theorem 3.12 in \cite{K.M.} and obtain a similar result for $Q$-commuting contractions when $Q$ is any unitary. Let $T_1,\dots, T_n$ be some contractions which are $Q_{ij}$-commutant with respect to certain unitaries $Q_{ij}$. Then we represent these relations via a simple graph by putting the contractions on the vertices and by adding an edge between the two vertices if and only if the corresponding contractions are $Q_{ij}$ commuting. We prove that if such a graph is acyclic then $T_1,\dots ,T_n$ posses isometric lifts $V_1,\dots ,V_n$ on $\mathcal{K}$ which satisfy similar $\overline{Q}_{ij}$-commuting relations where $\overline{Q}_{ij}=Q_{ij}\oplus I$. 
	Let $G=(V,E)$ be a graph with finite vertices. Let $V=\{1,2,\dots ,n\}$. Let $\widetilde{E}\subset V\times V$ be such that, $$\widetilde{E}=\{ (i,j):\{ i,j \}\in E,\,i<j \}.$$
	\begin{defn}[($(G,\textbf{Q},L)$-commuting)]
		Let $G=(V,E)$ be a graph with vertex set $V=\{ 1,\ldots ,n \}$ and let $\mathcal{H}$ be a Hilbert space. Let $\textbf{Q}$ be a function from $\widetilde{E}$ to $\mathcal{B}({\mathcal{H}})$ which takes any $(i,j)\in \widetilde{E} $ to a contraction $\textbf{Q}(i,j)$ on $\mathcal{H}$. A tuple $(T_1,\ldots ,T_n)$ of bounded linear operators on Hilbert space $\mathcal{H}$ is said to be $(G,\textbf{Q},L)$-commuting if $T_iT_j=\textbf{Q}(i,j)T_jT_i$ whenever $(i,j)\in \widetilde{E}$.  
	\end{defn} 
\begin{defn}[$(G,\textbf{Q},M)$-commuting]
	Let $G=(V,E)$ be a graph with vertex set $V=\{ 1,\ldots ,n \}$ and let $\mathcal{H}$ be a Hilbert space. Let $\textbf{Q}$ be a function from $\widetilde{E}$ to $\mathcal{B}({\mathcal{H}})$ which takes any $(i,j)\in \widetilde{E} $ to a unitary $\textbf{Q}(i,j)$ on $\mathcal{H}$. A tuple $(T_1,\ldots ,T_n)$ of bounded linear operators on Hilbert space $\mathcal{H}$ is said to be $(G,\textbf{Q},M)$-commuting if $T_iT_j=T_j\textbf{Q}(i,j)T_i$ whenever $(i,j)\in \widetilde{E}$.  
\end{defn} 
\begin{defn}[$(G,\textbf{Q},R)$-commuting]
	Let $G=(V,E)$ be a graph with vertex set $V=\{ 1,\ldots ,n \}$ and let $\mathcal{H}$ be a Hilbert space. Let $\textbf{Q}$ be a function from $\widetilde{E}$ to $\mathcal{B}({\mathcal{H}})$ which takes any $(i,j)\in \widetilde{E} $ to a unitary $\textbf{Q}(i,j)$ on $\mathcal{H}$. A tuple $(T_1,\ldots ,T_n)$ of bounded linear operators on Hilbert space $\mathcal{H}$ is said to be $(G,\textbf{Q},R)$-commuting if $T_iT_j=T_jT_i\textbf{Q}(i,j)$ whenever $(i,j)\in \widetilde{E}$.  
\end{defn} 

\begin{lem}\label{genQlemma}
	Let $Q,T_1,T_2\in \mathcal{B}(\mathcal{H})$ be any contractions such that $T_1T_2=QT_2T_1$. Let $\mathcal{K}'$ be any Hilbert space containing $\mathcal{H}$ and $W_2\in  \mathcal{B}(\mathcal{K}')$ be any isometric lift of $T_2$. Then there is a Hilbert space $\mathcal{K}$ containing $\mathcal{K}'$ and there are isometric lifts $\overline{Q},V_1,V_2\in \mathcal{B}(\mathcal{K})$ of $Q$, $T_1$, $W_2$ respectively, such that $V_1V_2=\overline{Q}V_2V_1$.	
	Moreover, $\mathcal{K}'$ is a reducing subspace for $V_2$. Further, in particular if $Q$ is a unitary then we can choose $\overline{Q}=Q\oplus qI$ on $\mathcal{H}\oplus (\mathcal{K}\ominus \mathcal{H})$ for any $q$ of modulus one.  
\end{lem}
\begin{proof}
	Clearly, if $Q$ is a unitary and $q$ is any complex number of modulus one then let $\overline{Q}_0=Q\oplus qI$ on $\mathcal{H}\oplus (\mathcal{K}'\ominus \mathcal{H}) $ where $\mathcal{K}'$ is the given Hilbert space. Hence $\overline{Q}_0$ and $W_2$ are isometric lifts of $Q$ and $T_2$ on $\mathcal{K}'$. Therefore, by Theorem \ref{QLcoro}, there is a Hilbert space $\mathcal{K}$ containing $\mathcal{K}'$ and there are isometric lifts $\overline{Q},V_1,V_2\in \mathcal{B}(\mathcal{K})$ of $Q$, $T_1$, $W_2$ respectively, such that $V_1V_2=\overline{Q}V_2V_1$ with $\overline{Q}=Q\oplus qI$ on $\mathcal{H}\oplus (\mathcal{K}\ominus \mathcal{H})$.	
	Moreover, $\mathcal{K}'$ is a reducing subspace for $V_2$.

	Now  in general suppose $Q$ is any contraction. Let $\overline{Q}_0$ on $\mathcal{K}_0$ be the minimal isometric dilation of $Q$. Let $\mathcal{K}_1=\mathcal{H}\oplus (\mathcal{K}'\ominus \mathcal{H})\oplus (\mathcal{K}_0\ominus \mathcal{H})$. Suppose $\overline{Q}_0=\begin{bmatrix}
		Q&0\\
		D&S
	\end{bmatrix}$ with respect to the decomposition $\mathcal{H}\oplus (\mathcal{K}_0\ominus \mathcal{H}) $. Now define $\overline{Q}_1=\begin{bmatrix}
		Q&0&0\\
		0&I&0\\
		D&0&S
	\end{bmatrix}$ on $\mathcal{K}_1$ with respect to the decomposition $\mathcal{H}\oplus (\mathcal{K}'\ominus \mathcal{H})\oplus (\mathcal{K}_0\ominus \mathcal{H})$ and $V_2'=W_2\oplus I$ on $\mathcal{K}_1$ with respect to the decomposition $\mathcal{K}_1=\mathcal{K}'\oplus (\mathcal{K}_0\ominus \mathcal{H})$. Then clearly $\overline{Q}_1$ is an isometric lift of $Q$ and $V_2'$ is an isometric lift of $W_2$ and hence in turn is an isometric lift of $T_2$. Thus by Theorem \ref{QLcoro}, there is a Hilbert space $\mathcal{K}$ containing $\mathcal{K}_1$, there are isometric lifts $V_1,V_2$ of $T_1,V_2'$ such that $V_1V_2=\overline{Q}V_2V_1$ where $\overline{Q}=\overline{Q}_1\oplus I$ with respect to the decomposition $\mathcal{K}=\mathcal{K}_1\oplus \mathcal{K}_1^{\perp}$. Moreover, $\mathcal{K}_1$ is reducing for $V_2$. Hence $V_2=V_2'\oplus X_2$ with respect to the decomposition $ \mathcal{K}=\mathcal{K}_1\oplus \mathcal{K}_1^{\perp}$ for some isometry $X_2$ on $\mathcal{K}_1^{\perp}$. Since $V_2'=W_2\oplus I$ on $\mathcal{K}_1$ with respect to the decomposition $\mathcal{K}_1=\mathcal{K}'\oplus (\mathcal{K}_0\ominus \mathcal{H})$, it follows that  $\mathcal{K}'$ is a reducing subspace for $V_2$.
\end{proof}

\begin{lem}\label{genQMlemma}
	Let $Q,T_1,T_2\in \mathcal{B}(\mathcal{H})$ be any contractions such that $T_1T_2=T_2QT_1$. Let $\mathcal{K}'$ be any Hilbert space containing $\mathcal{H}$ and $W_2\in  \mathcal{B}(\mathcal{K}')$ be any isometric lift of $T_2$. Then there is a Hilbert space $\mathcal{K}$ containing $\mathcal{K}'$ and there are isometric lifts $\overline{Q},V_1,V_2\in \mathcal{B}(\mathcal{K})$ of $Q$, $T_1$, $W_2$ respectively, such that $V_1V_2=V_2\overline{Q}V_1$.	
	Moreover, $\mathcal{K}'$ is a reducing subspace for $V_2$. Further, in particular if $Q$ is a unitary then we can choose $\overline{Q}=Q\oplus qI$ on $\mathcal{H}\oplus (\mathcal{K}\ominus \mathcal{H})$ for any $q$ of modulus one.  
\end{lem}
\begin{proof}
	Clearly, if $Q$ is a unitary and $q$ is any complex number of modulus one then let $\overline{Q}_0=Q\oplus qI$ on $\mathcal{H}\oplus (\mathcal{K}'\ominus \mathcal{H}) $ where $\mathcal{K}'$ is the given Hilbert space. Hence $\overline{Q}_0$ and $W_2$ are isometric lifts of $Q$ and $T_2$ on $\mathcal{K}'$. Therefore, by Theorem \ref{QMcoro}, there is a Hilbert space $\mathcal{K}$ containing $\mathcal{K}'$ and there are isometric lifts $\overline{Q},V_1,V_2\in \mathcal{B}(\mathcal{K})$ of $Q$, $T_1$, $W_2$ respectively, such that $V_1V_2=V_2\overline{Q}V_1$ with $\overline{Q}=Q\oplus qI$ on $\mathcal{H}\oplus (\mathcal{K}\ominus \mathcal{H})$.	
	Moreover, $\mathcal{K}'$ is a reducing subspace for $V_2$.

	Now  in general suppose $Q$ is any contraction. Let $\overline{Q}_0$ on $\mathcal{K}_0$ be the minimal isometric dilation of $Q$. Let $\mathcal{K}_1=\mathcal{H}\oplus (\mathcal{K}'\ominus \mathcal{H})\oplus (\mathcal{K}_0\ominus \mathcal{H})$. Suppose $\overline{Q}_0=\begin{bmatrix}
		Q&0\\
		D&S
	\end{bmatrix}$ with respect to the decomposition $\mathcal{H}\oplus (\mathcal{K}_0\ominus \mathcal{H}) $. Now define $\overline{Q}_1=\begin{bmatrix}
		Q&0&0\\
		0&I&0\\
		D&0&S
	\end{bmatrix}$ on $\mathcal{K}_1$ with respect to the decomposition $\mathcal{H}\oplus (\mathcal{K}'\ominus \mathcal{H})\oplus (\mathcal{K}_0\ominus \mathcal{H})$ and $V_2'=W_2\oplus I$ on $\mathcal{K}_1$ with respect to the decomposition $\mathcal{K}_1=\mathcal{K}'\oplus (\mathcal{K}_0\ominus \mathcal{H})$. Then clearly $\overline{Q}_1$ is an isometric lift of $Q$ and $V_2'$ is an isometric lift of $W_2$ and hence in turn is an isometric lift of $T_2$. Thus by Theorem \ref{QMcoro}, there is a Hilbert space $\mathcal{K}$ containing $\mathcal{K}_1$, there are isometric lifts $V_1,V_2$ of $T_1,V_2'$ such that $V_1V_2=V_2\overline{Q}V_1$ where $\overline{Q}=\overline{Q}_1\oplus I$ with respect to the decomposition $\mathcal{K}=\mathcal{K}_1\oplus \mathcal{K}_1^{\perp}$. Moreover, $\mathcal{K}_1$ is reducing for $V_2$. Hence $V_2=V_2'\oplus X_2$ with respect to the decomposition $ \mathcal{K}=\mathcal{K}_1\oplus \mathcal{K}_1^{\perp}$ for some isometry $X_2$ on $\mathcal{K}_1^{\perp}$. Since $V_2'=W_2\oplus I$ on $\mathcal{K}_1$ with respect to the decomposition $\mathcal{K}_1=\mathcal{K}'\oplus (\mathcal{K}_0\ominus \mathcal{H})$, it follows that  $\mathcal{K}'$ is a reducing subspace for $V_2$.
\end{proof}

\begin{lem}\label{genQRlemma}
	Let $Q,T_1,T_2\in \mathcal{B}(\mathcal{H})$ be any contractions such that $T_1T_2=T_2T_1Q$ and $\|T_2\|<1$. Let $\mathcal{K}'$ be any Hilbert space containing $\mathcal{H}$ and $W_1\in  \mathcal{B}(\mathcal{K}')$ be any isometric lift of $T_1$. Then there is a Hilbert space $\mathcal{K}$ containing $\mathcal{K}'$ and there are isometric lifts $\overline{Q},V_1,V_2\in \mathcal{B}(\mathcal{K})$ of $Q$, $W_1$, $T_2$ respectively, such that $V_1V_2=V_2V_1\overline{Q}$.	
	Moreover, $\mathcal{K}'$ is a reducing subspace for $V_1$. Further, in particular if $Q$ is a unitary then we can choose $\overline{Q}=Q\oplus qI$ on $\mathcal{H}\oplus (\mathcal{K}\ominus \mathcal{H})$ for any $q$ of modulus one.  
\end{lem}
\begin{proof}
 The proof is similar to that of the Theorem \ref{genQlemma} and is obtained as an application of the Theorem \ref{QRisopurecoro}.
\end{proof}
	Now we are ready to state and prove the main result.
	
	\begin{thm}\label{G,Q,L contraction}
		Let $G=(V,E)$ be a connected acyclic graph with vertex set $V=\{ 1,2,\ldots ,n  \}$. Let $\textbf{Q}$ be a function from $\widetilde{E}$ to $\mathcal{B}(\mathcal{H})$ which takes any $(i,j)\in \widetilde{E} $ to a contraction $\textbf{Q}(i,j)$. Let $T_1,\ldots , T_n$ be contractions on $\mathcal{H}$ such that $(T_1,\ldots , T_n)$ is $(G,\textbf{Q},L)$-commuting (  $(G,\textbf{Q},M)$-commuting). Then there is a Hilbert space $\mathcal{K}$ containing $\mathcal{H} $, there is a function $\overline{\textbf{Q}}$ from $\widetilde{E}$ to $\mathcal{B}(\mathcal{K})$ such that for each $(i,j)\in \widetilde{E}$, $\overline{\textbf{Q}}(i,j) $ is an isometric lift of $\textbf{Q}(i,j)$ and there are isometric lifts $V_1,\ldots ,V_n$ of $T_1,T_2,\ldots ,T_n$ respectively on $\mathcal{K}$ such that $(V_1,\ldots , V_n)$ is $(G,\overline{\textbf{Q}},L)$-commuting ($(G,\textbf{Q},M)$-commuting).
		
		Moreover, if for any $(i,j)\in \widetilde{E}$, $\textbf{Q}(i,j)$ is a unitary then we can choose $\overline{\textbf{Q}}(i,j)=\textbf{Q}(i,j)\oplus I$ on $\mathcal{K}=\mathcal{H}\oplus (\mathcal{K}\ominus \mathcal{H})$. 
	\end{thm}
	
	\begin{proof}
		We will prove the result by induction on $n$. We assume that $(T_1,\dots ,T_n)$ is $ (G,\textbf{Q},L)$-commuting. The result is true for $n=2$, by Theorem \ref{Andodil1}. Now let $n\geq 3$. Assume that the result holds whenever $|\textbf{V}|\leq n-1$.  We will hence prove the result for when $|\textbf{V}|=n$.  Since $G$ is a tree, without loss of generality, let $1$ and $n$ be vertices of degree one and let $e=\{1,2\}$ be the single edge containing vertex $1$. Let $G_1=(\textbf{V}_1, E_1)$ be a subgraph of $G$ such that $\textbf{V}_1=\textbf{V}\setminus \{1 \}$ and $E_1=E\setminus \{e\}$. Clearly $G_1$ is a  connected acyclic graph of $n-1$ vertices. Then due to the induction hypothesis, there is a Hilbert space $ \mathcal{K}'$ containing $\mathcal{H}$, a function $\textbf{Q}'$ from $ \widetilde{E}$ to $\mathcal{B}(\mathcal{K}')$ such that for each $(i,j)\in \widetilde{E_1}$,  $\textbf{Q}'(i,j)$ is an isometric lift of $\textbf{Q}(i,j)$ and there are isometric lifts $W_2,\ldots ,W_{n}\in \mathcal{B}(\mathcal{K}')$ of $T_2,\ldots ,T_{n}$ such that $(W_2,\ldots ,W_{n})$ is $(G_1,\textbf{Q}',L)$-commuting. Hence now we have $T_{1}T_{2}=\textbf{Q}(1,2)T_{2}T_{1}$ with $W_{2}$ as a given isometric lift of $T_{2}$. Therefore, by Lemma \ref{genQlemma}, there is a Hilbert space $\mathcal{K}$ containing $\mathcal{K}'$ and isometric lifts $\overline{Q}_{1,2},V_1,V_{2}$ of $\textbf{Q}(1,2),T_1,W_{2}$ such that 
		$V_1V_{2}=\overline{Q}_{1,2}V_{2}V_1$. 
	  Moreover, $\mathcal{K}'$ is a reducing subspace for $V_{2}$. Therefore, let $V_{2}=W_{2}\oplus B_{2}$ on $\mathcal{K}'\oplus (\mathcal{K}\ominus \mathcal{K}')$. Now let $V_i=W_i\oplus I$  on $\mathcal{K}'\oplus (\mathcal{K}\ominus \mathcal{K}')$ for all $i \in \textbf{V}_1\setminus \{ 2 \}$. Define a function $\overline{\textbf{Q}}$ from $\widetilde{E} $ to $\mathcal{B}(\mathcal{K})$ as follows. For $(i,j)\in \widetilde{E_1}$, define $\overline{\textbf{Q}}(i,j)=\textbf{Q}'(i,j)\oplus I$ with respect to the decomposition $\mathcal{K}'\oplus (\mathcal{K}\ominus \mathcal{K}')$ and let $\overline{\textbf{Q}}(1,2)=\overline{Q}_{1,2}$ on $\mathcal{K}$. 
		Observe that $\overline{\textbf{Q}}(i,j)=\textbf{Q}'(i,j)\oplus I$ with respect to decomposition $\mathcal{K}'\oplus (\mathcal{K}\ominus \mathcal{K}')$ for all $ (i,j)\in \widetilde{E_1}$. Then one can easily observe that for any $(i,j)\in \widetilde{E_1}$, $i, j\neq 2$, 
		\begin{align*} 
			V_iV_j=\begin{bmatrix}
				W_i&0\\0&I
			\end{bmatrix} \begin{bmatrix}
				W_j&0\\0&I
			\end{bmatrix}&=\begin{bmatrix}
				W_iW_j&0\\0&I
			\end{bmatrix}\\
			&= \begin{bmatrix}
				\textbf{Q}'(i,j)W_jW_i&0\\0&I
			\end{bmatrix}\\
			&= \begin{bmatrix}
				\textbf{Q}'(i,j)&0\\0&I
			\end{bmatrix}\begin{bmatrix}
				W_j&0\\0&I
			\end{bmatrix}\begin{bmatrix}
				W_i&0\\0&I
			\end{bmatrix}\\
			&=\overline{\textbf{Q}}(i,j)V_jV_i 
		\end{align*}  with respect to the decomposition $\mathcal{K}'\oplus (\mathcal{K}\ominus \mathcal{K}')$. For edges $(2,i)\in \widetilde{E_1}$,
		\begin{align*} 
			V_{2}V_i=\begin{bmatrix}
				W_{2}&0\\0&B_{2}
			\end{bmatrix} \begin{bmatrix}
				W_i&0\\0&I
			\end{bmatrix}&=\begin{bmatrix}
				W_{2}W_i&0\\0&B_{2}
			\end{bmatrix}\\
			&= \begin{bmatrix}
				\textbf{Q}'(2,i)W_iW_{2}&0\\0&B_{2}
			\end{bmatrix}\\
			&= \begin{bmatrix}
				\textbf{Q}'(2,i)&0\\0&I
			\end{bmatrix}\begin{bmatrix}
				W_i&0\\0&I
			\end{bmatrix}\begin{bmatrix}
				W_{2}&0\\0&B_{n-1}
			\end{bmatrix}\\
			&=\overline{\textbf{Q}}(2,i)V_iV_{2} 
		\end{align*} with respect to the decomposition $\mathcal{K}'\oplus (\mathcal{K}\ominus \mathcal{K}')$.
		Hence $(V_1,\ldots ,V_n)$ is $(G,\overline{\textbf{Q}},L)$-commuting and each $V_i$ is an isometric lift of $T_i$.

		We can similarly prove the result for $(G,\textbf{Q},M)$-commuting contractions by using Theorem \ref{AndodilM} and Lemma \ref{genQMlemma}.      
	\end{proof}	 
	\begin{thm}\label{G,Q,R contraction}
	Let $G=(V,E)$ be a connected acyclic graph with vertex set $V=\{ 1,2,\ldots ,n  \}$. Let $\textbf{Q}$ be a function from $\widetilde{E}$ to $\mathcal{B}(\mathcal{H})$ which takes any $(i,j)\in \widetilde{E} $ to a contraction $\textbf{Q}(i,j)$. Let $T_1,\ldots , T_n$ be strict contractions on $\mathcal{H}$ such that $(T_1,\ldots , T_n)$ is $(G,\textbf{Q},R)$-commuting. Then there is a Hilbert space $\mathcal{K}$ containing $\mathcal{H} $, there is a function $\overline{\textbf{Q}}$ from $\widetilde{E}$ to $\mathcal{B}(\mathcal{K})$ such that for each $(i,j)\in \widetilde{E}$, $\overline{\textbf{Q}}(i,j) $ is an isometric lift of $\textbf{Q}(i,j)$ and there are isometric lifts $V_1,\ldots ,V_n$ of $T_1,T_2,\ldots ,T_n$ respectively on $\mathcal{K}$ such that $(V_1,\ldots , V_n)$ is $(G,\overline{\textbf{Q}},R)$-commuting.
	
	Moreover, if for any $(i,j)\in \widetilde{E}$, $\textbf{Q}(i,j)$ is a unitary then we can choose $\overline{\textbf{Q}}(i,j)=\textbf{Q}(i,j)\oplus I$ on $\mathcal{K}=\mathcal{H}\oplus (\mathcal{K}\ominus \mathcal{H})$. 
\end{thm}
\begin{proof}
	Proof is similar to that of Theorem \ref{G,Q,L contraction} and follows as an application of Theorem \ref{Andodil2} and Lemma \ref{genQRlemma}.
\end{proof}

\end{document}